\documentclass[a4paper]{article}

\usepackage[utf8]{inputenc}
\usepackage[width=15cm,height=22cm]{geometry}
\usepackage{amsmath,amsfonts,amssymb,amsthm,mathrsfs}
\usepackage[hidelinks,pdfusetitle]{hyperref}
\usepackage[capitalise]{cleveref}
\usepackage[dvipsnames]{pstricks}
\usepackage{bbm}
\usepackage{graphicx}
\usepackage{stmaryrd} 
\usepackage{tikz}
\usepackage{enumitem}
\crefname{equation}{}{}

\newtheorem{proposition}{Proposition}[section]
\newtheorem{theorem}[proposition]{Theorem}

\newtheorem{definition}[proposition]{Definition}
\newtheorem{corollary}[proposition]{Corollary}
\newtheorem{lemma}[proposition]{Lemma}
\newtheorem{remark}[proposition]{Remark}
\numberwithin{equation}{section}

\newcommand{\dd}{\,\mathrm{d} }

\newcommand{\sgn}{\mathrm{sgn}}

\newif\ifarxiv
\arxivtrue      

\begin{document}

\begin{center}
    \uppercase{\textbf{Fractional quadratic obstructions to the local controllability of the Burgers equation}}

    \vspace{\baselineskip}

    \large{\textsc{Thomas Perrin\footnote{\textit{Université de Rennes, ENS Rennes, INRIA, CNRS, IRMAR - UMR 6625, F-35000 Rennes, France.}}}}
\end{center}

\begin{abstract}
    We study the local controllability near zero of the Burgers equation with a scalar control and a fixed space-dependent source profile, in the case where the linearized system fails to be controllable and a second-order analysis is therefore required. We prove that quadratic obstructions to finite-time controllability can be quantified by Sobolev norms of the control with fractional negative exponents ranging over a full interval. To our knowledge, this is the first example, for a natural physical PDE, of a continuous scale of fractional quadratic obstructions, and the first such continuous scale for finite-time local controllability. Our explicit constructions shed light on the origin of fractional obstructions for partial differential equations, by relating the obstruction exponent to the regularity, and in some cases to the physical-space singularity, of the source profile.
    
    We identify the natural structural conditions on the source profile leading to obstructions quantified by the $H^{-1}$ and $H^{-5/4}$ norms of the control, thereby providing a general framework in which the previously studied case of a constant source profile fits naturally. In this constant-profile case, we improve existing results by identifying the arithmetic condition on the Fourier mode which ensures that a small-time obstruction actually persists in finite time. 
    
    Finally, we derive sharp nonlinear remainder estimates adapted to the precise regularity of the source profile. These estimates make most of our obstruction results optimal with respect to the smallness assumption imposed on the control.
\end{abstract}

\begingroup
\renewcommand{\thefootnote}{}
\footnotetext{
    
    \hspace{0.1cm} \textit{MSC2020:} 93B05, 93C20, 35K58, 35K10, 35Q93
}
\endgroup

\setcounter{tocdepth}{1}
\tableofcontents

\hspace{2cm}

\section{Introduction}

\subsection{Accessible version of the main results}

In this article, we study the local controllability around zero of the viscous Burgers equation with scalar control
\begin{equation}\label{eq:Burgers_intro}
    \left\{
    \begin{array}{lll}
        \partial_t y - \partial_x^2 y + y \partial_x y = u(t) \mu(x)
        & \quad t \in (0, T), & \quad x \in (0,1), \\
        y(t,x) = 0
        & \quad t \in (0, T),  & \quad x \in \{0, 1\},\\
        y(0,x)=y_0
        & & \quad x \in (0, 1),
    \end{array}
    \right.
\end{equation}
where $T>0$, $\mu$ is a fixed function (which we sometimes call the \emph{source profile}), $u$ is the control, $y_0 \in L^2(0,1)$ is the initial state, and $y$ is the unknown. In this article, all functions are real-valued. Once a source profile $\mu$ is fixed, we sometimes write $t \mapsto y(t; y_0, u)$ for the solution of the Burgers equation \eqref{eq:Burgers_intro} with initial data $y_0$ and control $u$. 

Write $Ay := - \frac{\mathrm{d}^2}{\mathrm{d} x^2} y$ for the Dirichlet Laplacian on $(0,1)$, of domain $D(A) = H^2 \cap H_0^1 (0,1) \subset L^2((0,1); \mathbb{R})$, and $\langle \cdot, \cdot \rangle$ for the $L^2(0, 1)$-scalar product. We denote by $\lambda_j = ( j\pi )^2$ for $j \in \mathbb{N}^\ast$ the eigenvalues of $A$, and by $\varphi_j(x) = \sqrt{2} \sin(j \pi x)$ the corresponding $L^2(0,1)$-orthonormal eigenfunctions.

The purpose of this article is twofold. First, we prove general obstruction results for local controllability which put the case $\mu \equiv 1$, previously studied in \cite{Nguyen2025Burgers, Marbach2018}, into a broader and more systematic framework. In particular, our analysis identifies the precise structural conditions on the source profile $\mu$ which are responsible for the obstruction quantified by the $H^{-\frac{5}{4}}(0,T)$-norm of the control. More precisely, our main results imply the following theorem.

\begin{theorem}[The case $\mu \equiv 1$: accessible version]
    The obstruction to small-time local controllability quantified by the $H^{-\frac{5}{4}}(0,T)$-norm of the control, previously obtained in the case $\mu \equiv 1$ in \cite{Nguyen2025Burgers, Marbach2018}, is not specific to that case. It arises for any $\mu$ such that, for some $k \geq 1$, $\langle \mu, \varphi_k \rangle = \langle (\mu^2)^\prime, \varphi_k \rangle = 0$ and $\mu(0)^2 + (-1)^k \mu(1)^2 \neq 0$. In addition, when $\mu \equiv 1$, the obstruction to finite-time local controllability established in \cite{Nguyen2025Burgers} for $k = 2$ and $k = 10$ actually holds if $k$ is even and satisfies $k \in \left( \sqrt{2}, 2 \sqrt{2} \right) + 2 \sqrt{2} \mathbb{Z}$.
\end{theorem}

Second, we establish what appears to be the first obstruction result for a natural physical system in which the obstruction to local controllability is quantified by the $H^{r}(0,T)$-norm of the control, with $r$ varying over a full interval. Fractional obstructions cannot occur for ordinary differential equations, but they have been observed for certain partial differential equations, notably in \cite{Nguyen2025Burgers, Marbach2018}, mentioned above, in \cite{coronKoenigNguyen} for the Korteweg-de Vries equation, and in \cite{coron2024lack} for the water tank equation. In those works, a fractional exponent appears and plays a role in the proof of the obstruction, but the origin of its precise value generally remains unclear. To our knowledge, the only previous result exhibiting obstructions quantified by fractional norms across an interval is \cite{BeauchardMarbach2020}, where an academic model is specifically designed to produce this phenomenon. By contrast, the present work shows that such a fractional scale of obstructions also arises in a natural Burgers-type system, and sheds light on the mechanisms underlying this type of obstruction. Moreover, our results seem to be the first to provide obstructions to \textit{finite-time} local controllability quantified by fractional norms across an interval. An accessible version of our results is the following.

\begin{theorem}[Fractional obstructions: accessible version]
    For every $r \in \left(-\frac{3}{2}, -\frac{1}{2}\right)$, there exists $\mu \in H^{-1}(0,1)$ for which an obstruction to local controllability quantified by the $H^{r}(0,T)$-norm of the control occurs. Moreover, we provide two complementary families of explicit examples. First, we construct such profiles $\mu$ in the physical space, showing that fractional obstructions to small-time local controllability are related to singularities of the source profile in the physical domain. Second, we construct such profiles $\mu$, on the Fourier side, leading to fractional obstructions to finite-time local controllability.
\end{theorem}

\subsection{Definitions and well-posedness}

\subsubsection{Functional spaces}

For $n \in \mathbb{N}^\ast$, set $H_{(0)}^n(0,1) = D(A^{\frac{n}{2}})$, with $\Vert y \Vert_{H_{(0)}^n}^2 := \sum_{j \geq 1} \lambda_j^n \langle y, \varphi_j \rangle^2$. Recall that for $n \geq 0$ and $\mu \in H_{(0)}^{2n + 1}(0,1)$, one has $\mu^{(2m)}(0) = \mu^{(2m)}(1) = 0$ for all $m \in \llbracket 0, n \rrbracket$. 
Set also $D(A^\infty) = \bigcap_{n \geq 0} D(A^{\frac{n}{2}})$. For $s \in \mathbb{R}$, the space $H_{(0)}^s(0,1)$ is defined as the closure of $D(A^{\infty})$ for the norm
\begin{equation}
     \Vert \mu \Vert_{H_{(0)}^s}^2 := \sum_{j \geq 1} \lambda_j^s \langle \mu, \varphi_j \rangle^2.
\end{equation}
Smooth functions on $[0,1]$ with non-vanishing boundary traces belong to $H_{(0)}^s(0,1)$ for all $s < \frac{1}{2}$, but fail to belong to $H_{(0)}^{\frac{1}{2}}(0,1)$.
We thus introduce a family of functional spaces that precisely captures the maximal regularity of such functions.

\begin{definition}[Definition of $\mathcal{H}_p^{s}(0,1)$]\label{def:mathcal_Hs}
    Let $s \in \mathbb{R}$ and $p \in [1, +\infty]$. We define $\mathcal{H}_p^{s}(0,1)$ as the closure of $D(A^{\infty})$ for the norm
    \begin{equation}
        \left\Vert \mu \right\Vert_{\mathcal{H}_p^{s}} := \left\Vert \left( \lambda_j^{\frac{s}{2}} \langle \mu, \varphi_j \rangle \right)_{j \geq 1} \right\Vert_{\ell^p(\mathbb{N}^\ast)} .
    \end{equation}
\end{definition}

We will only work with $\mu \in \mathcal{H}_p^{s}(0,1)$ for $p = 2$ or $p = + \infty$. By definition, one has $\mathcal{H}_2^{s}(0,1) = H_{(0)}^s(0,1)$, and smooth functions belong to $\mathcal{H}_\infty^{1}(0,1)$. Note that 
\begin{equation}
    \mathcal{H}_2^{s}(0,1) \hookrightarrow \mathcal{H}_\infty^{s}(0,1) \quad \text{ and } \quad \mathcal{H}_\infty^{s}(0,1) \hookrightarrow \mathcal{H}_2^{s - \frac{1}{2} - \varepsilon}(0,1), 
\end{equation}
for all $s \in \mathbb{R}$ and $\varepsilon > 0$. For negative Sobolev regularity of controls, we use the following definition.

\begin{definition}[Definition of $\widetilde{H}^{s}(0,T)$]\label{def:widetilde_Hs}
    Let $s < \frac{1}{2}$. We write $\widetilde{H}^{s}(0,T)$ for the closure of $C^\infty_\mathrm{c}((0, T); \mathbb{R})$ for the norm
    \begin{equation}
        \Vert u \Vert_{\widetilde{H}^{s}(0,T)} := \left\Vert \mathbbm{1}_{(0,T)} u \right\Vert_{H^s(\mathbb{R})},
    \end{equation}
    where the $H^s(\mathbb{R})$-norm is defined using the Fourier transform.
\end{definition}

Our conventions for the Fourier transform and Sobolev spaces are
\begin{equation}
    \widehat{u}(\omega) := \int_{\mathbb{R}} u(t) e^{-i \omega t} \dd t, \quad \left\Vert u \right\Vert_{H^s(\mathbb{R})}^2 := \int_{\mathbb{R}} (1 + \omega^2)^s \left\vert \widehat{u}(\omega) \right\vert^2 \dd \omega,
\end{equation}
and we sometimes use the notation $\langle \omega \rangle := (1 + \omega^2)^{\frac{1}{2}}$. When the context is clear, we sometimes write $\Vert u \Vert_{\widetilde{H}^{s}}$ instead of $\Vert u \Vert_{\widetilde{H}^{s}(0,T)}$. Since some estimates naturally involve the control's primitive, we introduce the following definition. Heuristically, one has $\Vert u_1 \Vert_{\widetilde{H}^{s}(0,T)} \approx \Vert u \Vert_{\widetilde{H}^{s-1}(0,T)}$.

\begin{definition}[Primitive of the control]\label{def:primitive}
    Let $T > 0$ and $u \in L^2(0,T)$. For $t \in [0, T]$, we write $u_1(t) := \int_0^t u(\tau) \dd \tau$ for the primitive of $u$ vanishing at $t = 0$.
\end{definition}

\subsubsection{Well-posedness} 

\ifarxiv
    The controlled Burgers equation \eqref{eq:Burgers_intro} can be solved for $\mu \in H^{-1}(0,1)$ and $u \in L^2(0,T)$, as established by the following result, whose proof is provided in Appendix \ref{sec:lem:burgers_weak_wellposedness}.
\fi

\begin{lemma}[Well-posedness for the Burgers equation]\label{lem:burgers_weak_wellposedness}
    For $T > 0$, $f \in L^2((0, T), H^{-1}(0,1))$ and $y_0 \in L^2(0,1)$, there exists a unique solution 
    \begin{equation}
        y \in C^0([0, T]; L^2(0,1)) \ \cap \ L^2((0,T); H_0^1(0,1)) \ \cap \ H^1((0,T); H^{-1}(0,1))
    \end{equation}
    of the Burgers equation 
    \begin{equation}\label{eq:burgers_weak_wellposedness}
        \left\{
        \begin{array}{cll}
            \partial_t y - \partial_x^2 y + y \partial_x y = f
            & \quad t \in (0, T), & \quad x \in (0,1), \\
            y(t,0) = y(t,1) = 0
            & & \quad t \in (0, T), \\
            y(0,x)=y_0
            & & \quad x \in (0, 1).
        \end{array}
        \right.
    \end{equation}
    In addition, there exists an absolute constant $C > 0$ such that
    \begin{equation}\label{eq:lem:Burgers_weak_wellposedness}
        \begin{split}
            & \Vert y \Vert_{L^\infty((0,T); L^2)} + \Vert y \Vert_{L^2((0,T); H_0^1)} + \Vert y \Vert_{H^{1}((0,T); H^{-1})} \\
            \leq \ & C \left( \Vert y_0 \Vert_{L^2} + \Vert f \Vert_{L^2((0, T); H^{-1})} + \Vert y_0 \Vert_{L^2}^2 + \Vert f \Vert_{L^2((0, T); H^{-1})}^2 \right).
        \end{split}
    \end{equation}
\end{lemma}

\begin{remark}
    If $\mu$ is more regular than $H^{-1}(0,1)$, namely if $\mu \in \mathcal{H}_2^s(0,1) \cup \mathcal{H}_\infty^{s + \frac{1}{2}}(0,1)$ for some $s \in (-1, 1)$, then \eqref{eq:Burgers_intro} can be solved for weaker controls. For simplicity, however, we restrict ourselves in this article to controls $u \in L^2(0,T)$. Note that this is only a regularity assumption; when studying local controllability, the smallness of $u$ will be assumed in an appropriate regularity space that depends on the regularity of $\mu$ (see Definitions \ref{def:intro_STLC_FTLC} below).
\end{remark}

\subsubsection{Power series expansion}

Assume that $y_0 = 0$. We consider the power series expansion $y = y_1 + y_2 + \cdots$, where $y_1$ and $y_2$ solve, respectively, the linearized and quadratic control problem 
\begin{equation}\label{eq:Burgers_y_1_intro}
    \left\{
    \begin{array}{lll}
        \partial_t y_1 - \partial_x^2 y_1 = u(t) \mu
        & \quad t \in (0, T), & \quad x \in (0, 1), \\
        y_1(t,x) = 0
        & \quad t \in (0, T),  & \quad x \in \{0, 1\},\\
        y_1(0,x)=0
        & & \quad x \in (0, 1),
    \end{array}
    \right.
\end{equation}
and
\begin{equation}\label{eq:Burgers_y_2_intro}
    \left\{
    \begin{array}{lll}
        \partial_t y_2 - \partial_x^2 y_2 + y_1 \partial_x y_1 = 0
        & \quad t \in (0, T), & \quad x \in (0, 1), \\
        y_2(t,x) = 0
        & \quad t \in (0, T), & \quad x \in \{0, 1\},\\
        y_2(0,x)=0
        & & \quad x \in (0, 1).
    \end{array}
    \right.
\end{equation}
Once a source profile $\mu$ is fixed, we sometimes write $y_1(t; u)$ and $y_2(t; u)$ for the solutions of \eqref{eq:Burgers_y_1_intro} and \eqref{eq:Burgers_y_2_intro} with control $u$. Note that in this article, we do not consider power series expansions of $u \mapsto y(\cdot ; y_0, u)$ for $y_0 \neq 0$, and that the notations $y_1$ and $y_2$ will always refer to functions that vanish at $t=0$.

Note that by definition, $y_1$ and $y_2$ are given by 
\begin{equation}\label{eq:y_1_Fourier_intro}
    y_1(t,x) = \sum_{j \geq 1} \left( \int_0^t e^{-\lambda_j (t-\tau)} u(\tau) \dd \tau \right) \left\langle \mu, \varphi_j \right\rangle \varphi_j(x)
\end{equation}
and
\begin{equation}\label{eq:y_2_Fourier_intro}
    y_2(t,x) = \frac{1}{2} \sum_{j \geq 1} \left( \int_0^t e^{-\lambda_j (t-\tau)} \left\langle (y_1(\tau))^2, \varphi_j^\prime \right\rangle \dd \tau \right) \varphi_j(x).
\end{equation}

\subsection{Detailed results}

\subsubsection{Definition of controllability and obstructions}

We use the following definitions for local controllability.

\begin{definition}\label{def:intro_STLC_FTLC}
    Let $\sigma < \frac{1}{2}$. For $T > 0$, we say that \eqref{eq:Burgers_intro} is \emph{locally null-controllable at time $T > 0$}, with controls small in $\widetilde{H}^{\sigma}$, if the following property holds: for any $\eta > 0$, there exists $\delta > 0$ such that, for any $y_0 \in L^2(0,1)$ satisfying $\Vert y_0 \Vert_{L^2} \leq \delta$, there exists a control $u \in L^2(0,T) \cap \widetilde{H}^\sigma(0,T)$ satisfying $\Vert u \Vert_{\widetilde{H}^{\sigma}} \leq \eta$, such that the corresponding solution satisfies $y(T; y_0, u) = 0$.
    
    We say that \eqref{eq:Burgers_intro} is \emph{small-time locally null-controllable} (in short, \emph{STLC}), with controls small in $\widetilde{H}^{\sigma}$, if for all sufficiently small $T>0$, \eqref{eq:Burgers_intro} is locally null-controllable at time $T > 0$ with controls small in $\widetilde{H}^{\sigma}$. We say that \eqref{eq:Burgers_intro} is \emph{finite-time locally null-controllable} (in short, \emph{FTLC}), with controls small in $\widetilde{H}^{\sigma}$, if there exists $T>0$ such that \eqref{eq:Burgers_intro} is locally null-controllable at time $T > 0$ with controls small in $\widetilde{H}^{\sigma}$.
\end{definition}


We shall also use the following terminology. Let $r \in \left[-\frac{3}{2},-\frac{1}{2}\right]$ and $\sigma < \frac{1}{2}$, with $r \leq \sigma$. We say that there is an \emph{obstruction to STLC with controls small in $\widetilde{H}^{\sigma}$, quantified by the $\widetilde{H}^{r}$-norm of the control}, if STLC with controls small in $\widetilde{H}^{\sigma}$ is prevented by a systematic small-time drift estimate involving the $\widetilde{H}^{r}(0,T)$-norm of the control. More precisely, this means an estimate of the form
\begin{equation}
    \left\langle y(T;y_0,u), \varphi_k \right\rangle = \left\langle y(T; y_0, 0), \varphi_k \right\rangle + \alpha_k \left\Vert u \right\Vert_{\widetilde{H}^{r}(0,T)}^2 + \text{ remainder terms},
\end{equation}
for some $\alpha_k \neq 0$, whenever $T$, $\left\Vert u \right\Vert_{\widetilde{H}^{r}(0,T)}$, and $y_0$ are sufficiently small. Here $\varphi_k$ is a direction lost at the linear level, in the sense that $\langle y_1(T), \varphi_k \rangle = \langle \mu, \varphi_k \rangle = 0$. If an estimate of this type holds for some fixed $T>0$, we say that there is an \emph{obstruction to FTLC quantified by the $\widetilde{H}^{r}$-norm of the control}. Precise statements are given in Theorems \ref{thm:main:kernel_implies_drift_small_time} and \ref{thm:main:kernel_implies_drift_finite_time}. We prove that drift estimates are indeed obstructions to controllability in Corollaries \ref{cor:obstruction_small_time} and \ref{cor:obstruction_finite_time}.

For $s \in [-1, 1]$, we define $\sigma(s)$ by
\begin{equation}\label{eq:defn_sigma(s)}
\sigma(s) := - \frac{1+s}{2} \text{ if } s \in \left[-1, \frac{1}{2} \right], \quad \text{ and } \quad \sigma(s) := -1 + \frac{s}{2} \text{ if } s \in \left[ \frac{1}{2}, 1 \right].
\end{equation}

To help the reader keep track of the exponents appearing in the statements of our results, we gather here the natural correspondence between the regularity of $\mu$, the norm of the control involved in the drift, and the norm of the control involved in the smallness assumption. We expect these exponents to be optimal, as explained in Remarks \ref{rq:optimal_y_1} and \ref{rq:optimal_y_2}. They are as follows. 

If $\mu \in \mathcal{H}^s_2(0,1)$ or $\mu \in \mathcal{H}^{s+\frac{1}{2}}_\infty(0,1)$ for some $s \in [-1,1]$, we aim to prove obstructions to controllability quantified by the $\widetilde{H}^{-1-\frac{s}{2}}(0,T)$-norm of the control, for controls small in $\widetilde{H}^{\sigma(s)}(0,T)$.
Equivalently, writing $r = -1-\frac{s}{2}$, that is, $s = -2-2r$, then $r \in \left[ - \frac{3}{2}, - \frac{1}{2} \right]$, and, for $\mu \in \mathcal{H}^{-2-2r}_2(0,1)$ or $\mu \in \mathcal{H}^{-\frac{3}{2} - 2r}_\infty(0,1)$, we aim to prove an obstruction to controllability quantified by the $\widetilde{H}^{r}(0,T)$-norm of the control, for controls small in $\widetilde{H}^{\sigma(-2-2r)}(0,T)$.

\subsubsection{The linearized system}

Let $\mu \in H^{-1}(0,1)$, and set $\mu_n := \langle \mu, \varphi_n \rangle$. On the one hand, if $\mu_k = 0$ for some $k \geq 1$, then $\langle y_1(t; u), \varphi_k \rangle = 0$ for every control $u$. We say that the direction $\varphi_k$ is \textit{lost at the linear level}. In this case, one has
\begin{equation}
    \langle y(t; 0, u), \varphi_k \rangle = \langle y_2(t; u), \varphi_k \rangle + \text{ remainder terms,}
\end{equation}
so that the behavior of the controlled system in the direction $\varphi_k$ is governed by the quadratic expansion. This is precisely the situation considered in this article.

On the other hand, if $\mu_n \neq 0$ for every $n \geq 1$, and if $(\mu_n)_n$ does not decay too fast, for instance if $\sup_{n \geq 1} \frac{-\ln \vert \mu_n \vert}{n} < + \infty$, then one can prove that STLC holds with controls small in $L^2(0,T)$. Indeed, the moment method gives null controllability of the scalar controlled heat equation with cost $e^{C/T}$. Moreover, by arguments similar to those used in the proofs of Lemma \ref{lem:burgers_weak_wellposedness} and Lemma \ref{lem:decoupling_estimate}, one can show that the difference between the nonlinear Burgers flow $u \mapsto y(\cdot; y_0, u)$ and its linearization is quadratic in $y_0$ and $u$. Applying the Liu-Takahashi-Tucsnak time-iteration method \cite{LiuTakahashiTucsnak2013}, in the black-box form given in \cite{MarbachTimeIteration}, then yields the result. We omit the details and refer to \cite{MarbachTimeIteration}.

\subsubsection{Quadratic kernel}

Most of the results of this article are based on Proposition \ref{prop:decomposition_y_2_kernel}, which establishes that for all $\mu \in H^{-1}(0,1)$ and $k \geq 1$, there exists a continuous function $\omega \mapsto K_{\mu,k}(\omega)$, which we call the \emph{quadratic kernel}, such that for all $T > 0$ and all $u \in L^2(0,T)$, 
\begin{equation}\label{eq:kernel_intro}
\left\langle y_2(T,u), \varphi_k \right\rangle = \frac{k e^{-\lambda_k T}}{4 \pi^2 \sqrt{2}} \int_{\mathbb{R}} \left\vert \widehat{u \rho_k}(\omega) \right\vert^2 K_{\mu,k}(\omega) \dd \omega + \text{remainder terms,}
\end{equation}
where $\rho_k : t \mapsto e^{\frac{\lambda_k}{2}t}$, and where $u$ is extended by zero outside $[0,T]$. An analogous formula was obtained in \cite{Nguyen2025Burgers} in the case $\mu \equiv 1$, under the additional assumption that $y_1(T)=0$. Under this assumption, the remainder terms in \eqref{eq:kernel_intro} vanish. Although this gives a cleaner quadratic expansion, it also forces one to use linear corrections in order to apply the formula to prove controllability obstructions.

Compared with \cite{Nguyen2025Burgers}, where the authors establish obstructions to controllability without deriving a systematic drift estimate, our approach makes the link between kernel properties and drift estimates transparent. More precisely, in Theorem \ref{thm:main:kernel_implies_drift_small_time}, we show that a high-frequency asymptotic formula for $K_{\mu,k}$ of the form $\frac{1}{\omega^{2+s}}$ yields an obstruction to STLC quantified by the $\widetilde{H}^{-1-\frac{s}{2}}$-norm of the control, and we prove in Theorem \ref{thm:main:kernel_implies_drift_finite_time} that this obstruction is in fact an obstruction to FTLC if, in addition, $K_{\mu,k}$ satisfies a global sign property.

With these results in hand, it remains only to study the connection between the source profile $\mu$ and the properties of $K_{\mu,k}$. We prove the following results.

\begin{enumerate}[label=(\roman*)]
    \item Asymptotic equivalents of the form $\frac{1}{\omega^{2}}$, which are typical of $H^{-1}$-drifts in small time, arise as soon as $\mu \in L^2(0,1)$ satisfies $\langle \mu^2, \varphi_k^\prime \rangle \neq 0$, by Proposition \ref{prop:asymptotic_estimate_kernel_L2}.
    
    \item Asymptotic equivalents of the form $\frac{1}{\omega^{{5}/{2}}}$, which are typical of $H^{-\frac{5}{4}}$-drifts in small time, arise when $\mu$ has a well-defined non-vanishing trace at the boundary and the coefficient of the $H^{-1}$-drift vanishes, that is, when $\langle \mu^2, \varphi_k^\prime \rangle = 0$. A precise statement is given in Proposition \ref{prop:asymptotic_estimate_kernel_affine+rest}, which in particular includes the case $\mu \equiv 1$.
    
    \item For $s$ varying over suitable intervals, there exists a source profile $\mu$ for which the quadratic kernel satisfies an asymptotic equivalent of the form $\frac{1}{\omega^{2+s}}$, which is typical of a $\widetilde{H}^{-1-\frac{s}{2}}$-drift in small time. In addition, such source profiles can be constructed explicitly in the physical space, thereby shedding light on the connection between the regularity of $\mu$ and fractional drifts, as established in Proposition \ref{prop:asymptotic_estimates_fractional_drifts}.
    
    \item For any $s \in [-1, 1]$, there exists a source profile $\mu$ for which the quadratic kernel satisfies both an asymptotic equivalent of the form $\frac{1}{\omega^{2+s}}$ and a global sign property, as established in Proposition \ref{prop:global_sign_AND_asymptotic_estimates}. These two properties are used to establish a $\widetilde{H}^{-1-\frac{s}{2}}$-drift estimate that holds for all $T>0$, yielding an obstruction to FTLC. The corresponding source profiles are constructed by prescribing their Fourier coefficients.

    \item When $\mu \equiv 1$, the asymptotic equivalent of (ii) can be strengthened into a global sign property if and only if $k$ is even and belongs to $\left( \sqrt{2}, 2 \sqrt{2} \right) + 2 \sqrt{2} \mathbb{Z}$, yielding a $H^{-\frac{5}{4}}$-drift estimate that holds for all $T>0$, and thus the corresponding obstruction to FTLC, as established in Proposition \ref{prop:mu_constant_sign_global}.
\end{enumerate}

These results, combined with Theorems \ref{thm:main:kernel_implies_drift_small_time} and \ref{thm:main:kernel_implies_drift_finite_time}, yield obstructions to local controllability, as detailed in the next sections.

\subsubsection{Obstructions to STLC}

We now state our main obstruction results for STLC. They are obtained by combining the asymptotic properties of the quadratic kernel established in Propositions \ref{prop:asymptotic_estimate_kernel_L2}, \ref{prop:asymptotic_estimate_kernel_affine+rest}, and \ref{prop:asymptotic_estimates_fractional_drifts}, with the small-time kernel-to-drift principle of Theorem \ref{thm:main:kernel_implies_drift_small_time}.

\begin{theorem}[Obstructions to STLC]\label{thm:main_obstruction_STLC}
    We have the following results.
    \begin{enumerate}[label=(\roman*)]
        \item \emph{(Obstruction quantified by the $\widetilde{H}^{-1}$-norm of the control.)} 
            For all sufficiently small $\nu > 0$, $\mu \in \mathcal{H}_2^\nu(0,1)$ and $k \geq 1$, such that $\langle \mu, \varphi_k \rangle = 0$ and $\langle \mu^2, \varphi_k^\prime \rangle \neq 0$, there is an obstruction to STLC with controls small in $\widetilde{H}^{-\frac{1}{2} + \nu}$, quantified by the $\widetilde{H}^{-1}$-norm of the control.
        
        \item \emph{(Obstruction quantified by the $\widetilde{H}^{-\frac{5}{4}}$-norm of the control.)} 
            Assume that $\mu := \ell + \widetilde{\mu}$, with $\ell : [0, 1] \rightarrow \mathbb{R}$ an affine function, and $\widetilde{\mu} \in H_{(0)}^s(0,1)$, for some $s \in \left(\frac{1}{2}, 2\right]$. Then $\mu$ has a well-defined trace at the boundary. Let $k \geq 1$. Assume that $\langle \mu, \varphi_k \rangle = \langle \mu^2, \varphi_k^\prime \rangle = 0$ and $\mu(0)^2 + (-1)^k \mu(1)^2 \neq 0$. Then, for all $\varepsilon > 0$ sufficiently small, there is an obstruction to STLC with controls small in $\widetilde{H}^{-\frac{3}{4}+\varepsilon}$, quantified by the $\widetilde{H}^{-\frac{5}{4}}$-norm of the control. In addition, if $\widetilde{\mu} \in \mathcal{H}^1_\infty(0,1)$, or simply if $\widetilde{\mu}=0$, then the $\varepsilon$-loss can be removed: the same obstruction holds for controls small in $\widetilde{H}^{-\frac{3}{4}}$.
            
        \item \emph{(Obstruction quantified by a continuous range of fractional norms of the control.)} 
            For $r \in \left( - \frac{3}{2}, - \frac{1}{2} \right)$, with $r \notin \left\{- 1, -\frac{3}{4} \right\}$, there exists $\mu \in \mathcal{H}^{- \frac{3}{2} - 2 r}_\infty(0,1)$, which is constructed explicitly in the physical space and involves a singularity of the form $x \mapsto (x - x_0)^{- \frac{5}{2} - 2 r}$, for some $x_0 \in (0,1)$, such that there is an obstruction to STLC with controls small in $\widetilde{H}^{\sigma(-2-2r)}$, quantified by the $\widetilde{H}^{r}$-norm of the control.
    \end{enumerate}
\end{theorem}

We provide some comments on Theorem \ref{thm:main_obstruction_STLC}. First, the assumption $\mu \in \mathcal{H}^\nu(0,1)$ of Theorem \ref{thm:main_obstruction_STLC} \emph{(i)} is not needed to prove an asymptotic equivalent of the form $\frac{1}{\omega^2}$ for $K_{\mu,k}$; the natural assumption for this is $\mu \in L^2(0,1)$. However, assuming $\mu \in \mathcal{H}^\nu(0,1)$ yields a sharper remainder term, which allows us to derive the drift estimate directly from our general result Theorem \ref{thm:main:kernel_implies_drift_small_time}. We nevertheless believe that, under the sole assumption $\mu \in L^2(0,1)$, one can still prove an obstruction to STLC with controls small in $\widetilde{H}^{-\frac{1}{2}}$, quantified, in a slightly different way, by the $\widetilde{H}^{-1}$-norm of the control.

Second, we choose $\mu$ of the form $\mu := \ell + \widetilde{\mu}$ in Theorem \ref{thm:main_obstruction_STLC} \emph{(ii)} in order to ensure that $\mu$ has a well-defined trace at the boundary, while still including the case $\mu \equiv 1$. Further comments on the case $\mu \equiv 1$ are given after Theorem \ref{thm:main_obstruction_mu_constant} below. Note that the condition $\widetilde{\mu} \in \mathcal{H}^{1}_\infty(0,1)$ is satisfied, for instance, if $\mu = \ell + \widetilde{\mu} \in C^1([0,1],\mathbb{R})$.
 
Finally, Theorem \ref{thm:main_obstruction_STLC} \emph{(iii)} appears to be the first example of a natural PDE for which obstructions quantified by an exponent varying over an interval occur. The closest existing result seems to be \cite{BeauchardMarbach2020}, where an academic model is constructed to exhibit this phenomenon, and where the corresponding source profiles are built by prescribing their Fourier coefficients. In contrast, we show here that a singularity in the physical space leads to such an obstruction, with an exact correspondence between the singularity exponent and the obstruction exponent. Note that when $r > -\frac{3}{4}$, the profile $x \mapsto (x - x_0)^{- \frac{5}{2} - 2 r}$ is understood as a distribution (see Proposition \ref{prop:asymptotic_estimates_fractional_drifts}), and that the case $r =  -\frac{3}{4}$ is excluded mostly for simplicity (see Remark \ref{rem:after:prop:asymptotic_estimates_fractional_drifts}).




\subsubsection{Obstructions to FTLC} 

We now state our main obstruction results for FTLC. They are obtained by combining the asymptotic and global sign properties of the quadratic kernel established in Proposition \ref{prop:global_sign_AND_asymptotic_estimates}, with the finite-time kernel-to-drift principle of Theorem \ref{thm:main:kernel_implies_drift_finite_time}.

\begin{theorem}[Obstructions to FTLC quantified by a continuous range of fractional norms]\label{thm:main_obstruction_FTLC}
    For $r \in \left( - \frac{3}{2}, - \frac{1}{2} \right)$, there exist $\mu \in H^{-1}(0,1)$ and $\theta < \frac{1}{2}$ such that there is an obstruction to FTLC with controls small in $\widetilde{H}^{\theta}$, quantified by the $\widetilde{H}^{r}$-norm of the control. More precisely, the following additional properties hold:
    \begin{enumerate}[label=(\roman*)]
        \item if $r \in \left( -1, - \frac{1}{2} \right)$, then one can choose $\mu \in \mathcal{H}_\infty^{-\frac{3}{2} - 2r}(0,1)$ and $\theta = \frac{1}{2} + r$,
        \item if $r = -1$, then one can choose $\mu \in D(A^\infty)$, and $\theta = - \frac{1}{2}$,
        \item if $r \in \left( - \frac{3}{2}, -1 \right)$, then one can choose $\mu \in \mathcal{H}_\infty^{-2 - 2r}(0,1)$ and $\theta = \sigma\left(-\frac{5}{2} - 2r\right) + \frac{1}{2}$.
    \end{enumerate}
\end{theorem}

Theorem \ref{thm:main_obstruction_FTLC} seems to be the first example of a PDE for which obstructions \textit{in finite time}, quantified by an exponent varying over an interval, occur. In contrast with Theorem \ref{thm:main_obstruction_STLC} \emph{(iii)}, the proof of Theorem \ref{thm:main_obstruction_FTLC} relies on a construction of $\mu$ by prescribing its Fourier coefficients, since ensuring both an asymptotic property and a global sign property is not straightforward. Moreover, to obtain an obstruction quantified by the $\widetilde{H}^{r}$-norm of the control when $r < -1$, one has to ensure that the $\widetilde{H}^{-1}$-drift appearing in Theorem \ref{thm:main_obstruction_STLC} \emph{(i)} vanishes. Note that Theorem \ref{thm:main_obstruction_FTLC} \emph{(iii)} is a case where the kernel $K_{\mu,k}$ decays faster than what one would expect from the regularity of $\mu$ (see Remark \ref{rq:kernel_decays_faster_than_expected}). This leads to an obstruction result whose smallness assumption is stronger than what the size of the drift would suggest.


\subsubsection{The case of a constant source profile} 

Our main result for the case $\mu \equiv 1$, previously studied in \cite{Nguyen2025Burgers, Marbach2018}, is the following theorem. It follows from Proposition \ref{prop:asymptotic_estimate_kernel_affine+rest} and from Theorems \ref{thm:main:kernel_implies_drift_small_time} and \ref{thm:main:kernel_implies_drift_finite_time}.

\begin{theorem}[Obstructions in the case $\mu \equiv 1$]\label{thm:main_obstruction_mu_constant}
    Assume that $\mu \equiv 1$. Then, for all $k \geq 1$ even, one has $\langle \mu, \varphi_k \rangle = 0$, and there is an obstruction to STLC quantified by a $\widetilde{H}^{-\frac{5}{4}}$-drift along the direction $\varphi_k$, for controls small in $\widetilde{H}^{-\frac{3}{4}}$. In addition, the following properties hold:
    \begin{enumerate}[label=(\roman*)]
        \item if $k \geq 1$ is even and belongs to $\left( \sqrt{2}, 2 \sqrt{2} \right) + 2 \sqrt{2} \mathbb{Z}$, then this obstruction is in fact an obstruction to FTLC,
        \item if $k \geq 1$ is even and belongs to $\left( 0,  \sqrt{2} \right) + 2 \sqrt{2} \mathbb{Z}$, then the quadratic kernel $K_{\mu, k}$ is not globally signed.
    \end{enumerate}
\end{theorem}

We give more precise comments on the differences between Theorem \ref{thm:main_obstruction_mu_constant} and the results of \cite{Nguyen2025Burgers, Marbach2018}. In \cite{Marbach2018}, the author proves an obstruction to small-time local controllability for the Burgers equation \eqref{eq:Burgers_intro} in the case $\mu \equiv 1$, involving the $H^{-5/4}(0,T)$-norm of the control. From a historical point of view, this seems to be the first obstruction to small-time local controllability involving a fractional Sobolev norm of the control. In \cite{Marbach2018}, the quadratic term is not tested against an eigenfunction $\varphi_k$, but against an explicit function $x \mapsto \rho(x)$, chosen with convenient homogeneous boundary conditions simplifying the computations. The proof of the obstruction is quite involved and relies on \textit{weakly singular integral operators}. The obstruction is proved under the assumptions $\Vert u \Vert_{L^2(0,T)} \leq 1$ and $T>0$ small.

More recently, the local controllability of \eqref{eq:Burgers_intro} in the case $\mu \equiv 1$ was studied in \cite{Nguyen2025Burgers}. The authors improve the original result of \cite{Marbach2018} in two directions. First, they show that the obstruction to small-time local controllability can be strengthened into an obstruction to \textit{finite-time} local controllability, when the system is projected onto the directions $\varphi_2$ or $\varphi_{10}$. Second, they prove this under a smallness assumption on the $H^{-3/4}(0,T)$-norm of the control, which is weaker than the smallness assumptions of \cite{Marbach2018}. Their proof relies on Fourier analysis. Whereas \cite{Nguyen2025Burgers} uses the specific choices $\varphi_2$ and $\varphi_{10}$ to obtain signed kernels, Theorem \ref{thm:main_obstruction_mu_constant} gives a precise characterization of the integers $k$ for which the corresponding kernel is globally signed, thereby identifying the arithmetic structure behind these particular choices. Moreover, the smallness assumption in Theorem \ref{thm:main_obstruction_mu_constant} is consistent with that of \cite{Nguyen2025Burgers}, since $\widetilde{H}^{-3/4}(0,T)$ is the natural optimal regularity expected for the control in this setting; see Remark \ref{rq:optimal_y_1}.


\subsection{Additional bibliographical comments}

\textbf{Integer obstructions.} 
In finite dimension, that is, for local controllability problems for ordinary differential equations with scalar controls, only obstructions quantified by integer negative Sobolev norms of the control can occur. These obstructions are related to iterated Lie brackets. We refer, for instance, to the recent developments in \cite{BeauchardMarbach2018,KBFM24Unified}, and to the references therein. Obstructions of this type have also been observed for partial differential equations, notably in \cite{BeauchardMorancey2014,Bournissou2023_Quad,Coron2006} for bilinear Schrödinger equations with Dirichlet boundary conditions, in \cite{BeauchardMarbachPerrin} for bilinear Schrödinger equations with Neumann boundary conditions, in \cite{BeauchardMarbach2020} for a parabolic model, in \cite{CoronKoenigNguyen2024} for a fluid mechanics model, and in \cite{NiuXiang2025} for a Korteweg-de Vries system.

The classical method for establishing such obstructions consists in performing integrations by parts in the quadratic term $\langle y_2(T;u),\varphi_k\rangle$, so as to make successive primitives of the control appear. This idea is used, for instance, in \cite[Lemma 3.1]{BeauchardMorancey2014} and \cite[Proposition 5.1]{Bournissou2023_Quad} for Schrödinger equations with Dirichlet boundary conditions, in \cite[Appendix A]{BeauchardMarbachPerrin} for bilinear Schrödinger equations with Neumann boundary conditions, in \cite[Proposition 3.3]{BeauchardMarbach2020} for a parabolic system, and in \cite[Lemma 3.8]{CoronKoenigNguyen2024} for a water-tank system.

In some situations, however, this method cannot be applied. This is the case, for instance, in \cite{BeauchardMarbachPerrin} for bilinear Schrödinger equations with Neumann boundary conditions, due to the presence of infinitely many singularities in the quadratic kernel, and in \cite{NiuXiang2025} for certain critical lengths of the Korteweg-de Vries system. In such cases, one relies instead on a method involving the Fourier transform of the control, closer to the one often used to derive fractional obstructions.

\textbf{Fractional obstructions.} 
Since the discovery of the $H^{-5/4}$ drift for the Burgers system in the case $\mu \equiv 1$ in \cite{Marbach2018}, several other fractional drifts have been observed. Most of them rely on variants of the argument described around \eqref{eq:kernel_intro}. This is the case, for instance, in \cite{CoronKoenigNguyen2022} and \cite{Nguyen2025}, where $H^{-2/3}$ and $H^{-1/6}$ drifts are obtained for Korteweg-de Vries systems, and in \cite{Nguyen2025Burgers}, where the $H^{-5/4}$ drift is obtained in finite time for the Burgers system in the case $\mu=1$. Finally, in \cite{BeauchardMarbach2020}, a similar method is used to obtain $H^{-s}$ drifts for an academic parabolic system, with $s$ varying over an interval.

\textbf{Other types of controllability results for the Burgers equation.} 
The controllability of the Burgers equation has been studied extensively, for different types of controls, boundary conditions, and controllability notions. Since the present paper focuses specifically on quadratic obstruction mechanisms for scalar-input source terms, we do not attempt to give a complete survey of this broader literature. We refer instead to the bibliographical discussions in \cite{Marbach2018,Nguyen2025Burgers}. 

\subsection{Organization of the paper}

In Section \ref{sec:expression_quad_fourier}, we derive the Fourier representation of the quadratic expansion and define the kernel $K_{\mu,k}$. In Section \ref{sec:properties_of_kernels}, we study the asymptotic and sign properties of this kernel for several classes of source profiles. Section \ref{sec:proof_obstructions} is devoted to the kernel-to-drift principles and to the proof of the small-time and finite-time obstruction results. Section \ref{sec:remainder_estimates} contains the nonlinear remainder estimates needed to establish the drift estimates. Finally, various technical tools used throughout the paper are gathered in the appendix, together with well-posedness results for the heat and Burgers equations.

\section{Expression of the quadratic expansion in frequency}\label{sec:expression_quad_fourier}

In this section, we derive the Fourier-domain expression of the quadratic term $y_2$, projected onto a direction $\varphi_k$. We do not assume that $y_1(T)=0$; as a result, the formula holds only up to remainder terms, which we identify explicitly. For $n, k \geq 1$, set 
\begin{equation}\label{eq:def:lambda_n_k}
    \lambda_{n,k} := \lambda_n - \frac{\lambda_k}{2}.
\end{equation}
    
\begin{proposition}\label{prop:decomposition_y_2_kernel}
    Recall that $y_2$ is the solution of \eqref{eq:Burgers_y_2_intro}. Let $k \geq 1$. There exists $C_k > 0$ such that the following property holds. Let $\mu \in H^{-1}(0,1)$, $T > 0$, and $u \in L^2(0,T)$. For $\omega \in \mathbb{R}$, set
    \begin{equation}\label{eq:def:K_mu_k}
    K_{\mu, k}(\omega) = \sum_{n \geq 1} \frac{\langle \mu, \varphi_n \rangle}{n}  \left( \frac{\langle \mu, \varphi_{n+k} \rangle}{n+k} + \frac{\langle \mu, \varphi_{\vert n-k \vert} \rangle}{\vert n-k \vert} \mathbbm{1}_{n \neq k} \right) \frac{\lambda_{n,k}}{ \lambda_{n,k}^2 + \omega^2}
    \end{equation}
    and 
    \begin{equation}\label{eq:def:v_k_T}
        v_{k,T}(t) :=  u(t) e^{\frac{\lambda_k t}{2}} \mathbbm{1}_{(0,T)}(t).
    \end{equation}
    Then
    \begin{equation}\label{eq:prop:decomposition_y_2}
        \begin{split}
            & \left\vert \left\langle y_2(T), \varphi_k \right\rangle - \frac{k e^{-\lambda_k T}}{4 \pi^2 \sqrt{2}} \int_{\mathbb{R}} \left\vert \widehat{v_{k,T}}(\omega) \right\vert^2 K_{\mu,k}(\omega) \dd \omega \right\vert \\
            \leq \ & C_k \Vert y_1(T) \Vert_{H^{-1}}^2 + C_k \Vert \mu \Vert_{H^{-1}} \sum_{n\geq 1} \mathbbm{1}_{n^2 < \frac{k^2}{2}} \left\vert \langle y_1(T), \varphi_n \rangle \right\vert \left\vert \int_0^T u(t) e^{\left(\lambda_k - \lambda_n \right) t} \dd t \right\vert. 
        \end{split} 
    \end{equation}
\end{proposition}

\begin{remark}
    In particular, if $y_1(T) = 0$, then Proposition \ref{prop:decomposition_y_2_kernel} gives
    \begin{equation}
        \left\langle y_2(T), \varphi_k \right\rangle = \frac{k e^{-\lambda_k T}}{4 \pi^2 \sqrt{2}} \int_{\mathbb{R}} \left\vert \widehat{v_{k,T}}(\omega) \right\vert^2 K_{\mu,k}(\omega) \dd \omega.
    \end{equation}
    In the case $\mu \equiv 1$ and $y_1(T) = 0$, a different formula for $\left\langle y_2(T), \varphi_k \right\rangle$ was derived in \cite{Nguyen2025Burgers}.
\end{remark}

\begin{proof}
    For $n \geq 1$, we write $\mu_n := \langle \mu, \varphi_n \rangle$. By definition
    \begin{equation}
        \left\langle y_2(T), \varphi_k \right\rangle = \frac{1}{2} \int_0^T e^{-\lambda_k (T-\tau)} \left\langle (y_1(\tau))^2, \varphi_k^\prime \right\rangle \dd \tau,
    \end{equation}
    and $y_1(\tau) = \int_0^\tau u(t) \Phi (\tau - t) \dd t$, where $\Phi(t) := S(t) \mu$, with $S$ the heat semigroup. In particular, this gives
    \begin{equation}\label{eq:proof:prop:decomposition_y_2}
        \left\langle y_2(T), \varphi_k \right\rangle = \frac{1}{2} \int_0^T \int_0^T u(t) u(s) H(t, s) \dd s \dd t,
    \end{equation}
    with
    \begin{equation}
        H(t, s) := \int_{s \vee t}^T e^{-\lambda_k (T-\tau)} \left\langle \Phi (\tau - t) \Phi (\tau - s) , \varphi_k^\prime \right\rangle \dd \tau.
    \end{equation}

    \textbf{Step 1: decomposition of $\left\langle y_2(T), \varphi_k \right\rangle$ into two terms.}
    To lighten the computations, we introduce the shorthand
    \begin{equation}
        B_k(n,m) := \frac{\left\langle \varphi_n \varphi_m , \varphi_k^\prime \right\rangle}{\lambda_k - \lambda_n - \lambda_m}. 
    \end{equation}
    if $\lambda_k - \lambda_n - \lambda_m \neq 0$, and $B_k(n,m) = 0$ otherwise. One has 
    \begin{equation}\label{eq:expression_varphi_varphi_varphiprime}
        \left\langle \varphi_n \varphi_m, \varphi_k^\prime \right\rangle = \frac{k \pi}{\sqrt{2}} \left( \mathbbm{1}_{\vert n - m \vert = k} - \mathbbm{1}_{n + m = k} \right).
    \end{equation}
    Note that $(n \pm m)^2 = k^2$ implies $\lambda_k - \lambda_n - \lambda_m = \pm 2 \pi^2 mn$, yielding
    \begin{equation}\label{eq:expression_varphi_varphi_varphiprime_bis}
        B_k(n,m) = - \frac{k}{2 \pi \sqrt{2}} \left( \frac{ \mathbbm{1}_{\vert n - m \vert = k} + \mathbbm{1}_{n + m = k} }{nm} \right),
    \end{equation}
    a formula which will be used repeatedly below. We decompose $H$ into two terms by writing
    \begin{equation}
        \begin{split}
            H(t, s) 
            & = \sum_{n, m \geq 1} \mu_n \mu_m \left\langle \varphi_n \varphi_m , \varphi_k^\prime \right\rangle \int_{s \vee t}^T e^{-\lambda_k (T-\tau) - \lambda_n (\tau - t) - \lambda_m (\tau - s)}  \dd \tau 
             =: H_1(t,s) + H_2(t,s) ,
        \end{split}
    \end{equation}
    with $H_1(t,s) = \sum_{n, m \geq 1} \mu_n \mu_m B_k(n,m) e^{- \lambda_n (T - t) - \lambda_m (T - s)}$, and
    \begin{equation}
        \begin{split}
            H_2(t,s) 
            & = - \sum_{n, m \geq 1} \mu_n \mu_m B_k(n,m)  e^{-\lambda_k \left( T- (s \vee t) \right) - \lambda_n \left( (s \vee t) - t \right) - \lambda_m  \left( (s \vee t) - s \right)} \\
            & = - e^{-\lambda_k \left( T- (s \vee t) \right)}  \sum_{n, m \geq 1} \mu_n \mu_m B_k(n,m) \left( \mathbbm{1}_{s \geq t} e^{- \lambda_n \vert s - t \vert} + \mathbbm{1}_{s < t} e^{- \lambda_m \vert s - t \vert} \right) \\
            & = - e^{-\lambda_k \left( T- (s \vee t) \right)}  \sum_{n, m \geq 1} \mu_n \mu_m B_k(n,m)  e^{- \lambda_n \vert s - t \vert} \\
            & = - e^{-\lambda_k \left( T- \frac{t+s}{2} \right)}  \sum_{n, m \geq 1} \mu_n \mu_m B_k(n,m)  e^{- \left( \lambda_n - \frac{\lambda_k}{2} \right) \vert s - t \vert} . \label{eq:Hconv}
        \end{split}
    \end{equation}
    
    \textbf{Step 2: Fourier expression of the convolution kernel $H_2$.}
    For simplicity, write $v(t) := v_{k,T}(t) = u(t) e^{\frac{\lambda_k t}{2}} \mathbbm{1}_{(0,T)}$, so that
    \begin{equation}
        \begin{split}
            & \int_0^T \int_0^T u(t) u(s) H_2(t, s) \dd s \dd t 
            = - e^{-\lambda_k T}  \sum_{n, m \geq 1} \mu_n \mu_m B_k(n,m) \int_0^T \int_0^T v(t) v(s) e^{- \lambda_{n, k} \vert s - t \vert} \dd s \dd t .
        \end{split}
    \end{equation} 
    
    We use the following lemma, whose proof is given below. 
    
    \begin{lemma}\label{lem:fourier_alpha_negatif}
        For $\alpha \in \mathbb{R}$, and $u \in L^2(\mathbb{R})$ compactly supported, one has
        \begin{equation}\label{eq:lem:fourier_alpha_negatif}
            \int_{\mathbb{R}} \int_{\mathbb{R}} u(t) u(s) e^{-\alpha \vert t-s \vert} \dd s \dd t = \frac{1}{\pi} \int_{\mathbb{R}} \frac{\alpha}{\alpha^2 + \omega^2} \left\vert \widehat{u}(\omega) \right\vert^2 \dd \omega + 2 \mathbbm{1}_{\alpha < 0} \widehat{u}(i \alpha) \widehat{u}(- i \alpha) + \mathbbm{1}_{\alpha = 0} \widehat{u}(0)^2.
        \end{equation}
    \end{lemma}
    
    For all $n \geq 1$, note that $\lambda_{n, k} = \lambda_n - \frac{\lambda_k}{2} \neq 0$.
    Lemma \ref{lem:fourier_alpha_negatif} thus gives
    \begin{equation}
        \int_0^T \int_0^T u(t) u(s) H_2(t, s) \dd s \dd t = I_1 + I_2,
    \end{equation}
    with
    \begin{equation}
        I_1 := - \frac{e^{-\lambda_k T}}{\pi} \int_{\mathbb{R}} \left\vert \widehat{v}(\omega) \right\vert^2 \sum_{n, m \geq 1} \mu_n \mu_m B_k(n,m) \frac{\lambda_{n,k}}{ \lambda_{n,k}^2 + \omega^2} \dd \omega,
    \end{equation}
    and 
    \begin{equation}
        \begin{split}
             I_2 := -2 e^{-\lambda_k T}  \sum_{n, m \geq 1} \mu_n \mu_m B_k(n,m)  \mathbbm{1}_{\lambda_n < \frac{\lambda_k}{2}} \left(\int_{\mathbb{R}} v(t) e^{\lambda_{n, k} t} \dd t \right) \left(\int_{\mathbb{R}} v(t) e^{- \lambda_{n, k} t} \dd t \right).
        \end{split}
   \end{equation}
    
    \textbf{Expression of the main term $I_1$.}
    By definition, one has
    \begin{equation}
        I_1 = \frac{k e^{-\lambda_k T}}{2 \pi^2 \sqrt{2}} \int_{\mathbb{R}} \left\vert \widehat{v_{k,T}}(\omega) \right\vert^2 K_{\mu,k}(\omega) \dd \omega ,
    \end{equation}
    where $K_{\mu,k}$ is given by 
    \begin{equation}\label{eq:ancienne_formula_K_mu_k}
        K_{\mu,k}(\omega) = \sum_{n \geq 1} \frac{\mu_n}{n}  \left( \sum_{m \geq 1} \frac{\mu_m}{m} \left( \mathbbm{1}_{\vert n - m \vert = k} + \mathbbm{1}_{n + m = k} \right) \right) \frac{\lambda_{n,k}}{ \lambda_{n,k}^2 + \omega^2}.
    \end{equation}
    Since 
    \begin{equation}
        \begin{split}
            K_{\mu, k}(\omega) & = \sum_{n \geq 1} \frac{\mu_n}{n}  \left( \frac{\langle \mu, \varphi_{n+k} \rangle}{n+k} + \frac{\langle \mu, \varphi_{n-k} \rangle}{n-k} \mathbbm{1}_{n \geq k + 1} + \frac{\langle \mu, \varphi_{k-n} \rangle}{k-n} \mathbbm{1}_{n \leq k - 1} \right) \frac{\lambda_{n,k}}{ \lambda_{n,k}^2 + \omega^2} \\
            & = \sum_{n \geq 1} \frac{\mu_n}{n}  \left( \frac{\langle \mu, \varphi_{n+k} \rangle}{n+k} + \frac{\langle \mu, \varphi_{\vert n-k \vert} \rangle}{\vert n-k \vert} \mathbbm{1}_{n \neq k} \right) \frac{\lambda_{n,k}}{ \lambda_{n,k}^2 + \omega^2}, 
        \end{split}
    \end{equation}
    the kernel defined in \eqref{eq:ancienne_formula_K_mu_k} coincides with the one defined in \eqref{eq:def:K_mu_k}. Keeping track of the factor $\frac{1}{2}$ in \eqref{eq:proof:prop:decomposition_y_2}, one finds that $I_1$ gives the main term in \eqref{eq:prop:decomposition_y_2}.

    \textbf{Estimation of the low frequency term $I_2$.}
    By definition of $v$, one has $v(t) e^{\lambda_{n, k} t} = u(t) e^{\lambda_n t}$ and $v(t) e^{- \lambda_{n, k} t} = u(t) e^{\left(\lambda_k - \lambda_n \right) t}$. Hence, using \eqref{eq:y_1_Fourier_intro}, one finds
    \begin{equation}
        \begin{split}
            I_2 = & -2 e^{-\lambda_k T} \sum_{n \geq 1} \mathbbm{1}_{\lambda_n < \frac{\lambda_k}{2}} e^{\lambda_n T} \langle y_1(T), \varphi_n \rangle \left(\int_{\mathbb{R}} u(t) e^{\left(\lambda_k - \lambda_n \right) t} \dd t \right) 
            \left( \sum_{m \geq 1} \mu_m B_k(n,m) \right).
        \end{split}
    \end{equation}
    For each $n$, the sum over the index $m$ contains at most three terms. Since $\frac{\mu_m^2}{\lambda_m} \leq \Vert \mu \Vert_{H^{-1}}^2$, this gives 
    \begin{equation}
        \left\vert \sum_{m \geq 1} \mu_m B_k(n,m) \right\vert
        \leq \frac{k}{2 n \pi \sqrt{2}} \left\vert \sum_{m \geq 1} \frac{\mu_m}{m} \left( \mathbbm{1}_{\vert n - m \vert = k} + \mathbbm{1}_{n + m = k} \right) \right\vert 
        \leq \frac{C k}{n} \Vert \mu \Vert_{H^{-1}},
    \end{equation}
    for some absolute constant $C > 0$. In addition, for $n$ such that $1 \leq n^2 < \frac{k^2}{2}$, one has $\frac{e^{-\lambda_k T} e^{\lambda_n T}}{n} \leq 1$. Hence, one obtains
    \begin{equation}
        \vert I_2 \vert \leq C_k \Vert \mu \Vert_{H^{-1}} \sum_{n\geq 1} \mathbbm{1}_{n^2 < \frac{k^2}{2}} \left\vert \langle y_1(T), \varphi_n \rangle \right\vert \left\vert \int_0^T u(t) e^{\left(\lambda_k - \lambda_n \right) t} \dd t \right\vert,
    \end{equation}
    for some constant $C_k > 0$, which depends only on $k$.  
    
    \textbf{Step 3: estimation of the term involving the moments, $H_1$.}
    Using \eqref{eq:y_1_Fourier_intro}, one finds
    \begin{equation}
        \begin{split}
            & \int_0^T \int_0^T u(t) u(s) H_1(t, s) \dd s \dd t 
            = \sum_{n, m \geq 1} B_k(n,m) \left\langle y_1(T), \varphi_n \right\rangle \left\langle y_1(T), \varphi_m \right\rangle.
        \end{split}
    \end{equation}
    The Cauchy-Schwarz inequality gives
    \begin{equation}
        \begin{split}
            \left\vert \int_0^T \int_0^T u(t) u(s) H_1(t, s) \dd s \dd t \right\vert 
            & \leq \frac{k}{2 \pi \sqrt{2}} \sum_{n, m \geq 1} \frac{ \left( \mathbbm{1}_{\vert n - m \vert = k} + \mathbbm{1}_{n + m = k} \right)^2}{n^2} \left\langle y_1(T), \varphi_n \right\rangle^2 \\
            & \leq \frac{3 k}{2 \pi \sqrt{2}} \sum_{n \geq 1} \frac{1}{n^2} \left\langle y_1(T), \varphi_n \right\rangle^2 
            \leq C k \Vert y_1(T) \Vert_{H^{-1}}^2,
        \end{split}
    \end{equation}
    for some absolute constant $C > 0$. Hence, the term involving $H_1$ is a remainder term of \eqref{eq:prop:decomposition_y_2}. This completes the proof of Proposition \ref{prop:decomposition_y_2_kernel}.
\end{proof}

It remains to prove Lemma \ref{lem:fourier_alpha_negatif}.

\begin{proof}[Proof of Lemma \ref{lem:fourier_alpha_negatif}.]
    Note that \eqref{eq:lem:fourier_alpha_negatif} holds true when $\alpha = 0$.
    Next, assume $\alpha > 0$. For $\omega \in \mathbb{R}$, one has
    \begin{equation}
        \int_{\mathbb{R}} e^{- \alpha \vert t \vert} e^{- i t \omega} \dd t = \frac{2 \alpha}{\alpha^2 + \omega^2}.
    \end{equation}
    Writing
    \begin{equation}
        \int_{\mathbb{R}} \int_{\mathbb{R}} u(t) u(s) e^{-\alpha \vert t-s \vert} \dd s \dd t = \left\langle e^{-\alpha \vert \cdot \vert} \ast u, u \right\rangle_{L^2(\mathbb{R})} = \frac{1}{2 \pi} \int_{\mathbb{R}} \left\vert \widehat{u}(\omega) \right\vert^2 \widehat{e^{-\alpha \vert \cdot \vert}}(\omega) \dd \omega,
    \end{equation}
    this gives \eqref{eq:lem:fourier_alpha_negatif}. Finally, if $\alpha < 0$, then the identity $e^{\alpha \vert t \vert} + e^{-\alpha \vert t \vert} = e^{\alpha t} + e^{-\alpha t}$ yields
    \begin{equation}
        \int_{\mathbb{R}} \int_{\mathbb{R}} u(t) u(s) e^{-\alpha \vert t-s \vert} \dd s \dd t
        = - \int_{\mathbb{R}} \int_{\mathbb{R}} u(t) u(s) e^{\alpha \vert t-s \vert} \dd s \dd t
        + 2 \left(\int_{\mathbb{R}} u(t) e^{-\alpha t} \dd t \right) \left(\int_{\mathbb{R}} u(t) e^{+\alpha t} \dd t \right),
    \end{equation}
    which implies \eqref{eq:lem:fourier_alpha_negatif}, by \eqref{eq:lem:fourier_alpha_negatif} applied with $- \alpha$. 
\end{proof}

\section{Asymptotic and sign properties of kernels}\label{sec:properties_of_kernels}

As established in Proposition \ref{prop:decomposition_y_2_kernel}, the projection of the quadratic expansion $y_2$ onto a direction $\varphi_k$ is governed by the kernel $K_{\mu,k}$ defined in \eqref{eq:def:K_mu_k}. The purpose of this section is to understand how the source profile $\mu$ influences the properties of this kernel. We first derive sufficient conditions leading to $\widetilde{H}^{-1}$ and $\widetilde{H}^{-5/4}$ drifts, and then construct source profiles in physical space that produce fractional drifts over a full range of Sobolev regularities. Second, we construct source profiles by prescribing their Fourier coefficients, in such a way that $K_{\mu,k}$ satisfies both a global sign property and a fractional asymptotic estimate. Finally, we study the classical case $\mu \equiv 1$, where we extend previous results by characterizing the directions $\varphi_k$ for which a global sign property holds.

We gather in Lemma~\ref{lem:asymptotic_sum_appendix} of the appendix several elementary estimates for sums that will be used throughout this section.

\subsection{Asymptotic properties of quadratic kernels}

\subsubsection{Obstructions quantified by the $H^{-1}(0,T)$-norm of the control}

We begin with the following result, which shows that the projection of the quadratic expansion is close to the square of the $\widetilde{H}^{-1}(0,T)$-norm of the control, under the sole assumption that $\mu \in L^2(0,1)$.

\begin{proposition}[Key asymptotic estimate for the $H^{-1}$-drift]\label{prop:asymptotic_estimate_kernel_L2}
    Let $\mu \in L^2(0,1)$ and $k \geq 1$. Then
    \begin{equation}\label{eq:prop:asymptotic_estimate_kernel_L2_1}
        K_{\mu,k}(\omega) = \frac{\pi \sqrt{2} \left\langle \mu^2, \varphi_k^\prime \right\rangle}{k \omega^{2}} + o\left( \frac{1}{ \omega^{2}} \right), \quad \text{     when $\omega \rightarrow + \infty$.}
    \end{equation}
    If $\mu \in H_{(0)}^\nu(0,1)$ for some $\nu \in \left(0, 2\right]$, then
    \begin{equation}\label{eq:prop:asymptotic_estimate_kernel_L2_2}
        K_{\mu,k}(\omega) = \frac{\pi \sqrt{2} \left\langle \mu^2, \varphi_k^\prime \right\rangle}{k \omega^{2}} + \mathcal{O}\left( \frac{1}{ \omega^{2+\nu}} \right), \quad \text{ when $\omega \rightarrow + \infty$.}
    \end{equation}
\end{proposition}

\begin{proof}
    In this proof, we use the shorthand $\mu_n := \langle \mu, \varphi_n \rangle$, for $n \geq 1$. Formally, one has $K_{\mu, k}(\omega) \approx \frac{A_{k}(\mu)}{\omega^2}$, where
    \begin{equation}
        A_k(\mu) := \sum_{n \geq 1} \frac{\mu_n}{n} \left( \frac{\mu_{n+k}}{n+k} + \frac{\mu_{\vert n-k \vert}}{\vert n-k \vert} \mathbbm{1}_{n \neq k} \right) \lambda_{n,k}.
    \end{equation}
    
    \textbf{Step 1: we prove that $A_k(\mu) = \frac{\pi \sqrt{2}}{k} \left\langle \mu^2, \varphi_k^\prime \right\rangle$.} 
    We first prove this under the additional assumption that there exists $N \geq 1$ such that $\mu_n = 0$ for all $n > N$, which implies that all the sums involved in this computation are finite. By \eqref{eq:expression_varphi_varphi_varphiprime}, one has
    \begin{equation}
        \left\langle \mu^2, \varphi_k^\prime \right\rangle = \frac{k \pi}{\sqrt{2}} \sum_{n,m \geq 1} \mu_n \mu_m \left( \mathbbm{1}_{\vert n - m \vert = k} - \mathbbm{1}_{n + m = k} \right).
    \end{equation}
    We claim that 
    \begin{equation}\label{eq_proof:prop:asymptotic_estimate_kernel_L2_step_1_1}
        \sum_{n,m \geq 1} \mu_n \mu_m \mathbbm{1}_{n + m = k} = - \frac{1}{\pi^2} \sum_{n \geq 1} \frac{\mu_n}{n} \left( \frac{\mu_{k-n}}{k-n} \mathbbm{1}_{k-n\geq 1} \right) \lambda_{n,k},
    \end{equation}
    and
    \begin{equation}\label{eq_proof:prop:asymptotic_estimate_kernel_L2_step_1_2}
        \sum_{n,m \geq 1} \mu_n \mu_m  \mathbbm{1}_{\vert n - m \vert = k} = \frac{1}{\pi^2} \sum_{n \geq 1} \frac{\mu_n}{n} \left( \frac{\mu_{n+k}}{n+k} + \frac{\mu_{n-k}}{n-k} \mathbbm{1}_{n-k \geq 1} \right) \lambda_{n,k},
    \end{equation}
    yielding $A_k(\mu) = \frac{\pi \sqrt{2}}{k} \left\langle \mu^2, \varphi_k^\prime \right\rangle$.
    
    \underline{Proof of \eqref{eq_proof:prop:asymptotic_estimate_kernel_L2_step_1_1}.} The change of index $n^\prime = k-n$ gives
    \begin{equation}
        \sum_{n = 1}^{k-1} \frac{\mu_n \mu_{k-n}}{n(k-n)} \left( n^2 - \frac{k^2}{2} \right)
        = \sum_{n = 1}^{k-1} \frac{\mu_{k-n} \mu_n}{(k-n)n} \left( n^2 - 2 kn + \frac{k^2}{2} \right).
    \end{equation}
    Averaging, one finds
    \begin{equation}
        \begin{split}
            \sum_{n \geq 1} \frac{\mu_n}{n} \left( \frac{\mu_{k-n}}{k-n} \mathbbm{1}_{k-n\geq 1} \right) \lambda_{n,k}
            = \ &  \pi^2 \sum_{n = 1}^{k-1} \frac{\mu_{k-n} \mu_n}{(k-n)n} \left( \frac{1}{2} \left( n^2 - \frac{k^2}{2} \right) + \frac{1}{2} \left( n^2 - 2 kn + \frac{k^2}{2} \right)  \right) \\
            = \ & - \pi^2 \sum_{n,m \geq 1} \mu_n \mu_m \mathbbm{1}_{n + m = k}. \\
        \end{split}
    \end{equation}

    \underline{Proof of \eqref{eq_proof:prop:asymptotic_estimate_kernel_L2_step_1_2}.} The change of index $n^\prime = n-k$ gives
    \begin{equation}
        \begin{split}
            \sum_{n \geq 1} \frac{\mu_n}{n} \left( \frac{\mu_{n+k}}{n+k} + \frac{\mu_{n-k}}{n-k} \mathbbm{1}_{n-k \geq 1} \right) \lambda_{n,k}
            & = \sum_{n \geq 1} \frac{\mu_n \mu_{n+k}}{n(n+k)} \lambda_{n,k} + \sum_{n \geq k+1} \frac{\mu_{n} \mu_{n-k}}{n(n-k)} \lambda_{n,k} \\
            & = \sum_{n \geq 1} \frac{\mu_n \mu_{n+k}}{n(n+k)} \left( \lambda_{n,k} + \lambda_{n+k,k} \right) \\
            & = 2 \pi^2 \sum_{n \geq 1} \mu_n \mu_{n+k} 
            = \pi^2 \sum_{n,m \geq 1} \mu_n \mu_m  \mathbbm{1}_{\vert n - m \vert = k}.
        \end{split}
    \end{equation}
    
    It remains to prove $A_k(\mu) = \frac{\pi \sqrt{2}}{k} \left\langle \mu^2, \varphi_k^\prime \right\rangle$ in the general case. Let $N \geq 1$. The preceding computation gives
    \begin{equation}
        \sum_{n \geq 1} \frac{\widetilde{\mu}_n}{n} \left( \frac{\widetilde{\mu}_{n+k}}{n+k} + \frac{\widetilde{\mu}_{\vert n-k \vert}}{\vert n-k \vert} \mathbbm{1}_{n \neq k} \right) \lambda_{n,k}
        = \pi^2 \sum_{n,m \geq 1} \widetilde{\mu}_n \widetilde{\mu}_m \left( \mathbbm{1}_{\vert n - m \vert = k} - \mathbbm{1}_{n + m = k} \right),
    \end{equation}
    where $\widetilde{\mu}_n := \mu_n \mathbbm{1}_{n \leq N}$. Since $(\mu_n) \in \ell^2(\mathbb{N}^\ast)$, the identity follows by letting $N \rightarrow + \infty$.
    
    \textbf{Step 2: we prove \eqref{eq:prop:asymptotic_estimate_kernel_L2_1}.} By Step 1, one has
    \begin{equation}
        \left\vert K_{\mu,k}(\omega) - \frac{\pi \sqrt{2} \left\langle \mu^2, \varphi_k^\prime \right\rangle}{k \omega^{2}} \right\vert
        \leq \frac{1}{\omega^2} \sum_{n \geq 1} \frac{\vert \mu_n \vert}{n} \left( \frac{ \vert \mu_{n+k} \vert}{n+k} + \frac{\vert \mu_{\vert n-k \vert} \vert}{\vert n-k \vert} \mathbbm{1}_{n \neq k} \right) \frac{ \vert \lambda_{n,k} \vert^3}{\lambda_{n,k}^2 + \omega^2}.
    \end{equation}
    For all $\omega > 0$, one has $\frac{ \vert \lambda_{n,k} \vert^3}{\lambda_{n,k}^2 + \omega^2} \leq \vert \lambda_{n,k} \vert$. Since $\mu \in L^2(0,1)$, the Cauchy-Schwarz inequality gives
    \begin{equation}
        \sum_{n \geq 1} \frac{\vert \mu_n \vert}{n} \left( \frac{ \vert \mu_{n+k} \vert}{n+k} + \frac{\vert \mu_{\vert n-k \vert} \vert}{\vert n-k \vert} \mathbbm{1}_{n \neq k} \right) \vert \lambda_{n,k} \vert 
        \leq C \sum_{n \geq 1} \vert \mu_n \vert \left( \vert \mu_{n+k} \vert + \vert \mu_{\vert n-k \vert} \vert \right) \leq 2 C \Vert \mu \Vert_{L^2}^2,
    \end{equation}
    for some $C>0$ depending only on $k$. Hence, the dominated convergence theorem yields 
    \begin{equation}
         \sum_{n \geq 1} \frac{\vert \mu_n \vert}{n} \left( \frac{ \vert \mu_{n+k} \vert}{n+k} + \frac{\vert \mu_{\vert n-k \vert} \vert}{\vert n-k \vert} \mathbbm{1}_{n \neq k} \right) \frac{ \vert \lambda_{n,k} \vert^3}{\lambda_{n,k}^2 + \omega^2} \longrightarrow 0,
    \end{equation}
    when $\omega \rightarrow + \infty$. This completes the proof of \eqref{eq:prop:asymptotic_estimate_kernel_L2_1}.

    \textbf{Step 3: we prove \eqref{eq:prop:asymptotic_estimate_kernel_L2_2}.} Assume that $\mu \in H_{(0)}^\nu(0,1)$ for some $\nu \in \left(0, 2 \right]$. By Step 2, it suffices to show
    \begin{equation}
        \sum_{n \geq 1} \frac{\vert \mu_n \vert}{n} \left( \frac{ \vert \mu_{n+k} \vert}{n+k} + \frac{\vert \mu_{\vert n-k \vert} \vert}{\vert n-k \vert} \mathbbm{1}_{n \neq k} \right) \frac{ \vert \lambda_{n,k} \vert^3}{\lambda_{n,k}^2 + \omega^2}
        \leq \frac{C}{\omega^{\nu}},
    \end{equation}
    for some $C > 0$ depending on $\mu$, $\nu$ and $k$, and for all $\omega > 0$ sufficiently large. For $\omega > 0$ and $n \geq 1$, one has 
    \begin{equation}
        \frac{ \vert \lambda_{n,k} \vert^{2 - \nu}}{ \lambda_{n,k}^2 + \omega^2 } 
        = \frac{1}{\omega^{\nu}} \chi \left( \frac{\vert \lambda_{n,k} \vert}{\omega} \right),
    \end{equation}
    with $\chi(x) := \frac{x^{2-\nu}}{x^2 + 1}$, for $x \geq 0$. Since $\nu \in [0, 2]$, the function $\chi$ is bounded, implying that 
    \begin{equation}
        \sum_{n \geq 1} \frac{\vert \mu_n \vert}{n} \left( \frac{ \vert \mu_{n+k} \vert}{n+k} + \frac{\vert \mu_{\vert n-k \vert} \vert}{\vert n-k \vert} \mathbbm{1}_{n \neq k} \right) \frac{ \vert \lambda_{n,k} \vert^3}{\lambda_{n,k}^2 + \omega^2}
        \leq \frac{C}{\omega^{\nu}} \sum_{n \geq 1} \frac{\vert \mu_n \vert}{n} \left( \frac{ \vert \mu_{n+k} \vert}{n+k} + \frac{\vert \mu_{\vert n-k \vert} \vert}{\vert n-k \vert} \mathbbm{1}_{n \neq k} \right)  \vert \lambda_{n,k} \vert^{1+\nu},
    \end{equation}
    for some constant $C>0$ which depends only on $\nu$. Moreover, there exists $C> 0$, which depends only on $k$ and $\nu$, such that 
    \begin{equation}
        \frac{\vert \lambda_{n,k} \vert^{1+\nu}}{n(n+k)} \leq C n^\nu (n+k)^\nu, \quad \text{for all $n \geq 1$,}
    \end{equation}
    and 
    \begin{equation}
        \frac{\vert \lambda_{n,k} \vert^{1+\nu}}{n \vert n - k \vert} \leq C n^\nu \vert n - k \vert^\nu, \quad \text{for all $n \geq 1$, with $n \neq k$.}
    \end{equation}
    Hence, one obtains
    \begin{equation}
        \sum_{n \geq 1} \frac{\vert \mu_n \vert}{n} \left( \frac{ \vert \mu_{n+k} \vert}{n+k} + \frac{\vert \mu_{\vert n-k \vert} \vert}{\vert n-k \vert} \mathbbm{1}_{n \neq k} \right) \frac{ \vert \lambda_{n,k} \vert^3}{\lambda_{n,k}^2 + \omega^2}
        \leq \frac{C}{\omega^{\nu}} \sum_{n \geq 1} n^{2\nu} \mu_n^2,
    \end{equation}
    by the Cauchy-Schwarz inequality, for some $C> 0$ which depends only on $k$ and $\nu$. This completes the proof.
\end{proof}

\subsubsection{Obstructions quantified by the $H^{-5/4}(0,T)$-norm of the control}

Next, we prove the following result, which will be one of the main ingredients in the proof of the $H^{-5/4}$-drift in small time. When $\left\langle \mu^2, \varphi_k^\prime \right\rangle = 0$, \eqref{eq:prop:asymptotic_estimate_kernel_affine+rest_1} implies that the projection along the direction $\varphi_k$ of the solution of the quadratic problem is close to the square of the $\widetilde{H}^{-\frac{5}{4}}(0, T)$-norm of the control. Note that an estimate similar to \eqref{eq:prop:asymptotic_estimate_kernel_affine+rest_2} was obtained in \cite[Proposition 3.1]{Nguyen2025Burgers}, with a different formula for the kernel.

\begin{proposition}[Key asymptotic estimate for the $H^{-5/4}$-drift]\label{prop:asymptotic_estimate_kernel_affine+rest}
    Let $\ell : [0, 1] \rightarrow \mathbb{R}$ be an affine function; in particular, $\ell \in \mathcal{H}^1_{\infty}(0,1)$. Let $k \geq 1$, $s \in \left(\frac{1}{2}, 2\right]$ and $\widetilde{\mu} \in\mathcal{H}^s_2(0,1)$. Set $\mu := \ell + \widetilde{\mu}$. One has
    \begin{equation}\label{eq:prop:asymptotic_estimate_kernel_affine+rest_1}
        K_{\mu,k}(\omega) = \frac{\pi \sqrt{2} \left\langle \mu^2, \varphi_k^\prime \right\rangle}{k \omega^{2}} - \frac{\pi^2 \sqrt{2} \left(\mu(0)^2 + (-1)^k \mu(1)^2 \right)}{\omega^{\frac{5}{2}}} + \mathcal{O}\left( \frac{1}{\omega^{\frac{5}{2}+\nu}} \right)
    \end{equation}
    when $\omega \rightarrow + \infty$, where $\nu := \frac{s}{2} - \frac{1}{4} > 0$. If $\mu \equiv 1$ and $k$ is even, then
    \begin{equation}\label{eq:prop:asymptotic_estimate_kernel_affine+rest_2}
        K_{\mu,k}(\omega) = - \frac{2 \pi^2 \sqrt{2}}{\omega^{\frac{5}{2}}} + \mathcal{O}\left( \frac{1}{\omega^{\frac{7}{2}}} \right), \quad \text{ when $\omega \rightarrow + \infty$.}
    \end{equation}
\end{proposition}

\begin{proof}
    In this proof, the symbol $\lesssim$ is used for constants which may depend only on $k$, $\mu$ and $s$, and all $\mathcal{O}$ estimates are understood as $\omega \rightarrow +\infty$. For $n \geq 1$, write $\mu_n := \langle \mu, \varphi_n \rangle$, $\ell_n := \langle \ell, \varphi_n \rangle$ and $\widetilde{\mu}_n := \langle \widetilde{\mu}, \varphi_n \rangle$. One has
    \begin{equation}
        \ell_n = \frac{\sqrt{2} \left( \ell(0) - (-1)^n \ell(1) \right)}{\pi n} \text{ for all $n \geq 1$, and } \sum_{n \geq 1} n^{2s} \widetilde{\mu}_n^2 < + \infty.
    \end{equation}
    Note that $\ell(0) = \mu(0)$ and $\ell(1) = \mu(1)$, since $s > \frac{1}{2}$. By the proof of Proposition \ref{prop:asymptotic_estimate_kernel_L2}, one has
    \begin{equation}
        K_{\mu, k}(\omega) - \frac{\pi \sqrt{2} \left\langle \mu^2, \varphi_k^\prime \right\rangle}{k \omega^{2}} = - \frac{1}{\omega^2} \sum_{n \geq 1} \frac{\mu_n}{n} \left( \frac{\mu_{n+k} }{n+k} + \frac{ \mu_{\vert n-k \vert} }{\vert n-k \vert} \mathbbm{1}_{n \neq k} \right) \frac{\lambda_{n,k}^3}{ \lambda_{n,k}^2 + \omega^2}. 
    \end{equation}
    Set
    \begin{equation}
        K_1(\omega) := - \frac{1}{\omega^2} \sum_{n \geq 1} \frac{\ell_n}{n} \left( \frac{\ell_{n+k} }{n+k} + \frac{\ell_{\vert n-k \vert} }{\vert n-k \vert} \mathbbm{1}_{n \neq k} \right) \frac{\lambda_{n,k}^3}{ \lambda_{n,k}^2 + \omega^2},
    \end{equation}
    and $K_2(\omega) := K_{\mu, k}(\omega) - \frac{\pi \sqrt{2} \left\langle \mu^2, \varphi_k^\prime \right\rangle}{k \omega^{2}} - K_1(\omega)$.
    
    \textbf{Step 1: we prove that}
    \begin{equation}\label{eq:prop:asymptotic_estimate_kernel_affine+rest_step1_1}
        K_1(\omega) = - \frac{\pi^2 \sqrt{2} \left(\mu(0)^2 + (-1)^k \mu(1)^2 \right)}{\omega^{\frac{5}{2}}} + \mathcal{O}\left( \frac{1}{\omega^{\frac{7}{2}}} \right),
    \end{equation}
    when $\omega \rightarrow + \infty$. The contribution of the terms with $n \leq k$ is $\mathcal{O}(\omega^{-4})$, implying
    \begin{equation}
        K_1(\omega) = - \frac{1}{\omega^2} \sum_{n > k} \frac{\ell_n}{n} \left( \frac{\ell_{n+k} }{n+k} + \frac{\ell_{n-k} }{n-k} \right) \frac{\lambda_{n,k}^3}{ \lambda_{n,k}^2 + \omega^2} + \mathcal{O} \left( \frac{1}{\omega^4} \right).
    \end{equation}
    One has 
    \begin{equation}\label{eq:prop:asymptotic_estimate_kernel_affine+rest_step1_2}
        \begin{split}
            K_1(\omega) 
            = \ & - \frac{2}{\pi^2 \omega^2} \left( \ell(0)^2 + (-1)^k \ell(1)^2 \right) \sum_{n > k} \frac{1}{n^2} \left( \frac{1}{(n+k)^2} + \frac{1}{(n-k)^2} \right) \frac{\lambda_{n,k}^3}{ \lambda_{n,k}^2 + \omega^2} \\
            & + \frac{2}{\pi^2 \omega^2} \ell(0) \ell(1) \left( 1 + (-1)^k \right) \sum_{n > k} \frac{(-1)^n}{n^2} \left( \frac{1}{(n+k)^2} + \frac{1}{(n-k)^2} \right)  \frac{\lambda_{n,k}^3}{ \lambda_{n,k}^2 + \omega^2} \\
            & + \mathcal{O} \left( \frac{1}{\omega^4} \right).
        \end{split}
    \end{equation}
    For $n > k$, one has $n^2 - \frac{k^2}{2} \geq \frac{n^2}{2}$, implying that for all $\omega > 0$, 
    \begin{equation}\label{eq:prop:asymptotic_estimate_kernel_affine+rest_step1_2_bis}
        \left\vert \frac{1}{ \lambda_{n,k}^2 + \omega^2} - \frac{1}{\pi^4 n^4 + \omega^2} \right\vert \lesssim \frac{n^2}{\left( n^4 + \omega^2 \right)^2} \leq \frac{1}{n^2 \left( n^4 + \omega^2 \right)}.
    \end{equation}
    Using also
    \begin{equation}
        \left\vert \left( \frac{1}{(n+k)^2} + \frac{1}{(n-k)^2} \right) - \frac{2}{n^2} \right\vert \lesssim \frac{1}{n^4},
    \end{equation}
    one obtains
    \begin{equation}\label{eq:prop:asymptotic_estimate_kernel_affine+rest_step1_3}
        \left\vert \frac{1}{n^2} \left( \frac{1}{(n+k)^2} + \frac{1}{(n-k)^2} \right) \frac{\lambda_{n,k}^3}{ \lambda_{n,k}^2 + \omega^2} - \frac{2 \pi^6 n^2}{\pi^4 n^4 + \omega^2} \right\vert \lesssim \frac{1}{n^4 + \omega^2},
    \end{equation}
    for all $n > k$ and $\omega > 0$.

    Using \eqref{eq:prop:asymptotic_estimate_kernel_affine+rest_step1_3} and \eqref{eq:lem:asymptotic_sum_appendix_1} with $\alpha = 0$, in \eqref{eq:prop:asymptotic_estimate_kernel_affine+rest_step1_2}, one finds 
    \begin{equation}
        \begin{split}
            K_1(\omega) 
            = \ & - \frac{4\pi^4}{\omega^2} \left( \ell(0)^2 + (-1)^k \ell(1)^2 \right) \sum_{n > k} \frac{n^2}{\pi^4 n^4 + \omega^2} \\
            & + \frac{4 \pi^4}{\omega^2} \ell(0) \ell(1) \left( 1 + (-1)^k \right) \sum_{n > k} (-1)^n \frac{n^2}{\pi^4 n^4 + \omega^2}
            + \mathcal{O} \left( \frac{1}{\omega^{\frac{7}{2}}} \right).
        \end{split}
    \end{equation}
    Using \eqref{eq:lem:asymptotic_sum_appendix_3} and \eqref{eq:lem:asymptotic_sum_appendix_4}, this gives \eqref{eq:prop:asymptotic_estimate_kernel_affine+rest_step1_1}. In particular, \eqref{eq:prop:asymptotic_estimate_kernel_affine+rest_2} holds true.
    
    \textbf{Step 2: we prove that $K_2(\omega) = \mathcal{O}\left( \frac{1}{\omega^{\frac{5}{2}+\nu}} \right)$, where $\nu := \frac{s}{2} - \frac{1}{4} > 0$.}
    As above, the terms with $n \leq k$ yield a remainder term of size $\mathcal{O}\left( \frac{1}{\omega^4} \right)$, so they may be removed from the sums. Using also $\vert \ell_n \vert \lesssim \frac{\Vert \ell \Vert_{L^{\infty}}}{n}$ and $n^2 \geq n^2 - \frac{k^2}{2} \geq \frac{n^2}{2}$, for $n > k$, one finds
    \begin{equation}
        \begin{split}
            \left\vert K_2(\omega) \right\vert 
            \lesssim \ & \frac{1}{\omega^2} \sum_{n > k} \frac{n^6}{n^4 + \omega^2} 
            \left\{ 
                \left( \frac{\Vert \ell \Vert_{L^\infty}}{n^2} + \frac{\left\vert \widetilde{\mu}_n \right\vert}{n} \right) \left( \frac{\left\vert \widetilde{\mu}_{n+k} \right\vert}{n + k} + \frac{\left\vert \widetilde{\mu}_{n-k} \right\vert}{n - k} \right) \right. \\
                & \hspace{3cm} \left. + \frac{\left\vert \widetilde{\mu}_n \right\vert}{n} \left( \frac{\Vert \ell \Vert_{L^\infty}}{(n + k)^2} + \frac{\Vert \ell \Vert_{L^\infty}}{(n - k)^2} \right) 
            \right\} + \frac{1}{\omega^4} \\
            \lesssim \ & \frac{\Vert \ell \Vert_{L^\infty}}{\omega^2} \sum_{n \geq 1} \frac{n^3 \left\vert \widetilde{\mu}_n \right\vert }{n^4 + \omega^2} + \frac{1}{\omega^2} \sum_{n \geq 1} \frac{n^4 \left( \widetilde{\mu}_n \right)^2 }{n^4 + \omega^2}  + \frac{1}{\omega^4}.
        \end{split}
    \end{equation}
    First, the Cauchy-Schwarz inequality gives 
    \begin{equation}
        \sum_{n \geq 1} \frac{n^3 \left\vert \widetilde{\mu}_n \right\vert }{n^4 + \omega^2}
        \lesssim \left\Vert \widetilde{\mu} \right\Vert_{H_{(0)}^s} \left( \sum_{n \geq 1} \frac{n^{6-2s}}{\left( n^4 + \omega^2 \right)^2} \right)^{\frac{1}{2}} ,
    \end{equation}
    yielding
    \begin{equation}
        \sum_{n \geq 1} \frac{n^3 \left\vert \widetilde{\mu}_n \right\vert }{n^4 + \omega^2} = \mathcal{O}\left( \frac{1}{\omega^{\frac{s}{2} + \frac{1}{4}}} \right),
    \end{equation}
    by \eqref{eq:lem:asymptotic_sum_appendix_2}, applied with $\beta = 6 - 2s \in [2, 5)$. Second, one has
    \begin{equation}
        \sum_{n \geq 1} \frac{n^4 \left( \widetilde{\mu}_n \right)^2 }{n^4 + \omega^2} 
        \lesssim \frac{1}{\omega^{s}} \sup_{n \geq 1} \left( \frac{ \left(\frac{n}{\sqrt{\omega}}\right)^{4-2s} }{\left(\frac{n}{\sqrt{\omega}}\right)^4 + 1} \right) \left\Vert \widetilde{\mu} \right\Vert_{H_{(0)}^s}^2 \lesssim \frac{1}{\omega^{s}},
    \end{equation}
    since $s \leq 2$. Hence, one obtains 
    \begin{equation}
         K_2(\omega) = \mathcal{O}\left( \frac{1}{\omega^{2 + \frac{s}{2} + \frac{1}{4}}} \right) + \mathcal{O}\left( \frac{1}{\omega^{2 + s}} \right) + \mathcal{O}\left( \frac{1}{\omega^{4}} \right) = \mathcal{O}\left( \frac{1}{\omega^{\frac{5}{2} + \nu}} \right),
    \end{equation}
    since $s > \frac{1}{2}$.
\end{proof}

\subsubsection{Obstructions quantified by other fractional norms of the control}
 
To obtain fractional drifts, we need to consider source profiles $\mu$ with a sharp threshold of regularity, namely such that
\begin{equation}
    \mu \in \mathcal{H}^s_p(0, 1) \quad \text{ and } \quad \mu \notin \bigcup_{s^\prime > s} \mathcal{H}^{s^\prime}_p(0, 1).
\end{equation}
A simple way to construct such a profile is to prescribe its Fourier coefficients. For instance, one may set $\langle \mu, \varphi_n \rangle := \frac{1}{n^\alpha}$ if $p = \infty$, and $\langle \mu, \varphi_n \rangle := \frac{1}{n^\alpha \ln(n) }$ if $p = 2$, and then 
\begin{equation}
    \mu := \sum_{n \geq 1} \langle \mu, \varphi_n \rangle \varphi_n,
\end{equation}
where $\alpha$ is chosen so that the series converges in $\mathcal{H}^s_p(0,1)$. This purely spectral construction can be used to obtain asymptotic estimates for $K_{\mu,k}$ with fractional order $-2-s$, for any $s\in[-1,1]$, with a logarithmic factor in some cases. However, this construction gives no structural explanation of which features of the source profile $\mu$ are responsible for fractional drifts. We therefore replace this artificial prescription of Fourier coefficients by a more intrinsic construction, based on source profiles with a singularity of the form $x^\alpha$ in the physical space. When $\alpha<-1$, these profiles are no longer locally integrable near the singularity, and their Fourier coefficients have to be understood in the sense of finite-part distributions (see, for instance, \cite[Section~3.2]{HormanderI}).  The assumptions $x_0 = 1/k$ and $k \geq 2$ are made only to simplify the computations. The argument can be adapted to the case $k \geq 1$, and there are other admissible choices of $x_0$ for which \eqref{eq:prop:asymptotic_estimates_fractional_drifts} remains valid.
 
\begin{proposition}[Sufficient condition for the fractional drifts in small time]\label{prop:asymptotic_estimates_fractional_drifts}
    Let $s \in [-1, 1]$, with $s \notin \left\{- \frac{1}{2}, 0 \right\}$, $k \geq 2$, $x_0 := \frac{1}{k} \in (0, 1)$ and $\chi \in C^\infty([0, 1], \mathbb{R})$ be such that $\chi = 1$ in a neighborhood of $x_0$ and $\chi = 0$ in a neighborhood of $1$. For $x \in [0,1]$, set
    \begin{equation}
        \mathcal{M}(x) := 
        \begin{cases}
             0 & \text{ if } x \in [0, x_0] \\
             (x - x_0)^{s - \frac{1}{2}} \chi(x) & \text{ if } x \in (x_0, 1] 
        \end{cases} .
    \end{equation}
    If $s > - \frac{1}{2}$, then $\mathcal{M} \in L^1(0,1)$, and we naturally define $\langle \mathcal{M}, \varphi_n \rangle := \int_0^1 \mathcal{M} \varphi_n \dd x$. If $s < -\frac{1}{2}$, we then define $\langle \mathcal{M}, \varphi \rangle$ for $\varphi \in C^\infty([0, 1], \mathbb{R})$ by setting
    \begin{equation}\label{eq:def:mu_distribution_finite_part_2}
        \langle \mathcal{M}, \varphi \rangle := \lim_{\varepsilon \rightarrow 0^+} \left( \int_{x_0 + \varepsilon}^1 (x - x_0)^{s - \frac{1}{2}} \chi(x) \varphi(x) \dd x + \frac{\varphi(x_0) \varepsilon^{s + \frac{1}{2}}}{s + \frac{1}{2}} \right) .
    \end{equation}
    Then $\mathcal{M} \in \mathcal{H}^{s+\frac{1}{2}}_\infty(0,1)$. 
    There exists $\psi \in C_\mathrm{c}^\infty((0, 1), \mathbb{R})$ such that, setting $\mu := \mathcal{M} + \psi$, one has $\langle \mu, \varphi_k \rangle = 0$ if $s < 0$, and one has $\langle \mu, \varphi_k \rangle = \langle \mu^2, \varphi_k^\prime \rangle = 0$ if $s > 0$. Then
    \begin{equation}\label{eq:prop:asymptotic_estimates_fractional_drifts}
        K_{\mu,k}(\omega) = \frac{\alpha_s}{\omega^{2 + s}} + \mathcal{O}\left( \frac{1}{\omega^{2+s+\nu}} \right)
    \end{equation}
    when $\omega \rightarrow + \infty$, for some $\alpha_s \neq 0$, and some $\nu > 0$.
\end{proposition}

\begin{remark}
    Directly from \eqref{eq:def:mu_distribution_finite_part_1} and \eqref{eq:def:mu_distribution_finite_part_2}, one finds the equivalent definition
    \begin{equation}\label{eq:def:mu_distribution_finite_part_4}
        \langle \mathcal{M}, \varphi \rangle = \int_{x_0}^1 (x - x_0)^{s - \frac{1}{2}} \left( \chi(x) \varphi(x) - \varphi(x_0) \right) \dd x 
        + \frac{\varphi(x_0) (1-x_0)^{s + \frac{1}{2}}}{s + \frac{1}{2}}  \quad \text{ if $s \in \left[-1, -\frac{1}{2} \right)$.}
    \end{equation}
\end{remark}

\begin{remark}\label{rem:after:prop:asymptotic_estimates_fractional_drifts}
    If $s=-\frac{1}{2}$, then one can also define $\mathcal{M}$ naturally by
    \begin{equation}\label{eq:def:mu_distribution_finite_part_1}
            \langle \mathcal{M}, \varphi \rangle := \lim_{\varepsilon \rightarrow 0^+} \left( \int_{x_0 + \varepsilon}^1 (x - x_0)^{s - \frac{1}{2}} \chi(x) \varphi(x) \dd x + \varphi(x_0) \ln(\varepsilon) \right).
    \end{equation}
    However, the corresponding Fourier coefficients contain a logarithmic factor, and the same is true for the kernel. One could introduce a modified version of \eqref{eq:def:mu_distribution_finite_part_1} leading to the asymptotic behavior $\omega^{-(2+s)}$ when $s=-\frac{1}{2}$, but for simplicity we choose to exclude this case.
\end{remark}
    
\begin{proof}
    \textbf{Step 1: a general identity.}
    We prove that for $\alpha > -1$ and $N > \alpha + 1$, there exists $C_N > 0$ such that for all $n$ sufficiently large, one has
    \begin{equation}\label{eq:proof:prop:asymptotic_estimates_fractional_drifts_step_1_1}
        \left\vert \int_{x_0}^1 (x-x_0)^\alpha \chi(x) e^{i n \pi x} \dd x - \frac{\Gamma(\alpha + 1)}{(\pi n)^{\alpha + 1}} e^{ i\pi n x_0} e^{i \pi \frac{(\alpha + 1)}{2} } \right\vert \leq \frac{C_N}{n^N},
    \end{equation}
    where $\Gamma(z):= \int_0^\infty t^{z-1}e^{-t} \dd t$ is the standard $\Gamma$ function. 
    To ensure convergence of oscillatory integrals, write
    \begin{equation}\label{eq:proof:prop:asymptotic_estimates_fractional_drifts_step_1_1_bis}
        \begin{split}
            \int_{x_0}^1 (x-x_0)^\alpha \chi(x) e^{i\pi n x} \dd x 
            & = \frac{e^{i n \pi x_0}}{(n\pi)^{\alpha+1}} \int_0^{n\pi(1-x_0)} x^\alpha \chi\left( x_0 + \frac{x}{n\pi} \right) e^{i x } \dd x \\
            & = \frac{e^{i n \pi x_0}}{(n\pi)^{\alpha+1}} \lim_{\varepsilon \to 0^+} \int_0^{n\pi(1-x_0)} x^\alpha \chi\left( x_0 + \frac{x}{n\pi} \right) e^{i x - \varepsilon x } \dd x.
        \end{split}
    \end{equation}
    Let $\varepsilon > 0$. One has
    \begin{equation}
        \begin{split}
            & \int_0^{n\pi(1-x_0)} x^\alpha \chi\left( x_0 + \frac{x}{n\pi} \right) e^{i x - \varepsilon x } \dd x \\
            = \ & \int_0^{\infty} x^\alpha e^{i x - \varepsilon x } \dd x + \int_0^{\infty} x^\alpha \left( \chi\left( x_0 + \frac{x}{n\pi} \right) - 1 \right) e^{i x - \varepsilon x } \dd x
            =: A_\varepsilon + B_\varepsilon,
        \end{split}
    \end{equation}
    where $\chi(x) := 0$ for $x\geq 1$. Classical contour integration gives
    \begin{equation}
        \int_0^{\infty} x^\alpha e^{i x - \varepsilon x } \dd x 
        = \frac{1}{(\varepsilon - i)^{\alpha +1}} \int_0^{\infty} x^\alpha e^{-x} \dd x ,
    \end{equation}
    implying 
    \begin{equation}
        \lim_{\varepsilon \to 0^+} A_\varepsilon = e^{\frac{i\pi(\alpha+1)}{2}} \Gamma(\alpha + 1).
    \end{equation}
    We prove below that for all $N > \alpha + 1$, there exists $C_N > 0$ independent of $\varepsilon$ and $n$ such that 
    \begin{equation}\label{eq:proof:prop:asymptotic_estimates_fractional_drifts_step_1_2}
        \left\vert B_\varepsilon \right\vert \leq \frac{C}{n^{N-\alpha-1}}.
    \end{equation}
    Hence, letting $\varepsilon \rightarrow 0^+$ in the estimate
    \begin{equation}
        \left\vert \int_0^{n\pi(1-x_0)} x^\alpha \chi\left( x_0 + \frac{x}{n\pi} \right) e^{i x - \varepsilon x } \dd x - A_\varepsilon \right\vert \leq \vert B_\varepsilon \vert \leq \frac{C}{n^{N-\alpha-1}},
    \end{equation}
    one finds \eqref{eq:proof:prop:asymptotic_estimates_fractional_drifts_step_1_1} (remember the $\frac{1}{n^{\alpha + 1}}$ factor from \eqref{eq:proof:prop:asymptotic_estimates_fractional_drifts_step_1_1_bis}).
    
    It remains to prove \eqref{eq:proof:prop:asymptotic_estimates_fractional_drifts_step_1_2}. Note that the function
    \begin{equation}
        \eta: x \in (0, +\infty) \mapsto x^\alpha \left( \chi\left( x_0 + \frac{x}{n\pi} \right) - 1 \right) 
    \end{equation}
    is smooth and equals $0$ on $(0, n \pi \delta)$, for some $\delta > 0$. Performing $N$ integrations by parts thus gives
    \begin{equation}
        \left\vert B_\varepsilon \right\vert \leq \frac{1}{\vert \varepsilon - i \vert^{N}}  \int_0^\infty \left\vert \eta^{(N)}(x) \right\vert \dd x \leq \int_0^\infty \left\vert \eta^{(N)}(x) \right\vert \dd x.
    \end{equation}
    If $\alpha - N < -1$, then 
    \begin{equation}
        \begin{split}
            & \int_0^\infty \left\vert \eta^{(N)}(x) \right\vert \dd x \\
            \leq \ & C \int_{n \pi \delta}^\infty \left\vert \chi\left( x_0 + \frac{x}{n\pi} \right) - 1 \right\vert x^{\alpha - N} \dd x  + C \sum_{p = 1}^N \frac{1}{n^p} \int_{n \pi \delta}^\infty \left\vert \chi^{(p)}\left( x_0 + \frac{x}{n\pi} \right) \right\vert x^{\alpha - N + p} \dd x \\
            = \ & C n^{\alpha - N + 1} \int_{\pi \delta}^\infty \left\vert \chi\left( x_0 + \frac{x}{\pi} \right) - 1 \right\vert x^{\alpha - N} \dd x + C n^{\alpha - N + 1} \sum_{p = 1}^N \int_{\pi \delta}^\infty \left\vert \chi^{(p)}\left( x_0 + \frac{x}{\pi} \right) \right\vert x^{\alpha - N + p} \dd x \\
            \leq \ & C^\prime n^{\alpha - N + 1},
        \end{split}
    \end{equation}
    for some $C, C^\prime > 0$ independent of $\varepsilon$ and $n$. This gives \eqref{eq:proof:prop:asymptotic_estimates_fractional_drifts_step_1_2}.
    
    \textbf{Step 2: computation of the Fourier coefficients of $\mathcal{M}$.} 
    Here, all $\mathcal{O}$ estimates are understood as $n \rightarrow +\infty$. If $s > - \frac{1}{2}$, then Step 1 gives
    \begin{equation}\label{eq:proof:prop:asymptotic_estimates_fractional_drifts_step_2_1}
        \langle \mathcal{M}, \varphi_n \rangle = \frac{\sqrt{2} \Gamma\left( s + \frac{1}{2} \right)}{(\pi n)^{s + \frac{1}{2}}} \sin\left( \pi n x_0 + \pi \left( \frac{s}{2} + \frac{1}{4} \right) \right) + \mathcal{O}\left(\frac{1}{n^N}\right).
    \end{equation}
    Now, assume that $s \in \left[-1, -\frac{1}{2} \right)$. Integration by parts in \eqref{eq:def:mu_distribution_finite_part_4} gives
    \begin{equation}
        \langle \mathcal{M}, \varphi_n \rangle = - \frac{1}{s + \frac{1}{2}} \int_{x_0}^1 (x - x_0)^{s + \frac{1}{2}} \left( \chi^\prime(x) \varphi_n(x) + \chi(x) \varphi_n^\prime(x) \right) \dd x .
    \end{equation}
    Step 1 gives 
    \begin{equation}
        \begin{split}
            \int_{x_0}^1 (x - x_0)^{s + \frac{1}{2}}  \chi(x) \varphi_n^\prime(x) \dd x 
            & = \sqrt{2} n \pi \Re \int_{x_0}^1 (x - x_0)^{s + \frac{1}{2}} \chi(x) e^{i \pi n x} \dd x \\
            & = \sqrt{2} n \pi \frac{\Gamma(s + \frac{3}{2})}{(\pi n)^{s + \frac{3}{2}}} \cos \left( \pi n x_0 + \frac{\pi \left( s + \frac{3}{2} \right)}{2} \right) + \mathcal{O}\left( \frac{1}{n^N} \right).
        \end{split}
    \end{equation}
    Using $\cos \left( \pi n x_0 + \frac{\pi \left( s + \frac{3}{2} \right)}{2} \right) = - \sin \left( \pi n x_0 + \frac{\pi \left( s + \frac{1}{2} \right)}{2} \right)$ and $\Gamma(z+1) = z \Gamma(z)$, one obtains that 
    \begin{equation}
        - \frac{1}{s + \frac{1}{2}} \int_{x_0}^1 (x - x_0)^{s + \frac{1}{2}} \chi(x) \varphi_n^\prime(x) \dd x 
        = \frac{\sqrt{2} \Gamma\left( s + \frac{1}{2} \right)}{(\pi n)^{s + \frac{1}{2}}} \sin\left( \pi n x_0 + \pi \left( \frac{s}{2} + \frac{1}{4} \right) \right) + \mathcal{O}\left(\frac{1}{n^N}\right).
    \end{equation}
    Since $\chi^\prime = 0$ in a neighbourhood of $x_0$, performing $N$ integrations by parts gives
    \begin{equation}
        - \frac{1}{s + \frac{1}{2}} \int_{x_0}^1 (x - x_0)^{s + \frac{1}{2}} \chi^\prime(x) \varphi_n(x) \dd x = \mathcal{O}\left(\frac{1}{n^N}\right),
    \end{equation}
    implying that \eqref{eq:proof:prop:asymptotic_estimates_fractional_drifts_step_2_1} also holds when $s \in \left[-1, -\frac{1}{2} \right)$. 

    \textbf{Step 3: proof of the existence of $\psi$ satisfying the desired properties.}
    Assume that $s < 0$, and let $\psi \in C^\infty_{\mathrm{c}}((0,1); \mathbb{R})$ be such that $\langle \psi, \varphi_k \rangle \neq 0$. Then
    \begin{equation}
        \mu := \mathcal{M} - \frac{\langle \mathcal{M}, \varphi_k \rangle}{\langle \psi, \varphi_k \rangle } \psi
    \end{equation}
    satisfies $\langle \mu, \varphi_k \rangle = 0$, which is the only required property. If $s > 0$, we use the following general result, whose proof is given below.
    
    \begin{lemma}\label{lem:correction_moment_and_first_drift}
        Let $\mu \in L^2((0, 1); \mathbb{R})$ and $k \geq 1$. There exists $\psi \in C^\infty_{\mathrm{c}}((0,1); \mathbb{R})$ such that $\langle \mu+\psi, \varphi_k \rangle = \left\langle \left( \mu + \psi \right)^2, \varphi_k^\prime \right\rangle = 0$.
    \end{lemma}
    
    \textbf{Step 4: proof of \eqref{eq:prop:asymptotic_estimates_fractional_drifts}.} 
    Here, all $\mathcal{O}$ estimates are understood as $\omega \rightarrow +\infty$. Let $\psi \in C_\mathrm{c}^\infty((0, 1), \mathbb{R})$, and set $\mu := \mathcal{M} + \psi$. For all $N \geq 1$, there exists $C_N > 0$ such that for all $n$ sufficiently large,
    \begin{equation}\label{eq:proof:prop:asymptotic_estimates_fractional_drifts_step_3_1}
        \left\vert \langle \psi, \varphi_n \rangle \right\vert \leq \frac{C_N}{n^N}.
    \end{equation}
    Set $A_s := \frac{\sqrt{2} \Gamma\left( s + \frac{1}{2} \right)}{\pi^{s + \frac{1}{2}}} \neq 0$, $\theta = \pi x_0 = \frac{\pi}{k}$, and $\rho_s := \pi \left( \frac{s}{2} + \frac{1}{4} \right)$. 
    
    \underline{Proof of \eqref{eq:prop:asymptotic_estimates_fractional_drifts} in the case $s < 0$.} Assume $s < 0$. Using \eqref{eq:proof:prop:asymptotic_estimates_fractional_drifts_step_3_1}, one finds 
    \begin{equation}
        K_{\mu, k}(\omega) = \sum_{n \geq 1} \frac{\langle \mathcal{M}, \varphi_n \rangle}{n}  \left( \frac{\langle \mathcal{M}, \varphi_{n+k} \rangle}{n+k} + \frac{\langle \mathcal{M}, \varphi_{\vert n-k \vert} \rangle}{\vert n-k \vert} \mathbbm{1}_{n \neq k} \right) \frac{\lambda_{n,k}}{ \lambda_{n,k}^2 + \omega^2} + \mathcal{O}\left( \frac{1}{\omega^2} \right).
    \end{equation}
    Using \eqref{eq:proof:prop:asymptotic_estimates_fractional_drifts_step_2_1}, one finds
    \begin{equation}
        \begin{split}
            K_{\mu, k}(\omega) = \ & A_s^2 \sum_{n \geq 1} \frac{\sin(\theta n + \rho_s)}{n^{s + \frac{3}{2}}}  \left( \frac{\sin\left(\theta(n+k) + \rho_s\right)}{(n+k)^{s + \frac{3}{2}}} + \frac{\sin\left(\theta \vert n - k \vert + \rho_s\right)}{\vert n-k \vert^{s + \frac{3}{2}}} \mathbbm{1}_{n \neq k} \right) \frac{\lambda_{n,k}}{ \lambda_{n,k}^2 + \omega^2} \\
            & + \mathcal{O}\left( \frac{1}{\omega^2} \right).
        \end{split}
    \end{equation}
    Since 
    \begin{equation}
        \sum_{n = 1}^{k-1} \frac{\sin(\theta n + \rho_s)}{n^{s + \frac{3}{2}}} \frac{\sin\left(\theta \vert n - k \vert + \rho_s\right)}{\vert n-k \vert^{s + \frac{3}{2}}} \frac{\lambda_{n,k}}{ \lambda_{n,k}^2 + \omega^2} = \mathcal{O}\left( \frac{1}{\omega^2} \right),
    \end{equation}
    one has
    \begin{equation}
        K_{\mu, k}(\omega) = A_s^2 \sum_{n \geq 1} \frac{\sin(\theta n + \rho_s) \sin\left(\theta(n+k) + \rho_s\right)}{n^{s + \frac{3}{2}} (n+k)^{s + \frac{3}{2}}}  \left( \frac{\lambda_{n,k}}{ \lambda_{n,k}^2 + \omega^2} + \frac{\lambda_{n+k,k}}{ \lambda_{n+k,k}^2 + \omega^2} \right) + \mathcal{O}\left( \frac{1}{\omega^2} \right).
    \end{equation} 
    For $n > k$, using 
    \begin{equation}
        \left\vert \left( 1 + \frac{k}{n} \right)^{s + \frac{1}{2}} - 1 \right\vert \leq \frac{C}{n},
    \end{equation}
    for some $C>0$ independent of $n$, and arguing as in the proof of Proposition \ref{prop:asymptotic_estimate_kernel_affine+rest}, one finds
    \begin{equation}
        \left\vert \frac{1}{n^{s + \frac{3}{2}} (n+k)^{s + \frac{3}{2}}}  \left( \frac{\lambda_{n,k}}{ \lambda_{n,k}^2 + \omega^2} + \frac{\lambda_{n+k,k}}{ \lambda_{n+k,k}^2 + \omega^2} \right) -  \frac{2 \pi^2 }{ n^{2s + 1} \left( \pi^4 n^4 + \omega^2 \right)} \right\vert
        \leq \frac{C^\prime}{ n^{2s + 2} \left( n^4 + \omega^2 \right)},
    \end{equation}
    for some $C^\prime > 0$ independent of $n$. By \eqref{eq:lem:asymptotic_sum_appendix_1} with $\alpha = - 2s - 1 \in (-1, 1]$, this gives
    \begin{equation}
        K_{\mu, k}(\omega) = 2 \pi^2 A_s^2 \sum_{n \geq 1} \frac{\sin(\theta n + \rho_s) \sin\left(\theta(n+k) + \rho_s\right)}{ n^{2s + 1} \left( \pi^4 n^4 + \omega^2 \right)} + \mathcal{O}\left( \frac{1}{\omega^{\frac{5}{2}+s}} + \mathcal{O}\left( \frac{1}{\omega^2} \right) \right).
    \end{equation} 
    Since $x_0 = \frac{1}{k}$, one has $\sin(\theta n + \rho_s) \sin\left(\theta(n+k) + \rho_s\right) = - \sin(\theta n + \rho_s)^2$. Hence, \eqref{eq:lem:asymptotic_sum_appendix_5} yields
    \begin{equation}
        K_{\mu, k}(\omega) = - \frac{\pi^2 A_s^2}{\omega^{s+2}} \int_0^\infty \frac{1}{ t^{2s + 1} \left( \pi^4 t^4 + 1 \right)} \dd t + \mathcal{O}\left( \frac{1}{\omega^{2+s+\nu}} \right),
    \end{equation} 
    for some $\nu > 0$, which is \eqref{eq:prop:asymptotic_estimates_fractional_drifts}.
    
    \underline{Proof of \eqref{eq:prop:asymptotic_estimates_fractional_drifts} in the case $s > 0$.} Assume that $s > 0$, and that $\langle \mu^2, \varphi_k^\prime \rangle = 0$. As in the proof of Proposition \ref{prop:asymptotic_estimate_kernel_affine+rest}, it implies
    \begin{equation}
        K_{\mu, k}(\omega) = - \frac{1}{\omega^2} \sum_{n \geq 1} \frac{\langle \mu, \varphi_n \rangle}{n}  \left( \frac{\langle \mu, \varphi_{n+k} \rangle}{n+k} + \frac{\langle \mu, \varphi_{\vert n-k \vert} \rangle}{\vert n-k \vert} \mathbbm{1}_{n \neq k} \right) \frac{\lambda_{n,k}^3}{ \lambda_{n,k}^2 + \omega^2},
    \end{equation}
    yielding
    \begin{equation}
        K_{\mu, k}(\omega) = - \frac{1}{\omega^2} \sum_{n \geq 1} \frac{\langle \mathcal{M}, \varphi_n \rangle}{n}  \left( \frac{\langle \mathcal{M}, \varphi_{n+k} \rangle}{n+k} + \frac{\langle \mathcal{M}, \varphi_{\vert n-k \vert} \rangle}{\vert n-k \vert} \mathbbm{1}_{n \neq k} \right) \frac{\lambda_{n,k}^3}{ \lambda_{n,k}^2 + \omega^2} + \mathcal{O}\left( \frac{1}{\omega^4} \right),
    \end{equation}
    by \eqref{eq:proof:prop:asymptotic_estimates_fractional_drifts_step_3_1}. Arguing as above, this gives
    \begin{equation}
        K_{\mu, k}(\omega) = \frac{2 \pi^6 A_s^2}{\omega^2} \sum_{n \geq 1} \frac{\sin(\theta n + \rho_s)^2}{ n^{2s - 3} \left( \pi^4 n^4 + \omega^2 \right)} + \mathcal{O}\left( \frac{1}{\omega^{2+s+\nu}} \right),
    \end{equation} 
    for some $\nu > 0$, where we have used
    \begin{equation}
        \frac{1}{\omega^2} \sum_{n \geq 1} \frac{1}{ n^{2s-2} \left( n^4 + \omega^2 \right) } = \mathcal{O}\left( \frac{1}{\omega^{2+s+\nu}} \right),
    \end{equation}
    which is \eqref{eq:lem:asymptotic_sum_appendix_1} applied with $\alpha = - 2s + 2 \in [0, 2)$. Using \eqref{eq:lem:asymptotic_sum_appendix_5}, applied with $s^\prime = s - 2 \in (-2, -1]$, one obtains \eqref{eq:prop:asymptotic_estimates_fractional_drifts}. This completes the proof of Proposition \ref{prop:asymptotic_estimates_fractional_drifts}.
\end{proof}

It remains to prove Lemma \ref{lem:correction_moment_and_first_drift}.

\begin{proof}[Proof of Lemma \ref{lem:correction_moment_and_first_drift}.]
    Let $\psi_0 \in C^\infty_{\mathrm{c}}((0,1); \mathbb{R})$ be such that $\langle \psi_0, \varphi_k \rangle \neq 0$, and set
    \begin{equation}
        \mu_0 := \mu - \frac{\langle \mu, \varphi_k \rangle}{\langle \psi_0, \varphi_k \rangle } \psi_0.
    \end{equation}
    We prove below the existence of $\psi_-, \psi_+ \in C^\infty_{\mathrm{c}}((0,1); \mathbb{R})$ such that $\langle \psi_\pm, \varphi_k \rangle = 0$, $\left\langle (\psi_-)^2, \varphi_k^\prime \right\rangle < 0$ and $\left\langle (\psi_+)^2, \varphi_k^\prime \right\rangle > 0$. This is sufficient to complete the proof. Indeed, the intermediate value theorem then implies the existence of $\lambda \geq 0$ such that $\left\langle \left( \mu_0 + \lambda \psi_\pm \right)^2, \varphi_k^\prime \right\rangle = 0$, where the sign $\pm$ is chosen according to the sign of $\left\langle \mu_0^2, \varphi_k^\prime \right\rangle$.  
    
    We construct $\psi_-$ and $\psi_+$ supported in $\left(0, \frac{1}{k}\right)$, so that we can in fact assume $k = 1$ without loss of generality. Once the function $\psi_+$ is constructed, the function $\psi_-$ can be defined by $\psi_-(x):=\psi_+(1-x)$. To construct $\psi_+$, the idea is to choose $\psi_+$ as a linear combination of two small cutoff functions, one localized on $(0, \delta)$, and the other localized in $\left(\frac{1}{2}, \frac{1}{2}+\delta \right)$, as the following picture suggests.
    
    
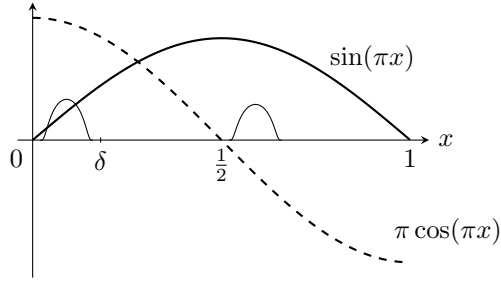
\begin{center}
        \begin{tikzpicture}[x=5cm,y=1.35cm,>=stealth]
            \draw[->] (-0.04,0) -- (1.05,0) node[right] {$x$};
            \draw[->] (0,-1.35) -- (0,1.35) ;

            \draw (0,0) node[below left] {$0$};
            \draw (1,0) node[below] {$1$};
            \draw (0.5,0) node[below] {$\frac{1}{2}$};
            \draw (0.18,0.025) -- (0.18,-0.025) node[below] {$\delta$};
        
            \draw[thick,domain=0:1,samples=200] plot (\x,{sin(180*\x)});
            \draw[thick,dashed,domain=0:1,samples=200] plot (\x,{1.2*cos(180*\x)});
            \draw (0.9,0.8) node {$\sin(\pi x)$};
            \draw (1.1,-0.9) node {$\pi \cos(\pi x)$};
        
            \draw[thin,domain=0.02:0.16,samples=100] plot (\x,{0.4*exp(-1/(1-((\x-0.09)/0.07)^2))/exp(-1)}); 
            \draw[thin,domain=0.52:0.66,samples=100] plot (\x,{0.35*exp(-1/(1-((\x-0.59)/0.07)^2))/exp(-1)});  
        \end{tikzpicture}
    \end{center}
    
    Let $\delta \in \left(0, \frac{1}{2} \right)$, $\rho \in C^\infty_{\mathrm{c}}((0,1); \mathbb{R}_+)$, and set 
    \begin{equation}
        \psi_+ := \rho_1 - \frac{\langle \rho_1, \varphi_1 \rangle}{\langle \rho_2, \varphi_1 \rangle} \rho_2 \quad \text{ with } \rho_1(x) := \rho \left( \frac{x}{\delta} \right) \text{ and } \rho_2(x) := \rho \left( \frac{x}{\delta} - \frac{1}{2 \delta} \right).
    \end{equation}
    Then, by definition, $\langle \psi_+, \varphi_1 \rangle = 0$. We prove that if $\delta$ is sufficiently small, then $\left\langle (\psi_+)^2, \varphi_1^\prime \right\rangle > 0$. Straightforward computations give
    \begin{equation}
        \begin{split}
            & \langle \rho_1, \varphi_1 \rangle = c_1 \delta^2 + o(\delta^2), \quad \langle \rho_2, \varphi_1 \rangle = c_2 \delta + o(\delta), \\
            & \langle (\rho_1)^2, \varphi_1^\prime \rangle = c_3 \delta + o(\delta), \quad \langle (\rho_2)^2, \varphi_1^\prime \rangle = - c_4 \delta^2 + o(\delta^2),
        \end{split}
    \end{equation}    
    for some $c_1, c_2, c_3, c_4 > 0$. Since the support of $\rho_1$ and $\rho_2$ are disjoint, this gives
    \begin{equation}
        \left\langle (\psi_+)^2, \varphi_1^\prime \right\rangle = \left\langle (\rho_1)^2, \varphi_1^\prime \right\rangle + \frac{\langle \rho_1, \varphi_1 \rangle^2}{\langle \rho_2, \varphi_1 \rangle^2} \left\langle (\rho_2)^2, \varphi_1^\prime \right\rangle = c_3 \delta + o(\delta).
    \end{equation}
    This completes the proof.
\end{proof}

\subsection{Kernels with global sign properties}

In Proposition \ref{prop:global_sign_AND_asymptotic_estimates} below, we prove the existence of $\mu$ for which the kernel $K_{\mu, k}$ satisfies both an asymptotic estimate and a global sign property. 

\begin{remark}\label{rq:kernel_decays_faster_than_expected}
    Note that the regularity assumption $\mu \in \mathcal{H}_\infty^{s+\frac{1}{2}}(0,1)$ typically gives rise to asymptotic behaviors of the form $K_{\mu,k} \approx \frac{1}{\omega^{2+s}}$. Thus, the weaker assumption $\mu \in \mathcal{H}_{\infty}^{s}(0,1)$, with $s \in (0,1]$, appearing in Proposition \ref{prop:global_sign_AND_asymptotic_estimates} \textit{(i)} below may at first seem surprising. It is, however, natural: heuristically, a profile $\mu$ with low regularity, and hence slowly decaying Fourier coefficients, can nevertheless produce a kernel that decays rapidly at infinity because of cancellations, typically arising from alternating signs.
\end{remark}

\begin{proposition}\label{prop:global_sign_AND_asymptotic_estimates}
    Let $s \in [-1, 1]$, $\varepsilon_{\sgn} \in \{-1, 1\}$ and $k \geq 1$. There exists $\mu$ satisfying $\langle \mu, \varphi_k \rangle = 0$ and the following properties.
    \begin{enumerate}[label=(\roman*)]
        \item One has $\mu \in \mathcal{H}_\infty^{s+\frac{1}{2}}(0,1)$ if $s \in [-1, 0)$; $\mu$ is a linear combination of two eigenfunctions $\varphi_n$ if $s = 0$; and $\mu \in \mathcal{H}_{\infty}^{s}(0,1)$ if $s \in (0, 1]$.
        \item There exists $\alpha > 0$ and $\nu > 0$ such that 
            \begin{equation}\label{eq:prop:global_sign_AND_asymptotic_estimates}
                K_{\mu,k}(\omega) = \varepsilon_{\sgn} \frac{\alpha}{\langle \omega \rangle^{2+s}} + \mathcal{O}\left( \frac{1}{\langle \omega \rangle^{2+s+\nu}} \right), \quad \text{when $\omega \rightarrow + \infty$.}
            \end{equation}
        \item For all $\omega \in \mathbb{R}$, $\varepsilon_{\sgn} K_{\mu,k}(\omega) > 0$.
    \end{enumerate}
\end{proposition} 
    
\begin{proof}
    In this proof, we construct $\mu$ by prescribing its Fourier coefficients $\mu_n := \langle \mu, \varphi_n \rangle$. In all constructions below, $\mu_1 = \cdots = \mu_k = 0$, implying
    \begin{equation}
        K_{\mu, k}(\omega) = \sum_{n \geq 1} \frac{\mu_n \mu_{n+k}}{n(n+k)}  \left( \frac{\lambda_{n,k}}{ \lambda_{n,k}^2 + \omega^2} + \frac{\lambda_{n+k,k}}{ \lambda_{n+k,k}^2 + \omega^2} \right).
    \end{equation}
    Note also that $n > k$ implies $\lambda_{n,k} > 0$.
    
    \textbf{Step 1: case $s=0$.}
    Set $\mu_{k+1} = \varepsilon_{\sgn}$ and $\mu_{2k+1} = 1$, and set $\mu_n=0$ for all other indices $n$. Then
    \begin{equation} 
        K_{\mu, k}(\omega) = \frac{\varepsilon_{\sgn}}{(k+1)(2k+1)}  \left( \frac{\lambda_{k+1,k}}{ \lambda_{k+1,k}^2 + \omega^2} + \frac{\lambda_{2k+1,k}}{ \lambda_{2k+1,k}^2 + \omega^2} \right).
    \end{equation}
    Since $\lambda_{2k+1,k} > \lambda_{k+1,k} > 0$, the profile $\mu$ satisfies all the desired properties.

    \textbf{Step 2: case $s \in [-1, 0)$.} Let $N > 2k$, and set $\mu_n = \frac{\varepsilon_{\sgn}}{n^{\frac{1}{2}+s}}$ if $n \in N \mathbb{N}^\ast$, $\mu_n = \frac{1}{n^{\frac{1}{2}+s}}$ if $n \in N \mathbb{N}^\ast + k$, and $\mu_n=0$ for all other indices $n$. Then
    \begin{equation}
        K_{\mu, k}(\omega) = \varepsilon_{\sgn} \sum_{n \geq 1} \frac{1}{(Nn)^{\frac{3}{2}+s} (Nn+k)^{\frac{3}{2}+s}}  \left( \frac{\lambda_{Nn,k}}{ \lambda_{Nn,k}^2 + \omega^2} + \frac{\lambda_{Nn+k,k}}{ \lambda_{Nn+k,k}^2 + \omega^2} \right).
    \end{equation}
    Arguing as in the proof of Proposition \ref{prop:asymptotic_estimates_fractional_drifts}, one finds $K_{\mu, k}(\omega) = A(\omega)+B(\omega)$, with
    \begin{equation}
        A(\omega) = C_N \varepsilon_{\sgn} \sum_{n \geq 1} \frac{1}{ n^{1+2s} \left( \pi^4 n^4 + \omega^2\right) },
        \quad \text{ and } \quad
        \left\vert B(\omega) \right\vert \leq C_N^\prime \sum_{n \geq 1} \frac{1}{ n^{2+2s} \left( n^4 + \omega^2\right) },
    \end{equation}
    for some $C, C^\prime > 0$ depending only on $N$, $k$, and $s$. By \eqref{eq:lem:asymptotic_sum_appendix_1}, applied with $\alpha = -1-2s \in (-1, 1]$, one has
    \begin{equation}
        A(\omega) = \frac{c \varepsilon_{\sgn}}{\omega^{2+s}} + \mathcal{O}\left( \frac{1}{\omega^{2+s+\nu}} \right),
    \end{equation}
    for some $c, \nu > 0$. If $s \in \left[-1, -\frac{1}{2} \right)$ then \eqref{eq:lem:asymptotic_sum_appendix_1}, applied with $\alpha = -2-2s \in (-1, 0]$, yields
    \begin{equation}
        B(\omega) = \mathcal{O}\left( \frac{1}{\omega^{\frac{5}{2}+s}} \right).
    \end{equation}
    If $s \in \left( -\frac{1}{2}, 0 \right)$, then one has
    \begin{equation}
        \left\vert B(\omega) \right\vert \leq \frac{C_N^\prime}{\omega^2} \sum_{n \geq 1} \frac{1}{ n^{2+2s} }.
    \end{equation}
    Finally, if $s = - \frac{1}{2}$, a straightforward estimate gives $B(\omega) = \mathcal{O}\left( \frac{1}{\omega^{2-\varepsilon}} \right)$, for all $\varepsilon > 0$. Hence, in any case, the profile $\mu$ satisfies all the desired properties.

    \textbf{Step 3: case $s \in (0,1]$.} 
    The construction is slightly more involved, since we must ensure that $\langle \mu^2, \varphi_k^\prime \rangle = 0$, in order to remove the $\frac{1}{\omega^2}$ term in the asymptotic expansion, while still preserving the global sign property. The idea is as follows. Instead of summing terms which are individually of order $\mathcal{O}\left(\frac{1}{\omega^2}\right)$, as above, we first group the terms in pairs and choose the coefficients so that the $\frac{1}{\omega^2}$ contributions cancel. This reduces the problem to summing terms of order $\mathcal{O}\left(\frac{1}{\omega^4}\right)$. The key point is to perform this cancellation while ensuring that each resulting summand has the desired sign globally. We now give the details. 
    
    \underline{Definition of $\mu$.} 
    Let $N > 2k$. For $n \geq 1$, let $\gamma_n > 0$, which will be chosen explicitly below, to ensure that each summand is of order  $\mathcal{O}\left(\frac{1}{\omega^4}\right)$. For $n \in \mathbb{N}^\ast$, set
    \begin{equation}
        \begin{split}
            & \text{$\mu_{2nN} = \frac{\varepsilon_{\sgn}}{(2nN)^{s}}$  , \quad
            $\mu_{2nN+k} = \frac{1}{(2nN+k)^{s}}$ ,} \\
            & \text{$\mu_{(2n+1)N} = -\frac{\varepsilon_{\sgn} \gamma_n}{\left( (2n+1)N \right)^{s}}$ , \quad
            $\mu_{(2n+1)N+k} = \frac{1}{\left( (2n+1)N + k \right)^{s}}$ , }
        \end{split}
    \end{equation}
    and $\mu_n=0$ for all other indices $n$. Then $K_{\mu, k}(\omega) = \varepsilon_{\sgn} \sum_{n \geq 1} L_n(\omega)$, with
    \begin{equation}
        \begin{split}
            L_n(\omega) := \ & \frac{1}{n^{2+2s} \alpha_n}  \left( \frac{\lambda_{2nN,k}}{ \lambda_{2nN,k}^2 + \omega^2} + \frac{\lambda_{2nN+k,k}}{ \lambda_{2nN+k,k}^2 + \omega^2} \right) \\
            & - \frac{\gamma_n}{n^{2+2s} \beta_n}  \left( \frac{\lambda_{(2n+1)N,k}}{ \lambda_{(2n+1)N,k}^2 + \omega^2} + \frac{\lambda_{(2n+1)N+k,k}}{ \lambda_{(2n+1)N+k,k}^2 + \omega^2} \right),
        \end{split}
    \end{equation}
    \begin{equation}
        \alpha_n := \frac{(2nN)^{1+s} (2nN+k)^{1+s}}{n^{2+2s}}, \quad \text{ and } \quad \beta_n := \frac{((2n+1)N)^{1+s} ((2n+1)N+k)^{1+s}}{n^{2+2s}}.
    \end{equation}
    We define $\gamma_n > 0$ such that 
    \begin{equation}
        \frac{\lambda_{2nN,k} + \lambda_{2nN+k,k}}{\alpha_n} - \gamma_n \frac{\lambda_{(2n+1)N,k} + \lambda_{(2n+1)N+k,k} }{\beta_n}  = 0.
    \end{equation}
    
    \underline{Global sign property.} 
    Let $0 < a_n < b_n < c_n < d_n$ be given by 
    \begin{equation}
        a_n:=\lambda_{2nN,k}, \quad b_n := \lambda_{2nN+k,k}, \quad c_n := \lambda_{(2n+1)N,k} , \quad d_n :=\lambda_{(2n+1)N+k,k}.
    \end{equation}
    Then
    \begin{equation}\label{eq:proof:prop:global_sign_AND_asymptotic_estimates_step3_1}
       \frac{\alpha_n n^{2+2s} L_n(\omega)}{a_n + b_n} =  \frac{1}{ a_n + b_n }  \left( \frac{a_n}{ a_n^2 + \omega^2 } + \frac{b_n}{ b_n^2 + \omega^2 } \right) 
        - \frac{1}{c_n + d_n}  \left( \frac{c_n}{ c_n^2 + \omega^2 } + \frac{d_n}{ d_n^2 + \omega^2 } \right).
    \end{equation}
    The comparison of weighted averages gives
    \begin{equation}
         \frac{1}{ a_n + b_n }  \left( \frac{a_n}{ a_n^2 + \omega^2 } + \frac{b_n}{ b_n^2 + \omega^2 } \right) > \frac{1}{ b_n^2 + \omega^2 }
         > \frac{1}{ c_n^2 + \omega^2 } > \frac{1}{c_n + d_n}  \left( \frac{c_n}{ c_n^2 + \omega^2 } + \frac{d_n}{ d_n^2 + \omega^2 } \right),
    \end{equation}
    implying that $L_n(\omega) > 0$ for all $\omega \in \mathbb{R}$, and thus $\varepsilon_{\sgn} K_{\mu, k}(\omega) > 0$. 
    
    \underline{Asymptotic estimate.} It remains to prove \eqref{eq:prop:global_sign_AND_asymptotic_estimates}. We use the symbol $\lesssim$ for constants which may depend only on $s$, $k$ and $N$. Let $n\geq 1$ and $\omega \in \mathbb{R}$. One has
    \begin{equation}
        \frac{1}{ a_n + b_n }  \left( \frac{a_n}{ a_n^2 + \omega^2 } + \frac{b_n}{ b_n^2 + \omega^2 } \right) = \frac{\omega^2 + a_n b_n}{ (a_n^2 + \omega^2 )( b_n^2 + \omega^2)}.
    \end{equation}
    Write $M(x,y):=\frac{x+y}{2}$. A straightforward calculation gives
    \begin{equation}
        \begin{split}
            \left\vert \frac{\omega^2 + a_n b_n}{ (a_n^2 + \omega^2 )( b_n^2 + \omega^2)} - \frac{1}{M(a_n,b_n)^2 + \omega^2} \right\vert
            & \ = \frac{(b_n - a_n)^2 \left\vert 3 \omega^2 - a_n b_n \right\vert}{ 4 (a_n^2 + \omega^2 )( b_n^2 + \omega^2) \left( M(a_n,b_n)^2 + \omega^2 \right)} \\
            & \ \lesssim \frac{n^2}{(n^4+\omega^2)^2} , 
        \end{split} 
    \end{equation}
    with a similar estimate with $a_n$ and $b_n$ replaced by $c_n$ and $d_n$. In addition, one has
    \begin{equation}
        \left\vert \frac{1}{M(a_n,b_n)^2 + \omega^2} - \frac{1}{M(c_n,d_n)^2 + \omega^2} - \frac{ 32 \pi^4 N^4 n^3 }{\left( (2\pi n N)^4 + \omega^2 \right)^2} \right\vert
        \lesssim \frac{n^2}{(n^4+\omega^2)^2}.
    \end{equation}
    Together with \eqref{eq:proof:prop:global_sign_AND_asymptotic_estimates_step3_1}, it gives
    \begin{equation}
        \left\vert \frac{\alpha_n n^{2+2s} L_n(\omega)}{a_n + b_n} - \frac{ 32 \pi^4 N^4 n^3 }{\left( (2\pi n N)^4 + \omega^2 \right)^2}  \right\vert
        \lesssim \frac{n^2}{(n^4+\omega^2)^2}.
    \end{equation}
    Since
    \begin{equation} 
        \left\vert \frac{a_n + b_n}{\alpha_n} - \frac{2 \pi^2 n^2}{(2N)^{2s}} \right\vert \lesssim n, 
    \end{equation} 
    this yields $L_n(\omega) = A_n(\omega) + B_n(\omega)$, with 
    \begin{equation}
        A_n(\omega) = C \frac{n^{3-2s}}{\left( (2\pi n N)^4 + \omega^2 \right)^2}, \quad \text{ and } \quad \left\vert B_n(\omega) \right\vert \lesssim \frac{n^{2-2s}}{\left( n^4 + \omega^2 \right)^2},
    \end{equation}
    for some $C > 0$ independent of $n$ and $\omega$. By \eqref{eq:lem:asymptotic_sum_appendix_2} applied with $\beta = 3 - 2s \in [1, 3)$, and with $\beta = 2 - 2s \in [0, 2)$, this gives \eqref{eq:prop:global_sign_AND_asymptotic_estimates}.
\end{proof}

\subsection{The case of a constant source profile}

Here, we consider the case $\mu \equiv 1$, previously studied in \cite{Nguyen2025Burgers, Marbach2018}. In that case, one has $\langle \mu, \varphi_n \rangle = 0$ if $n \geq 1$ is even, and $\langle \mu, \varphi_n \rangle = \frac{2 \sqrt{2}}{n \pi}$ if $n \geq 1$ is odd. Hence, if $k$ is odd, then $K_{\mu,k}(\omega)=0$, and if $k$ is even, then
\begin{equation}
    K_{\mu,k}(\omega) = \frac{8}{\pi^2} \sum_{\substack{n \geq 1 \\ n \ \mathrm{odd}}} \frac{1}{n^2} \left( \frac{1}{(n+k)^2} + \frac{1}{(n-k)^2} \right) \frac{\lambda_{n,k}}{ \lambda_{n,k}^2 + \omega^2} .
\end{equation}
We prove the following result. The cases $k = 2$ and $k = 10$ were previously obtained in \cite[Proposition 3.1]{Nguyen2025Burgers}, with a different formula for the kernel.

\begin{proposition}\label{prop:mu_constant_sign_global}
     Let $\mu \equiv 1$ and $k \geq 1$ even. Then $K_{\mu,k}(\omega) < 0$ for all $\omega \in \mathbb{R}$ if and only if $k \in \left( \sqrt{2}, 2 \sqrt{2} \right) + 2 \sqrt{2} \mathbb{Z}$. 
\end{proposition}

\begin{proof}
    By \eqref{eq:prop:asymptotic_estimate_kernel_affine+rest_2}, the kernel $K_{\mu,k}$ is asymptotically negative. The idea of the proof is to show that if $K_{\mu,k}(0) < 0$, then $K_{\mu,k}$ is negative on $\mathbb{R}$, and that the condition $K_{\mu,k}(0) < 0$ is equivalent to $k \in \left( \sqrt{2}, 2 \sqrt{2} \right) + 2 \sqrt{2} \mathbb{Z}$. Equivalently, we study 
    \begin{equation}
        f(\omega) := \sum_{\substack{n \in \mathbb{Z} \\ n \ \mathrm{odd}}} \frac{1}{n^2} \left( \frac{1}{(n+k)^2} + \frac{1}{(n-k)^2} \right) \frac{2 n^2 - k^2}{ (2 n^2 - k^2)^2 + \omega^2},
    \end{equation}
    which satisfies $K_{\mu,k}(\omega) = \frac{8}{\pi^4} f \! \left( \frac{2 \omega}{\pi^2} \right)$.
    We use the decomposition
    \begin{equation}
        \begin{split}
            \frac{1}{n^2} \left( \frac{1}{(n+k)^2} + \frac{1}{(n-k)^2} \right) \frac{2 n^2 - k^2}{ (2 n^2 - k^2)^2 + \omega^2}
            = & - \frac{8 \left( (\omega^2 - 3k^4) (2 n^2 - k^2) + 4 k^2 \omega^2 \right)}{\left(k^{4} + \omega^{2}\right)^{2} \left( (2 n^2 - k^2)^2 + \omega^{2}\right)} \\ & 
            +  \frac{1}{k^{4} + \omega^{2}} \left( \frac{1}{\left(k + n\right)^{2}} + \frac{1}{\left(k - n\right)^{2}} - \frac{2}{n^{2}} \right)  \\ & 
            + \frac{4 \left(3 k^{4} - \omega^{2}\right)}{\left(k^2 - n^2 \right) \left(k^{4} + \omega^{2}\right)^{2}}
        \end{split}
    \end{equation}
    to derive a simpler formula for $f$. First, note that 
    \begin{equation}
        \sum_{\substack{n \in \mathbb{Z} \\ n \ \mathrm{odd}}} \left( \frac{1}{\left(k + n\right)^{2}} + \frac{1}{\left(k - n\right)^{2}} - \frac{2}{n^{2}} \right) = 0.
    \end{equation}
    Second, we use the Mittag-Leffler formula
    \begin{equation}\label{eq:Mittag-Leffler_tan}
        \pi \tan(\pi z) = 8 z  \sum_{n=0}^{+\infty} \frac{1}{(2n+1)^2 - 4z^2} = 4 z \sum_{\substack{n \in \mathbb{Z} \\ n \ \mathrm{odd}}} \frac{1}{n^2 - 4z^2},
    \end{equation}
    which yields
    \begin{equation}
        \sum_{\substack{n \in \mathbb{Z} \\ n \ \mathrm{odd}}} \frac{4 \left(3 k^{4} - \omega^{2}\right)}{\left(k^2 - n^2 \right) \left(k^{4} + \omega^{2}\right)^{2}}
        = - \frac{2 \pi}{k} \frac{\left(3 k^{4} - \omega^{2}\right)}{\left(k^{4} + \omega^{2}\right)^{2}} \tan \! \left( \frac{\pi k}{2} \right) = 0.
    \end{equation}
    Hence, for now, we have proved 
    \begin{equation}
        f(\omega) = - \frac{8}{\left(k^{4} + \omega^{2}\right)^{2}} \sum_{\substack{n \in \mathbb{Z} \\ n \ \mathrm{odd}}} \frac{(\omega^2 - 3k^4) (2 n^2 - k^2) + 4 k^2 \omega^2}{ (2 n^2 - k^2)^2 + \omega^{2}} .
    \end{equation} 
    A straightforward calculation gives
    \begin{equation}
        \frac{(\omega^2 - 3k^4) (2 n^2 - k^2) + 4 k^2 \omega^2}{ (2 n^2 - k^2)^2 + \omega^{2}} = \Re \left( \frac{\omega^2 - 3 k^4 - 4 i k^2 \omega}{2n^2 - k^2 - i \omega} \right).
    \end{equation}
    Set
    \begin{equation}
        z^2 = \frac{k^2 + i\omega}{2}, \quad z := a+ib, \quad a^2 - b^2 = \frac{k^2}{2}, \quad 4 ab = \omega,
    \end{equation}
    with $a > 0$ and $b \geq 0$, so that $2n^2 - k^2 - i \omega = 2 \left( n^2 - z^2 \right)$ and $\omega^2 - 3 k^4 - 4 i k^2 \omega = - 4 z^2 (k^2 + z^2)$. Using \eqref{eq:Mittag-Leffler_tan} again, one finds
    \begin{equation}
        f(\omega) = \frac{8\pi}{\left(k^{4} + \omega^{2}\right)^{2}} \Re \left( z (k^2 + z^2)  \tan \! \left( \frac{\pi z}{2} \right) \right).
    \end{equation}
    If $\omega = 0$, then $b = 0$ and $z = a = \frac{k}{\sqrt{2}}$, implying that $f(0) < 0$ if and only if $\tan \! \left( \frac{\pi k}{2 \sqrt{2}} \right) < 0$, that is, if and only if $k \in \left( \sqrt{2}, 2 \sqrt{2} \right) + 2 \sqrt{2} \mathbb{Z}$. To complete the proof, we assume that this condition holds true, and we prove that it implies $f(\omega) < 0$ for all $\omega > 0$. We use the elementary formula
    \begin{equation}
        \tan \! \left( \frac{\pi (a+ib)}{2} \right) = \frac{\sin(\pi a) + i \sinh(\pi b)}{\cos(\pi a) + \cosh(\pi b)}.
    \end{equation}
    Note that since $b > 0$, the denominator $\cos(\pi a) + \cosh(\pi b)$ is a positive real number. Hence, $f(\omega)$ is of the same sign as
    \begin{equation}
        g(\omega) := \Re \left( z (k^2 + z^2) \left( \sin(\pi a) + i \sinh(\pi b) \right) \right)
        = a \sin( \pi a ) (3 a^2 - 5 b^2) - b \sinh(\pi b) (5 a^2 - 3 b^2).
    \end{equation}
    Note that since $a > b > 0$, one has $b \sinh(\pi b) > 0$ and $5 a^2 - 3 b^2 > 0$. We distinguish two cases.
    
    \underline{Case 1.} Assume that $\sin(\pi a) \leq 0$. 
    This case does not rely on the assumption $k \in \left( \sqrt{2}, 2 \sqrt{2} \right) + 2 \sqrt{2} \mathbb{Z}$. 
    If $3 a^2 - 5 b^2 \geq 0$, then $g(\omega) < 0$, so we can assume that $3 a^2 - 5 b^2 < 0$. Using $\left\vert \sin(\pi a) \right\vert \leq 1 \leq \pi$ and $\pi b < \sinh(\pi b)$, one finds
    \begin{equation}
        \begin{split}
            g(\omega) & = a \left\vert \sin( \pi a ) \right\vert (5 b^2 - 3 a^2) - b \sinh(\pi b) (5 a^2 - 3 b^2) \\
            & < \pi \left( a (5 b^2 - 3 a^2) - b^2 (5 a^2 - 3 b^2) \right) 
            = \pi a \left( (5 b^2 - 3 a^2) - \frac{b^2}{a} (5 a^2 - 3 b^2) \right).
        \end{split}
    \end{equation}
    Note that $a > \frac{k}{\sqrt{2}} \geq \sqrt{2}$, implying $\frac{b^2}{a} \geq \sqrt{2} \frac{b^2}{a^2}$. Using also $5 a^2 - 3 b^2 > 0$, and setting $r := \frac{b^2}{a^2}$, one obtains
    \begin{equation}
        g(\omega) < \pi a^3 \left( (5 r - 3 ) - \sqrt{2} r (5 - 3 r) \right).
    \end{equation}
    Since we assume $3 a^2 - 5 b^2 < 0$, one has $r \in \left(\frac{3}{5}, 1 \right)$. The function $r \mapsto (5 r - 3 ) - \sqrt{2} r (5 - 3 r)$ is convex, and is negative at $\frac{3}{5}$ and $1$, implying that $g(\omega) < 0$. 
    
    \underline{Case 2.} Assume that $\sin(\pi a) \geq 0$. 
    If $3 a^2 - 5 b^2 \leq 0$, then $g(\omega) < 0$, so we can assume that $3 a^2 - 5 b^2 > 0$. 
    Since $\sin(\pi a) \geq 0$, there exists $m \geq 0$ such that $2m \leq a \leq 2m+1$.  By assumption, there exists $n \geq 0$ such that $2n+1 < \frac{k}{\sqrt{2}} < 2(n+1)$. Using $a > \frac{k}{\sqrt{2}}$, one finds $m \geq n+1$, and hence $2m \geq \frac{k}{\sqrt{2}}$. Using also $\left\vert \sin(\pi a) \right\vert \leq \pi(a - 2m)$ and $\pi b < \sinh(\pi b)$, one finds
    \begin{equation}
        \begin{split}
            g(\omega) & = a \sin( \pi a )  (3 a^2 - 5 b^2) - b \sinh(\pi b) (5 a^2 - 3 b^2) \\
            & < \pi \left( a \left( a - 2m \right) (3 a^2 - 5 b^2) - b^2 (5 a^2 - 3 b^2) \right)\\
            & \leq \pi \left( a \left( a - \frac{k}{\sqrt{2}} \right) (3 a^2 - 5 b^2) - b^2 (5 a^2 - 3 b^2) \right).
        \end{split}
    \end{equation}
    As above, we set $r := \frac{b^2}{a^2}$, which gives, using also $b^2 = a^2 - \frac{k^2}{2}$,
    \begin{equation}
        g(\omega) < a^3 \pi \left( a - \frac{k}{\sqrt{2}} \right) \left( (3 - 5 r) - \left( 1 + \frac{k}{a\sqrt{2}} \right) (5 - 3 r) \right).
    \end{equation}
    Since we assume $3 a^2 - 5 b^2 > 0$, one has $r \in \left(0, \frac{3}{5}\right)$. The function $r \mapsto (3 - 5 r) - \left( 1 + \frac{k}{a \sqrt{2}} \right) (5 - 3 r)$ is affine, negative at $0$ and $\frac{3}{5}$, and hence negative on $\left(0, \frac{3}{5}\right)$, implying that $g(\omega) < 0$.
\end{proof}

\section{Quadratic obstructions across a full range of fractional Sobolev norms}\label{sec:proof_obstructions}

In this section, we relate the properties of the kernel $K_{\mu,k}$ to drift estimates, and show that these estimates lead to obstructions to controllability. We first prove that an asymptotic estimate on $K_{\mu,k}$ implies a drift estimate in small time, and that this drift estimate provides an obstruction to small-time controllability. We then establish analogous finite-time results under an additional global sign assumption on $K_{\mu,k}$. The proofs rely crucially on the remainder estimates obtained in Section \ref{sec:remainder_estimates}, together with the technical results on Sobolev norms provided in Appendix \ref{sec:appendix_sobolev}.

\subsection{Obstruction to small-time controllability}

The most natural case in Theorem \ref{thm:main:kernel_implies_drift_small_time} below is $s=s_0$, where the natural asymptotic decay of the kernel $K_{\mu,k}$ reflects the maximal regularity of $\mu$. However, as in Proposition \ref{prop:global_sign_AND_asymptotic_estimates}, the kernel $K_{\mu,k}$ may decay faster than expected; in that case, one must introduce a second regularity threshold $s_0 \geq s$. Recall that $\sigma(s) := - \frac{1+s}{2}$ for $s \in \left[-1, \frac{1}{2} \right]$, and $\sigma(s) := -1 + \frac{s}{2}$ for $s \in \left[ \frac{1}{2}, 1 \right)$. 

\begin{theorem}[An asymptotic estimate of the kernel implies a drift estimate in small time]\label{thm:main:kernel_implies_drift_small_time}  
    Let $p \in \left\{2, +\infty\right\}$, $s, s_0 \in [-1, 1)$, with $s_0 \geq s$ and $\sigma(s) + s_0 - s < \frac{1}{2}$, and let $k \geq 1$. If $p=2$, then let $\mu \in \mathcal{H}_2^s(0,1)$. If $p = +\infty$, assume that $s \notin \{-1, 0\}$, and let $\mu \in \mathcal{H}_\infty^{s+\frac{1}{2}}(0,1)$. Assume that $\langle \mu, \varphi_k \rangle = 0$, and
    \begin{equation}\label{eq:test:asymptotique_kernel}
        K_{\mu,k}(\omega) = \frac{\alpha}{\langle \omega \rangle^{2+s_0}} + \mathcal{O}\left( \frac{1}{\langle \omega \rangle^{2+s_0+\nu}} \right)
    \end{equation}
    when $\omega \rightarrow + \infty$, for some $\nu > 0$ and $\alpha \neq 0$. Then, there exist $\nu_0 > 0$, $\eta_0 > 0$, $\delta_0 > 0$, $T_0 > 0$ and $C > 0$ such that for all $T \in (0, T_0)$, all $u \in L^2(0,T) \cap \widetilde{H}^{\sigma(s) + s_0 - s}(0,T)$ satisfying $\delta := \Vert u \Vert_{\widetilde{H}^{\sigma(s) + s_0 - s}(0,T)} \leq \delta_0$, and all $y_0 \in L^2(0, 1)$ satisfying $\Vert y_0 \Vert_{L^2} \leq \eta_0$, one has
    \begin{equation}\label{eq:thm:main:kernel_implies_drift_small_time} 
        \begin{split}
            & \left\vert \left\langle y(T; y_0, u) - y(T; y_0, 0), \varphi_k \right\rangle - \frac{\alpha k }{4 \pi^2 \sqrt{2}} \Vert u \Vert_{\widetilde{H}^{- 1 - \frac{s_0}{2}}(0,T)}^2\right\vert \\
            \leq \ & C \left( T^{\nu_0} + \delta \right) \Vert u \Vert_{\widetilde{H}^{-1 - \frac{s_0}{2}}}^2 + C \left\Vert y(T; y_0, u) - y(T; y_0, 0) \right\Vert_{H^{-1}}^2 + C T^{\frac{1}{4}} \delta \Vert y_0 \Vert_{L^2}.
        \end{split}
    \end{equation}
\end{theorem}

\begin{proof}
    In this proof, the symbol $\lesssim$ is used for constants that may depend on $k$, $\mu$, $s$ and $s_0$. 
    
    \textbf{Step 1.} \textit{Using \eqref{eq:test:asymptotique_kernel} to estimate the quadratic term.}
    Let $T > 0$. Write $v_{k,T} := u \rho_k \mathbbm{1}_{[0,T]}$, with $\rho_k(t) := e^{\frac{\lambda_k t}{2}}$. By \eqref{eq:test:asymptotique_kernel}, and since $K_{\mu,k}$ is even and continuous, one has 
    \begin{equation}
        \left\vert K_{\mu,k}(\omega) - \frac{\alpha}{\langle \omega \rangle^{2+s_0}} \right\vert \lesssim \frac{1}{\langle \omega \rangle^{2+s_0+\nu}},
    \end{equation}
    for all $\omega \in \mathbb{R}$, implying
    \begin{equation}
        \left\vert \int_{\mathbb{R}} \left\vert \widehat{v_{k,T}}(\omega) \right\vert^2 K_{\mu,k}(\omega) \dd \omega  - \alpha \Vert u \rho_k \Vert_{\widetilde{H}^{-1 - \frac{s_0}{2}}}^2 \right\vert \lesssim \Vert u \rho_k \Vert_{\widetilde{H}^{-1 - \frac{s_0 + \nu}{2}}}^2.
    \end{equation} 
    By Proposition \ref{prop:decomposition_y_2_kernel}, this gives
    \begin{equation}
        \begin{split}
            & \left\vert \left\langle y_2(T; u), \varphi_k \right\rangle - \frac{\alpha k e^{-\lambda_k T}}{4 \pi^2 \sqrt{2}} \Vert u \rho_k \Vert_{\widetilde{H}^{-1 - \frac{s_0}{2}}}^2\right\vert \\
            \lesssim \ & \Vert u \rho_k \Vert_{\widetilde{H}^{-1 - \frac{s_0 + \nu}{2}}}^2 + \Vert y_1(T;u) \Vert_{H^{-1}}^2 
            + \sum_{n\geq 1} \mathbbm{1}_{n^2 < \frac{k^2}{2}} \left\vert \langle y_1(T;u), \varphi_n \rangle \right\vert \left\vert \int_0^T u(t) e^{\left(\lambda_k - \lambda_n \right) t} \dd t \right\vert,
        \end{split}   
    \end{equation}
    where $t \mapsto y_1(t; u)$ and $t \mapsto y_2(t; u)$ are the solutions of \eqref{eq:Burgers_y_1_intro} and \eqref{eq:Burgers_y_2_intro}. This estimate can be simplified. First, by Lemma \ref{lem:appendix:weak_norm_power_T_new}, applied to $\rho_k u$, one has
    \begin{equation}
        \Vert u \rho_k \Vert_{\widetilde{H}^{-1 - \frac{s_0 + \nu}{2}}}^2 \lesssim T^{\nu_0} \Vert u \rho_k \Vert_{\widetilde{H}^{-1 - \frac{s_0}{2}}}^2 + \left\vert  \int_0^T u(t) \rho_k(t) \dd t \right\vert^2,
    \end{equation}
    for some $\nu_0 > 0$. Second, closed loop estimates (Lemma \ref{lem:closed_loop} \emph{(i)}), together with the elementary estimate $ab \lesssim a^2+b^2$, give
    \begin{equation}
        \begin{split}
            & \left\vert  \int_0^T u(t) \rho_k(t) \dd t \right\vert^2 + \sum_{n\geq 1} \mathbbm{1}_{n^2 < \frac{k^2}{2}} \left\vert \langle y_1(T;u), \varphi_n \rangle \right\vert \left\vert \int_0^T u(t) e^{\left(\lambda_k - \lambda_n \right) t} \dd t \right\vert \\
            \lesssim \ & T^{\nu_0} \Vert u \Vert_{\widetilde{H}^{-1 - \frac{s_0}{2}}}^2 + \Vert y_1(T;u) \Vert_{H^{-1}}^2,
        \end{split}   
    \end{equation}
    up to reducing $\nu_0 > 0$. Third, by Lemma \ref{lem:appendix:removing_exp_from_sobolev_norm_small_time}, one has
    \begin{equation}
        \begin{split}
            \left\vert \left\Vert u \rho_k \right\Vert_{\widetilde{H}^{-1 - \frac{s_0}{2}}}^2 - \left\Vert u \right\Vert_{\widetilde{H}^{-1 - \frac{s_0}{2}}}^2 \right\vert 
            \leq \ & \left\vert \left\Vert u \rho_k \right\Vert_{\widetilde{H}^{-1 - \frac{s_0}{2}}} - \left\Vert u \right\Vert_{\widetilde{H}^{-1 - \frac{s_0}{2}}} \right\vert \left( \left\Vert u \rho_k \right\Vert_{\widetilde{H}^{-1 - \frac{s_0}{2}}} + \left\Vert u \right\Vert_{\widetilde{H}^{-1 - \frac{s_0}{2}}} \right) \\
            \lesssim \ & T^{\nu_0} \Vert u \Vert_{\widetilde{H}^{-1 - \frac{s_0}{2}}}^2, 
        \end{split}
    \end{equation}
    up to reducing $\nu_0 > 0$. Using also $\left\vert 1 - e^{-\lambda_k T} \right\vert = 1 - e^{-\lambda_k T} \leq \lambda_k T$, one obtains
    \begin{equation}
        \begin{split}
            \left\vert \left\langle y_2(T; u), \varphi_k \right\rangle - \frac{\alpha k}{4 \pi^2 \sqrt{2}} \Vert u \Vert_{\widetilde{H}^{-1 - \frac{s_0}{2}}}^2\right\vert 
            \lesssim   T^{\nu_0} \Vert u \Vert_{\widetilde{H}^{-1 - \frac{s_0}{2}}}^2 + \Vert y_1(T;u) \Vert_{H^{-1}}^2 ,
        \end{split}   
    \end{equation}
    up to reducing $\nu_0 > 0$.

    \textbf{Step 2.} \textit{Coming back to the full controlled solution.}
    Write $t \mapsto R_2(t; u)$ and $t \mapsto R_3(t; u)$ for the quadratic and cubic remainders, defined by $R_2(t;u) := y(t; 0, u) - y_1(t; u)$ and $R_3(t;u) := y(t; 0, u) - y_1(t; u) - y_2(t; u)$. Let $y_0 \in L^2(0,1)$, and write $t \mapsto R(t; y_0, u)$ for the remainder of the decoupling estimate (Lemma \ref{lem:decoupling_estimate}), that is, $R(t; y_0, u) := y(t; y_0, u) - y(t; y_0, 0) - y(t; 0, u)$. Since $\langle \mu, \varphi_k \rangle = 0$, one has 
    \begin{equation}\label{eq:proof:thm:main_kernel_implies_STLC_step3_eq_1}
        \begin{split}
            \left\langle y_2(T; u), \varphi_k \right\rangle 
            & = \left\langle y(T; 0, u), \varphi_k \right\rangle - \left\langle R_3(T; u), \varphi_k \right\rangle \\
            & = \left\langle y(T; y_0, u) - y(T; y_0, 0), \varphi_k \right\rangle - \left\langle R_3(T; u), \varphi_k \right\rangle - \left\langle R(T; y_0, u), \varphi_k \right\rangle.
        \end{split}
    \end{equation}
    Using also
    \begin{equation}\label{eq:proof:thm:main_kernel_implies_STLC_step3_eq_2}
        \begin{split}
            \Vert y_1(T; u) \Vert_{H^{-1}}^2 \lesssim  \left\Vert y(T; y_0, u) - y(T; y_0, 0) \right\Vert_{H^{-1}}^2 + \left\Vert R_2(T; u) \right\Vert_{H^{-1}}^2 + \left\Vert R(T; y_0, u) \right\Vert_{H^{-1}}^2,
        \end{split}
    \end{equation}
    one obtains
    \begin{equation}\label{eq:proof:thm:main_kernel_implies_STLC_step2_eq_1_new}
        \begin{split}
            & \left\vert \left\langle y(T; y_0, u) - y(T; y_0, 0), \varphi_k \right\rangle - \frac{\alpha k }{4 \pi^2 \sqrt{2}} \Vert u \Vert_{\widetilde{H}^{-1 - \frac{s_0}{2}}}^2\right\vert \\
            \lesssim \ &  T^{\nu_0} \Vert u \Vert_{\widetilde{H}^{-1 - \frac{s_0}{2}}}^2 +  \left\Vert y(T; y_0, u) - y(T; y_0, 0) \right\Vert_{H^{-1}}^2 + \mathscr{R}(T; y_0, u),
        \end{split}
    \end{equation}
    where
    \begin{equation}\label{eq:proof:thm:main_kernel_implies_STLC_step2_eq_1_bis_new}
        \mathscr{R}(T; y_0, u) := \left\Vert R_2(T; u) \right\Vert_{H^{-1}}^2 + \left\vert \left\langle R_3(T; u), \varphi_k \right\rangle \right\vert + \left\Vert R(T; y_0, u) \right\Vert_{H^{-1}}^2 + \left\vert \left\langle R(T; y_0, u), \varphi_k \right\rangle \right\vert.
    \end{equation}
    
    We now estimate $\mathscr{R}(T; y_0, u)$. First, by quadratic and cubic remainder estimates (Propositions \ref{prop:quad_estimates_new} and \ref{prop:cubic_estimates_new} when $p=2$, Proposition \ref{prop:remainder_estimates_p_infty} when $p=+\infty$), one has
    \begin{equation}
        \left\Vert R_2(T; u) \right\Vert_{H^{-1}}^2 + \left\vert \left\langle R_3(T; u), \varphi_k \right\rangle \right\vert 
        \lesssim \Vert u \Vert_{\widetilde{H}^{-1 - \frac{s}{2}}}^2 \Vert u \Vert_{\widetilde{H}^{\sigma(s)}} ( 1 + \delta)
        \lesssim \Vert u \Vert_{\widetilde{H}^{-1 - \frac{s}{2}}}^2 \Vert u \Vert_{\widetilde{H}^{\sigma(s)}},
    \end{equation}
    where we have used Lemma \ref{lem:appendix:weak_norm_primitive_new} to bound norms of $u_1$ by norms of $u$. To deal with the case $s_0 \neq s$, we use Lemma \ref{lem:remainder_K_decays_too_fast}, which gives
    \begin{equation}\label{eq:proof:thm:main_kernel_implies_STLC_step2_eq_2_new}
        \left\Vert R_2(T; u) \right\Vert_{H^{-1}}^2 + \left\vert \left\langle R_3(T; u), \varphi_k \right\rangle \right\vert 
        \lesssim \Vert u \Vert_{\widetilde{H}^{-1 - \frac{s_0}{2}}}^2 \Vert u \Vert_{\widetilde{H}^{\sigma(s)+s_0-s}} 
        =  \Vert u \Vert_{\widetilde{H}^{-1 - \frac{s_0}{2}}}^2 \delta. 
    \end{equation}
    Second, to use the decoupling estimate (Lemma \ref{lem:decoupling_estimate}), we need to estimate $t \mapsto y(t;0,u)$ in $L^2((0,T);H^1(0,1))$ and in $L^2((0,T);L^2(0,1))$. Using estimates on the linearized term (Lemma \ref{lem:linear_estimates_new}) and quadratic remainder estimates (Proposition \ref{prop:quad_estimates_new}), one finds
    \begin{equation}
        \left\Vert t \mapsto y(t;0,u) \right\Vert_{L^2(H^1)} 
        \leq \left\Vert y_1 \right\Vert_{L^2(H^1)} + \left\Vert R_2 \right\Vert_{L^2(H^1)} \lesssim \Vert u \Vert_{\widetilde{H}^{\sigma(s)}} 
        \lesssim \delta.
    \end{equation}
    Similarly, one has
    \begin{equation}
        \left\Vert t \mapsto y(t;0,u) \right\Vert_{L^2(L^2)} 
        \leq \left\Vert y_1 \right\Vert_{L^2(L^2)} + \left\Vert R_2 \right\Vert_{L^2(H^1)} 
        \lesssim \Vert u \Vert_{\widetilde{H}^{- 1 - \frac{s}{2}}} + \Vert u \Vert_{\widetilde{H}^{- 1 - \frac{s}{2}}} \Vert u \Vert_{\widetilde{H}^{\sigma(s)}}
        \lesssim \delta
    \end{equation}
    if $s \leq 0$, and 
    \begin{equation}
    \begin{split}
        \left\Vert t \mapsto y(t;0,u) \right\Vert_{L^2(L^2)} 
        & \ \leq \Vert u_1 \mu \Vert_{L^2(L^2)} + \left\Vert y_1 - u_1 \mu \right\Vert_{L^2(L^2)} + \left\Vert R_2 \right\Vert_{L^2(H^1)}  \\
        & \ \lesssim \Vert u \Vert_{\widetilde{H}^{- 1}} + \Vert u \Vert_{\widetilde{H}^{- 1 - \frac{s}{2}}} +  \Vert u \Vert_{\widetilde{H}^{- 1 - \frac{s}{2}}} \Vert u \Vert_{\widetilde{H}^{\sigma(s)}} \\
        & \ \lesssim \delta,
    \end{split}
    \end{equation}
    where we have used Lemma \ref{lem:appendix:weak_norm_primitive_new} again to bound norms of $u_1$ by norms of $u$, and $-1 \leq \sigma(s) + s_0 - s$. Hence, if $\delta$ and $\Vert y_0 \Vert_{L^2}$ are sufficiently small, then the decoupling estimate (Lemma \ref{lem:decoupling_estimate}) gives 
    \begin{equation}\label{eq:proof:thm:main_kernel_implies_STLC_step2_eq_3_new}
         \left\Vert R(T; y_0, u) \right\Vert_{H^{-1}}^2 + \left\vert \left\langle R(T; y_0, u), \varphi_k \right\rangle \right\vert 
         \lesssim T^{\frac{1}{4}} \delta \Vert y_0 \Vert_{L^2} .
    \end{equation}
     Gathering \eqref{eq:proof:thm:main_kernel_implies_STLC_step2_eq_1_new}, \eqref{eq:proof:thm:main_kernel_implies_STLC_step2_eq_2_new} and \eqref{eq:proof:thm:main_kernel_implies_STLC_step2_eq_3_new}, one obtains \eqref{eq:thm:main:kernel_implies_drift_small_time} .
\end{proof}

We now prove that the drift estimate of Theorem \ref{thm:main:kernel_implies_drift_small_time} provides an obstruction to small-time controllability.

\begin{corollary}\label{cor:obstruction_small_time}
    Under the assumptions of Theorem \ref{thm:main:kernel_implies_drift_small_time}, the Burgers equation \eqref{eq:Burgers_intro} is not STLC. More precisely, there exist $c > 0$, $T_1 > 0$, $\eta_1 > 0$ and $\delta_1 > 0$, such that for all $T \in (0, T_1]$, all $u \in L^2(0,T) \cap \widetilde{H}^{\sigma(s) + s_0 - s}(0,T)$ satisfying $\Vert u \Vert_{\widetilde{H}^{\sigma(s) + s_0 - s}(0,T)} \leq \delta_1$, and all $\eta \in (0, \eta_1)$, one has 
    \begin{equation}\label{eq:cor:obstruction_small_time}
        \left\Vert y\!\left(T; \eta \widetilde{y_0}, u \right) \right\Vert_{H^{-1}(0,1)} + \left\Vert y\!\left(T; \eta \widetilde{y_0}, u \right) \right\Vert_{H^{-1}(0,1)}^2 \geq c \left( \eta + \Vert u \Vert_{\widetilde{H}^{- 1 - \frac{s_0}{2}}(0,T)}^2 \right) > 0,
    \end{equation}
    where $\widetilde{y_0} := \frac{\alpha}{\vert \alpha\vert} \varphi_k$.
\end{corollary}

\begin{proof}
    In this proof, the symbol $\lesssim$ is used for constants that may depend on $k$, $\mu$, $s$ and $s_0$. Using Theorem \ref{thm:main:kernel_implies_drift_small_time} and the estimate for the free Burgers equation proved in Lemma \ref{lem:free_evolution_Burgers}, one finds
    \begin{equation}
        \begin{split}
            & \left\vert \left\langle y(T; y_0, u) - y_0, \varphi_k \right\rangle - \frac{\alpha k }{4 \pi^2 \sqrt{2}} \Vert u \Vert_{\widetilde{H}^{- 1 - \frac{s_0}{2}}(0,T)}^2\right\vert \\
            \lesssim \ & \left( T^{\nu_0} + \delta \right) \Vert u \Vert_{\widetilde{H}^{-1 - \frac{s_0}{2}}}^2 + \left\Vert y(T; y_0, u) - y_0 \right\Vert_{H^{-1}}^2 + T^{\frac{1}{4}} \Vert y_0 \Vert_{L^2} \\
            \lesssim \ & \left( T^{\nu_0} + \delta \right) \Vert u \Vert_{\widetilde{H}^{-1 - \frac{s_0}{2}}}^2 + \left\Vert y(T; y_0, u) \right\Vert_{H^{-1}}^2 + T^{\frac{1}{4}} \Vert y_0 \Vert_{L^2} + \Vert y_0 \Vert_{L^2}^2.
        \end{split}
    \end{equation}
    We assume that $\alpha > 0$; the case $\alpha < 0$ is similar. Using also $\langle \psi, \varphi_k \rangle \lesssim \Vert \psi \Vert_{H^{-1}}$, one obtains
    \begin{equation}
        \begin{split}
            & \frac{\alpha k }{4 \pi^2 \sqrt{2}} \Vert u \Vert_{\widetilde{H}^{- 1 - \frac{s_0}{2}}(0,T)}^2 + \left\langle  y_0, \varphi_k \right\rangle \\
            \lesssim \ & \left( T^{\nu_0} + \delta \right) \Vert u \Vert_{\widetilde{H}^{-1 - \frac{s_0}{2}}}^2 + \left\Vert y(T; y_0, u) \right\Vert_{H^{-1}} + \left\Vert y(T; y_0, u) \right\Vert_{H^{-1}}^2 + T^{\frac{1}{4}} \Vert y_0 \Vert_{L^2} + \Vert y_0 \Vert_{L^2}^2 .
        \end{split}
    \end{equation}
    Hence, if $y_0 = \eta \varphi_k$, then there exists $C > 0$ such that 
   \begin{equation}
        \begin{split}
            & \left( \frac{\alpha k }{4 \pi^2 \sqrt{2}} - C T^{\nu_0} - C \delta \right) \Vert u \Vert_{\widetilde{H}^{- 1 - \frac{s_0}{2}}(0,T)}^2 + \left( 1 - C T^{\frac{1}{4}} - C \eta \right) \eta  \\
            \lesssim \ & C \left\Vert y(T; y_0, u) \right\Vert_{H^{-1}} + C \left\Vert y(T; y_0, u) \right\Vert_{H^{-1}}^2.
        \end{split}
    \end{equation}
    This gives \eqref{eq:cor:obstruction_small_time} if $T$, $\eta$ and $\delta$ are sufficiently small.
 \end{proof}

\subsection{Obstruction to finite-time controllability}

Here, we prove the following finite-time analogue of Theorem \ref{thm:main:kernel_implies_drift_small_time}.

\begin{theorem}[An asymptotic estimate of the kernel and a global sign property imply a drift estimate in finite time]\label{thm:main:kernel_implies_drift_finite_time}
    Let $p \in \left\{2, +\infty\right\}$, $s, s_0 \in [-1, 1)$, with $s_0 \geq s$ and $\sigma(s) + s_0 - s < \frac{1}{2}$, and let $k \geq 1$. If $p=2$, then let $\mu \in \mathcal{H}_2^s(0,1)$. If $p = +\infty$, assume that $s \notin \{-1, 0\}$, and let $\mu \in \mathcal{H}_\infty^{s+\frac{1}{2}}(0,1)$. Assume that $\langle \mu, \varphi_k \rangle = 0$, that
    \begin{equation}\label{eq:thm:main:kernel_implies_drift_finite_time_1}
        K_{\mu,k}(\omega) = \frac{\alpha}{\langle \omega \rangle^{2+s_0}} + \mathcal{O}\left( \frac{1}{\langle \omega \rangle^{2+s_0+\nu}} \right), \quad \text{when $\omega \rightarrow + \infty$,}
    \end{equation}
    and that
    \begin{equation}\label{eq:thm:main:kernel_implies_drift_finite_time_2}
        \alpha K_{\mu,k}(\omega) > 0 , \quad \text{for all $\omega \in \mathbb{R}$},
    \end{equation}
    for some $\nu > 0$ and $\alpha \neq 0$. Then, for all $T > 0$, there exists $\eta_0 > 0$, $\delta_0 > 0$, and $C > 0$ such that for all $u \in L^2(0,T) \cap \widetilde{H}^{\sigma(s) + s_0 - s}(0,T)$ satisfying $\delta := \Vert u \Vert_{\widetilde{H}^{\sigma(s) + s_0 - s}(0,T)} \leq \delta_0$, and all $y_0 \in L^2(0, 1)$ satisfying $\Vert y_0 \Vert_{L^2} \leq \eta_0$, one has
    \begin{equation}\label{eq:thm:main:kernel_implies_drift_finite_time_3}
        \begin{split}
            & \left\vert \left\langle y(T; y_0, u) - y(T; y_0, 0), \varphi_k \right\rangle - \frac{\alpha}{\vert \alpha \vert} \frac{k e^{-\lambda_k T}}{4 \pi^2 \sqrt{2}} \Vert u \rho_k \Vert_{K_{\mu, k}}^2 \right\vert \\
            \leq \ & C \left\Vert y(T; y_0, u) - y(T; y_0, 0) \right\Vert_{H^{-1}}^2 + C \delta^{\frac{1}{2}} \Vert u \Vert_{\widetilde{H}^{-1-\frac{s_0}{2}}}^2 + C \delta^{\frac{1}{2}} \Vert y_0 \Vert_{L^2} \\
            & + C \Vert u \Vert_{\widetilde{H}^{- 1 - \frac{s_0}{2}}}  \left\Vert y(T; y_0, u) - y(T; y_0, 0) \right\Vert_{H^{-1}} ,
        \end{split}
    \end{equation}    
    where $\rho_k(t) := e^{\frac{\lambda_k t}{2}}$ and
    \begin{equation}\label{eq:def_norm_K_mu_k}
        \Vert u \Vert_{K_{\mu, k}}^2 := \frac{\alpha}{\vert \alpha \vert} \int_{\mathbb{R}} K_{\mu,k}(\omega) \left\vert \widehat{u}(\omega) \right\vert^2 \dd \omega.
    \end{equation}
 \end{theorem}
 
\begin{proof}
    Let $T > 0$. In this proof, the symbol $\lesssim$ is used for constants that may depend on $T$, $k$, $\mu$, $s$ and $s_0$. 

    \textbf{Step 1.} \textit{Using \eqref{eq:thm:main:kernel_implies_drift_finite_time_1} and \eqref{eq:thm:main:kernel_implies_drift_finite_time_2} to estimate the quadratic term.}
    By \eqref{eq:thm:main:kernel_implies_drift_finite_time_1} and \eqref{eq:thm:main:kernel_implies_drift_finite_time_2}, and since $K_{\mu, k}$ is even and continuous, one has
    \begin{equation}
        1 \lesssim \frac{K_{\mu,k}(\omega) \langle \omega \rangle^{2+s_0}}{\alpha} \lesssim 1,
    \end{equation}
    for all $\omega \in \mathbb{R}$. This implies that for all $u \in \widetilde{H}^{-1 - \frac{s_0}{2}}(0,T)$, one has
    \begin{equation}\label{eq:proof:thm:main:kernel_implies_drift_finite_time:step_1_eq_2}
        \Vert u \Vert_{\widetilde{H}^{-1 - \frac{s_0}{2}}}^2 \lesssim \frac{\alpha}{\vert \alpha \vert} \int_{\mathbb{R}} K_{\mu,k}(\omega) \left\vert \widehat{u}(\omega) \right\vert^2 \dd \omega \lesssim \Vert u \Vert_{\widetilde{H}^{-1 - \frac{s_0}{2}}}^2,
    \end{equation}
    where $u$ is implicitly extended by zero outside $(0,T)$, so that the norm defined by \eqref{eq:def_norm_K_mu_k} is indeed a norm. By Proposition \ref{prop:decomposition_y_2_kernel}, one has
    \begin{equation}
        \begin{split}
            & \left\vert \left\langle y_2(T; u), \varphi_k \right\rangle - \frac{\alpha}{\vert \alpha \vert} \frac{k e^{-\lambda_k T}}{4 \pi^2 \sqrt{2}} \Vert u \rho_k \Vert_{K_{\mu, k}}^2 \right\vert \\
            \lesssim \ & \Vert y_1(T;u) \Vert_{H^{-1}}^2 + \Vert y_1(T;u) \Vert_{H^{-1}} \sum_{n\geq 1} \mathbbm{1}_{n^2 < \frac{k^2}{2}} \left\vert \int_0^T u(t) e^{\left(\lambda_k - \lambda_n \right) t} \dd t \right\vert,
        \end{split}
    \end{equation}
    where $t \mapsto y_1(t; u)$ and $t \mapsto y_2(t; u)$ are the solutions of \eqref{eq:Burgers_y_1_intro} and \eqref{eq:Burgers_y_2_intro}. 
    By Lemma \ref{lem:u_1_finite_time}, one has
    \begin{equation}
        \sum_{n\geq 1} \mathbbm{1}_{n^2 < \frac{k^2}{2}} \left\vert \int_0^T u(t) e^{\left(\lambda_k - \lambda_n \right) t} \dd t \right\vert
        \lesssim \Vert u \Vert_{\widetilde{H}^{-1 - \frac{s_0}{2}}},
    \end{equation}
    yielding
    \begin{equation}
        \left\vert \left\langle y_2(T; u), \varphi_k \right\rangle - \frac{\alpha}{\vert \alpha \vert} \frac{k e^{-\lambda_k T}}{4 \pi^2 \sqrt{2}} \Vert u \rho_k \Vert_{K_{\mu, k}}^2 \right\vert 
        \lesssim  \Vert y_1(T;u) \Vert_{H^{-1}}^2 + \Vert y_1(T;u) \Vert_{H^{-1}} \Vert u \Vert_{\widetilde{H}^{-1 - \frac{s_0}{2}}} .
    \end{equation}
    
    \textbf{Step 2.} \textit{Coming back to the full controlled solution.}
    Let $y_0 \in L^2(0,1)$. As in the proof of Theorem \ref{thm:main:kernel_implies_drift_small_time}, write $R_2(t;u)$, $R_3(t;u)$ and $R(t; y_0, u)$ for the quadratic, cubic and decoupling remainders. Let also $\mathscr{R}(T; y_0, u)$ be given by \eqref{eq:proof:thm:main_kernel_implies_STLC_step2_eq_1_bis_new}. Arguing as in the proof of Theorem \ref{thm:main:kernel_implies_drift_small_time}, one finds 
    \begin{equation}\label{eq:proof:thm:main:kernel_implies_drift_finite_time:step_2_eq_1}
        \begin{split}
            & \left\vert \left\langle y(T; y_0, u) - y(T; y_0, 0), \varphi_k \right\rangle - \frac{\alpha}{\vert \alpha \vert} \frac{k e^{-\lambda_k T}}{4 \pi^2 \sqrt{2}} \Vert u \rho_k \Vert_{K_{\mu, k}}^2 \right\vert \\
            \lesssim \ & \left\Vert y(T; y_0, u) - y(T; y_0, 0) \right\Vert_{H^{-1}}^2 + \mathscr{R}(T; y_0, u)  \\
            & + \Vert u \Vert_{\widetilde{H}^{- 1 - \frac{s_0}{2}}} \left( \left\Vert y(T; y_0, u) - y(T; y_0, 0) \right\Vert_{H^{-1}} + \mathscr{R}(T; y_0, u)^{\frac{1}{2}} \right).
        \end{split}
    \end{equation}
    As in the proof of Theorem \ref{thm:main:kernel_implies_drift_small_time}, one has 
    \begin{equation}
        \mathscr{R}(T; y_0, u) \lesssim \delta \Vert u \Vert_{\widetilde{H}^{-1 - \frac{s_0}{2}}}^2 + \delta \Vert y_0 \Vert_{L^2},
    \end{equation}
    where Lemma \ref{lem:appendix:weak_norm_primitive_new} has been used with a constant depending on $T$. Together with \eqref{eq:proof:thm:main:kernel_implies_drift_finite_time:step_2_eq_1}, this gives \eqref{eq:thm:main:kernel_implies_drift_finite_time_3}.
\end{proof}

As in Corollary \ref{cor:obstruction_small_time}, we show that the drift estimate of Theorem \ref{thm:main:kernel_implies_drift_finite_time} yields an obstruction to finite-time controllability.

\begin{corollary}\label{cor:obstruction_finite_time}
    Under the assumptions of Theorem \ref{thm:main:kernel_implies_drift_finite_time}, the Burgers equation \eqref{eq:Burgers_intro} is not FTLC. More precisely, for all $T > 0$, there exist $c > 0$, $\eta_1 > 0$ and $\delta_1 > 0$, such that for all $u \in L^2(0,T) \cap \widetilde{H}^{\sigma(s) + s_0 - s}(0,T)$ satisfying $\Vert u \Vert_{\widetilde{H}^{\sigma(s) + s_0 - s}(0,T)} \leq \delta_1$, and all $\eta \in (0, \eta_1)$, one has 
    \begin{equation}\label{eq:cor:obstruction_finite_time}
        \left\Vert y\!\left(T; \eta \widetilde{y_0}, u \right) \right\Vert_{H^{-1}(0,1)} + \left\Vert y\!\left(T; \eta \widetilde{y_0}, u \right) \right\Vert_{H^{-1}(0,1)}^2 \geq c \left( \eta + \Vert u \Vert_{\widetilde{H}^{- 1 - \frac{s_0}{2}}(0,T)}^2 \right) > 0,
    \end{equation}
    where $\widetilde{y_0} := \frac{\alpha}{\vert \alpha\vert} \varphi_k$.
\end{corollary}

\begin{proof}
    In this proof, the symbol $\lesssim$ is used for constants that may depend on $T$, $k$, $\mu$, $s$ and $s_0$. Using Theorem \ref{thm:main:kernel_implies_drift_finite_time} and the estimate for the free Burgers equation proved in Lemma \ref{lem:free_evolution_Burgers}, one finds
    \begin{equation}
        \begin{split}
            & \left\vert \left\langle y(T; y_0, u) - S(T)y_0, \varphi_k \right\rangle - \frac{\alpha}{\vert \alpha \vert} \frac{k e^{-\lambda_k T}}{4 \pi^2 \sqrt{2}} \Vert u \rho_k \Vert_{K_{\mu, k}}^2 \right\vert \\
            \lesssim \ & \left\Vert y(T; y_0, u) - S(T) y_0 \right\Vert_{H^{-1}}^2 + \delta^{\frac{1}{2}} \Vert u \Vert_{\widetilde{H}^{-1-\frac{s_0}{2}}}^2 + \delta^{\frac{1}{2}} \Vert y_0 \Vert_{L^2} \\
                & \quad + \Vert u \Vert_{\widetilde{H}^{- 1 - \frac{s_0}{2}}}  \left\Vert y(T; y_0, u) - S(T) y_0 \right\Vert_{H^{-1}} + \Vert y_0 \Vert_{L^2}^2 \\
            \lesssim \ & \left\Vert y(T; y_0, u) \right\Vert_{H^{-1}} + \left\Vert y(T; y_0, u) \right\Vert_{H^{-1}}^2 
                        + \delta^{\frac{1}{2}} \Vert u \Vert_{\widetilde{H}^{-1-\frac{s_0}{2}}}^2 + \delta^{\frac{1}{2}} \Vert y_0 \Vert_{L^2} + \Vert y_0 \Vert_{L^2}^2,
        \end{split}
    \end{equation}
    where $S$ is the heat semigroup, and where we have also used the estimates $\Vert S(T) y_0 \Vert_{H^{-1}} \lesssim \Vert y_0 \Vert_{H^{-1}} \lesssim \Vert y_0 \Vert_{L^2}$ and $\Vert u \Vert_{\widetilde{H}^{- 1 - \frac{s_0}{2}}} \lesssim \delta \lesssim 1$. We assume that $\alpha > 0$; the case $\alpha < 0$ is similar. Note that by Lemma \ref{lem:appendix:removing_exp_from_sobolev_norm_large_time} and \eqref{eq:proof:thm:main:kernel_implies_drift_finite_time:step_1_eq_2}, one has
    \begin{equation}
        \Vert u \Vert_{\widetilde{H}^{- 1 - \frac{s_0}{2}}} \lesssim \Vert u \rho_k \Vert_{\widetilde{H}^{- 1 - \frac{s_0}{2}}} \lesssim \Vert u \rho_k \Vert_{K_{\mu, k}}.
    \end{equation}
    Note also that if $y_0 = \eta \varphi_k$, then $\left\langle S(T)y_0, \varphi_k \right\rangle = e^{- \lambda_k T} \eta$. Hence, there exists $c, C > 0$ such that $y_0 = \eta \varphi_k$, then
    \begin{equation}
        \left( c - C \delta^{\frac{1}{2}} \right) \Vert u \Vert_{\widetilde{H}^{- 1 - \frac{s_0}{2}}}^2 + \left( c - C \delta^{\frac{1}{2}} - C \eta \right) \eta \leq C \left\Vert y(T; y_0, u) \right\Vert_{H^{-1}} + C \left\Vert y(T; y_0, u) \right\Vert_{H^{-1}}^2.
    \end{equation}
    This gives \eqref{eq:cor:obstruction_finite_time} if $\eta$ and $\delta$ are sufficiently small.
 \end{proof}

\section{Remainder estimates}\label{sec:remainder_estimates}

In this section, we prove the various remainder estimates used to establish the drift estimates in Section \ref{sec:proof_obstructions}. The main task is to derive estimates for the quadratic and cubic remainders, for which we begin with some preliminaries.

For clarity, we first provide fully detailed proofs of remainder estimates in the case $\mu \in \mathcal{H}^s_2(0,1)$ for $s \in [-1,1)$, and then provide sketches of proof for the case $\mu \in \mathcal{H}^{s+\frac{1}{2}}_\infty(0,1)$ for $s \in (-1,0) \cup (0,1)$. For simplicity, we do not treat the case $\mu \in \mathcal{H}^1_2$, which requires a specific argument, nor the cases $\mu \in \mathcal{H}^{s+\frac{1}{2}}_\infty(0,1)$ for $s \in \left\{-1,0,1\right\}$, which typically requires the introduction of Sobolev spaces with logarithmic weights. The proofs of the remainder estimates are relatively straightforward when $\mu \in \mathcal{H}^s_2$ for $s \in [-1,0]$, but considerably more involved when $\mu \in \mathcal{H}^s_2$ for $s \in (0,1)$. 

We then establish the so-called closed-loop estimates, which show that remainder terms of the form $\int_0^T u(t)e^{\lambda t} \dd t$ can be absorbed into the drift, up to a term involving the controlled solution. Finally, we prove the so-called decoupling estimates, which show that the full controlled solution $y(t;y_0,u)$ can be estimated in terms of the controlled solution with zero initial data, $y(t;0,u)$, and of the free evolution, $y(t;y_0,0)$.

\subsection{Preliminaries}

For $s \in [-1, 1)$, recall that $\sigma(s) = - \frac{1+s}{2}$ if $s \leq \frac{1}{2}$, and $\sigma(s) = -1 + \frac{s}{2}$ if $s \geq \frac{1}{2}$. Our goal is to prove remainder estimates that yield an obstruction quantified by $\Vert u \Vert_{\widetilde{H}^{- 1 - \frac{s}{2}}}^2$ (or by $\Vert u_1 \Vert_{\widetilde{H}^{- \frac{s}{2}}}^2$ if $s \geq 0$), with a control such that $\Vert u \Vert_{\widetilde{H}^{\sigma(s)}}$ is small. 

We introduce the symmetric bilinear operator $\mathcal{B}$ defined by
\begin{equation}\label{eq:def:mathcalB}
    \mathcal{B}(f,g)(t) := \int_0^t S(t-\tau) \left[ \partial_x \left( f(\tau) g(\tau) \right) \right] \dd \tau,
\end{equation}
where $S$ is the heat semigroup. Note that Lemma \ref{lem:easy_estimate_L1L2} and Lemma \ref{lem:well-posedness_heat_derivative_L1L2} imply
\begin{equation}\label{eq:standard_estimates_on_B_1}
    \left\Vert \mathcal{B}(f,g) \right\Vert_{L^2(L^2)} + \left\Vert \mathcal{B}(f,g) \right\Vert_{L^\infty(H^{-1})} \leq C \Vert fg \Vert_{L^1(L^2)} \lesssim \Vert f \Vert_{L^2(H^1)} \Vert g \Vert_{L^2(L^2)}, 
\end{equation}
if $f \in L^2((0,T);H_0^1(0,1))$ and $g \in L^2((0,T);L^2(0,1))$, and that Lemma \ref{lem:well-posedness_heat_L1L2}, together with the estimate 
\begin{equation}
    \left\Vert \partial_x(fg) \right\Vert_{L^1(L^2)} \lesssim \left\Vert fg \right\Vert_{L^1(H^1)} \lesssim \left\Vert f \right\Vert_{L^2(H^1)} \left\Vert g \right\Vert_{L^2(H^1)},
\end{equation}
imply 
\begin{equation}\label{eq:standard_estimates_on_B_2}
    \left\Vert \mathcal{B}(f,g) \right\Vert_{L^2(H^1)} + \left\Vert \mathcal{B}(f,g) \right\Vert_{L^\infty(L^2)} \leq C^\prime \Vert f \Vert_{L^2(H^1)} \Vert g \Vert_{L^2(H^1)} ,
\end{equation}
if $f \in L^2((0,T);H_0^1(0,1))$ and $g \in L^2((0,T);H_0^1(0,1))$, where $C, C^\prime > 0$ are absolute constants. 

The proof of cubic remainder estimates when $s > 0$ relies crucially on some bilinear parabolic estimates for some norms of mixed regularity, established in Lemma \ref{lem:parabolic_estimates_anisotropic_sobolev}. In the case $p=2$, we use the following notation: $X^s_2 := \widetilde{H}^{-\frac{s}{2}}((0,T);\mathcal{H}^s_2)$, $Y^s_2 := \widetilde{H}^{\frac{s}{2}}((0,T);\mathcal{H}^{-s}_2)$, and $Z^s_2 := \widetilde{H}^{1+\sigma(s)}((0,T);\mathcal{H}^s_2)$, with
\begin{equation}\label{eq:def_X_s_2}
    \left\Vert y \right\Vert_{X^s_2}^2 := \sum_{n \geq 1} \lambda_n^s \left\Vert t \mapsto \langle y(t), \varphi_n \rangle \right\Vert_{\widetilde{H}^{-\frac{s}{2}}(0,T)}^2,
\end{equation}
and $\left\Vert y \right\Vert_{Y^s_2}$ and $\left\Vert y \right\Vert_{Z^s_2}$ defined analogously. In the case $p=+\infty$, we use the following notation: $X^s_\infty := \widetilde{H}^{-\frac{s}{2}}((0,T);\mathcal{H}^{s+\frac{1}{2}}_\infty)$, $Y^s_\infty := \widetilde{H}^{\frac{s}{2}}((0,T);\mathcal{H}^{-s-\frac{1}{2}}_1)$, and $Z^s_\infty := \widetilde{H}^{1+\sigma(s)}((0,T);\mathcal{H}^{s+\frac{1}{2}}_\infty)$, with
\begin{equation}\label{eq:def_X_s_infty}
    \left\Vert y \right\Vert_{X_\infty^s} := \sup_{n \geq 1} n^{s+\frac{1}{2}} \left\Vert t \mapsto \langle y(t), \varphi_n \rangle \right\Vert_{\widetilde{H}^{-\frac{s}{2}}(0,T)}, \quad 
    \left\Vert y \right\Vert_{Y_\infty^s} := \sum_{n \geq 1} n^{-s-\frac{1}{2}} \left\Vert t \mapsto \langle y(t), \varphi_n \rangle \right\Vert_{\widetilde{H}^{\frac{s}{2}}(0,T)},
\end{equation}
and $\left\Vert y \right\Vert_{Z_\infty^s}$ defined analogously. Heuristically, note that $X_p^s$ and $Y_p^s$ are essentially dual to one another. Moreover, the norm of the control involved in drift estimates is essentially $\left\Vert u_1 \mu \right\Vert_{X_p^s}$, and the norm involved in the smallness condition on the control is essentially $\left\Vert u_1 \mu \right\Vert_{Z_p^s}$.

The following elementary lemma will be used repeatedly.

\begin{lemma}\label{lem:easy_integral_with_exp_lambda}
    Let $r \in \left( - \frac{1}{2}, \frac{1}{2} \right)$ and $r^\prime \in [r - 1, r]$. There exists $C > 0$ such that for all $\lambda > 1$, $T > 0$ and $F \in H^{r^\prime}(\mathbb{R})$ supported in $[0, +\infty)$, one has
    \begin{equation}\label{eq:lem:easy_integral_with_exp_lambda}
        \left\Vert t \mapsto \int_0^t e^{-\lambda(t-\tau)} F(\tau) \dd \tau \right\Vert_{\widetilde{H}^r(0,T)} \leq \frac{C}{\lambda^{r^\prime - r + 1}} \Vert F \Vert_{H^{r^\prime}(\mathbb{R})}.
    \end{equation}
    In addition, for all $r \in \left[ 0, \frac{1}{2} \right)$, there exists $C > 0$ such that for all $\lambda > 1$ and $F \in H^{-r}(\mathbb{R})$ supported in $[0, +\infty)$, one has
    \begin{equation}\label{eq:lem:easy_integral_with_exp_lambda_2}
        \left\vert \int_0^t e^{-\lambda(t-\tau)} F(\tau) \dd \tau \right\vert \leq \frac{C}{\lambda^{\frac{1}{2} - r}} \Vert F \Vert_{H^{-r}(\mathbb{R})},
    \end{equation}
    for all $t \geq 0$. 
\end{lemma}

\begin{proof}
    First, we prove \eqref{eq:lem:easy_integral_with_exp_lambda}. For $t \geq 0$, one has 
    \begin{equation}
        \int_0^t e^{-\lambda(t-\tau)} F(\tau) \dd \tau = \left( \mathbbm{1}_{(0, +\infty)} e^{-\lambda \cdot} \right) \ast F (t),
    \end{equation}
    implying
    \begin{equation}
        \left\Vert t \mapsto \int_0^t e^{-\lambda(t-\tau)} F(\tau) \dd \tau \right\Vert_{\widetilde{H}^r(0,T)}
        \leq C \left\Vert \left( \mathbbm{1}_{(0, +\infty)} e^{-\lambda \cdot} \right) \ast F \right\Vert_{H^r(\mathbb{R})},
    \end{equation}
    by Lemma \ref{lem:product_indicatrice}, for some absolute constant $C>0$, since $\vert r \vert < \frac{1}{2}$. This gives
    \begin{align}
        \left\Vert t \mapsto \int_0^t e^{-\lambda(t-\tau)} F(\tau) \dd \tau \right\Vert_{\widetilde{H}^r(0,T)}^2
        & \leq C \int_{\mathbb{R}} \frac{(1+\omega^2)^r}{\lambda^2+\omega^2} \left\vert \widehat{F}(\omega) \right\vert^2 \dd \omega \label{eq:proof:lem:easy_integral_with_exp_lambda_1} \\
        & \leq C \Vert F \Vert_{H^{r^\prime}(\mathbb{R})}^2 \sup_{\omega \in \mathbb{R}} \frac{(1+\omega^2)^{r-r^\prime}}{\lambda^2+\omega^2} ,
    \end{align}
    which implies \eqref{eq:lem:easy_integral_with_exp_lambda} since $r^\prime \in [r - 1, r]$.
    
    Second, we prove \eqref{eq:lem:easy_integral_with_exp_lambda_2}. Using the inverse Fourier formula and Cauchy-Schwarz, one finds
    \begin{equation}
        \left\vert \int_0^t e^{-\lambda(t-\tau)} F(\tau) \dd \tau \right\vert \leq C \left\Vert \mathbbm{1}_{(0, +\infty)} e^{-\lambda \cdot} \right\Vert_{H^{r}(\mathbb{R})} \Vert F \Vert_{H^{-r}(\mathbb{R})},
    \end{equation}
    for all $t \geq 0$, for some absolute constant $C>0$, since $F$ is supported in $[0, +\infty)$. Since $r \in \left[ 0, \frac{1}{2} \right)$, one has
    \begin{equation}
        \left\Vert \mathbbm{1}_{(0, +\infty)} e^{-\lambda \cdot} \right\Vert_{H^{r}(\mathbb{R})}^2 
        \leq \int_{\mathbb{R}} \frac{(1 + \omega^2)^r}{\lambda^2+\omega^2} \dd \omega 
        \leq \frac{C}{\lambda^{1 - 2 r}},
    \end{equation}
    for some constant $C>0$ depending on $r$, completing the proof of \eqref{eq:lem:easy_integral_with_exp_lambda_2}.
\end{proof}

\subsection{Estimates for the linearized system}

We prove the following estimates for $y_1$.

\begin{lemma}[Estimates for the linearized system]\label{lem:linear_estimates_new}
    Let $T > 0$, $s \in [-1, 1]$, $u \in \widetilde{H}^{-\frac{1+s}{2}}(0,T)$, and $\mu \in \mathcal{H}_2^{s}(0,1)$. Recall that $y_1$ is the solution of \eqref{eq:Burgers_y_1_intro}, and that $u_1(t) = \int_0^t u(\tau) \dd \tau$. There exists a constant $C > 0$ depending only on $\mu$ and $s$, such that
    \begin{equation}\label{eq:lem:linear_estimates_new_1}
        \left\Vert y_1 \right\Vert_{L^2((0,T);H^1)} \leq C \left\Vert u \right\Vert_{\widetilde{H}^{-\frac{1+s}{2}}}, \text{ for any } s \in \left[-1, 1 \right],
    \end{equation}
    \begin{equation}\label{eq:lem:linear_estimates_new_2}
        \left\Vert y_1 \right\Vert_{L^2((0,T);L^2)} \leq C \left\Vert u \right\Vert_{\widetilde{H}^{-1-\frac{s}{2}}} \text{ if } s \in \left[-1, 0 \right],
    \end{equation}
    and
    \begin{equation}\label{eq:lem:linear_estimates_new_3}
        \left\Vert y_1 - u_1 \mu \right\Vert_{L^2((0,T);L^2)} \leq C \left\Vert u_1 \right\Vert_{\widetilde{H}^{- \frac{s}{2}}} \text{ if } s \in \left(0, 1 \right].
    \end{equation}
\end{lemma}

\begin{proof}
    We use the symbol $\lesssim$ for constants that may depend on $s$ only. Set $y_{1,n}(t) := \left\langle y_1(t), \varphi_n \right\rangle$ and $\mu_n := \left\langle \mu, \varphi_n \right\rangle$. We still write $y_{1,n}$, $u$ and $u_1$ for their extension by zero outside $[0,T]$. One has 
    \begin{equation}\label{eq:proof:lem:linear_estimates_new_1}
        y_{1,n}(t) = \mathbbm{1}_{[0, T]}(t) \mu_n \int_0^t e^{-\lambda_n (t - \tau)} u(\tau) \dd \tau = u_1(t) \mu_n - \mu_n \lambda_n \mathbbm{1}_{[0, T]}(t) \int_0^t e^{-\lambda_n (t - \tau)} u_1(\tau) \dd \tau,
    \end{equation}
    by integration by parts. In particular, for $\vert r \vert < \frac{1}{2}$ and $r^\prime \in [r-1, r]$, Lemma \ref{lem:easy_integral_with_exp_lambda} implies
    \begin{equation}\label{eq:proof:lem:linear_estimates_new_2}
        \left\Vert y_{1, n} \right\Vert_{\widetilde{H}^{r}(0,T)} 
        \leq \frac{C \vert \mu_n \vert}{\lambda_n^{r^\prime - r + 1}} \left\Vert u \right\Vert_{\widetilde{H}^{r^\prime}(0,T)} ,
    \end{equation}
    and
    \begin{equation}\label{eq:proof:lem:linear_estimates_new_3}
        \left\Vert y_{1, n} - \mu_n u_1 \right\Vert_{\widetilde{H}^{r}(0,T)} 
        \leq \frac{C \vert \mu_n \vert}{\lambda_n^{r^\prime - r}} \left\Vert u_1 \right\Vert_{\widetilde{H}^{r^\prime}(0,T)},
    \end{equation}
    for some constant $C > 0$ which depends only on $r$. Since $- \frac{1+s}{2} \in [-1, 0]$ for any $s \in [-1, 1]$, \eqref{eq:proof:lem:linear_estimates_new_2} gives
    \begin{equation}
        \left\Vert y_1 \right\Vert_{L^2((0,T);H^1)}^2 \lesssim \sum_{n \geq 1} \lambda_n \left\Vert y_{1, n} \right\Vert_{L^2(0,T)}^2 
        \lesssim \sum_{n \geq 1} \lambda_n^s \mu_n^2 \left\Vert u \right\Vert_{\widetilde{H}^{-\frac{1+s}{2}}}^2 
        = \Vert \mu \Vert_{\mathcal{H}^s_2}^2 \left\Vert u \right\Vert_{\widetilde{H}^{-\frac{1+s}{2}}}^2. 
    \end{equation}
    If $s \leq 0$, then $- 1 - \frac{s}{2} \in [-1, 0]$, and \eqref{eq:proof:lem:linear_estimates_new_2} gives
    \begin{equation}
        \left\Vert y_1 \right\Vert_{L^2((0,T);L^2)}^2 \lesssim \sum_{n \geq 1} \left\Vert y_{1, n} \right\Vert_{L^2(0,T)}^2 
        \lesssim \sum_{n \geq 1} \lambda_n^s \mu_n^2 \left\Vert u \right\Vert_{\widetilde{H}^{-1-\frac{s}{2}}}^2 
        = \Vert \mu \Vert_{\mathcal{H}^s_2}^2 \left\Vert u \right\Vert_{\widetilde{H}^{-1-\frac{s}{2}}}^2. 
    \end{equation}
    If $s > 0$, then $- \frac{s}{2} \in [-1, 0]$, and \eqref{eq:proof:lem:linear_estimates_new_3} gives
    \begin{equation}
        \left\Vert y_1 - u_1 \mu \right\Vert_{L^2((0,T);L^2)}^2 \lesssim \sum_{n \geq 1} \left\Vert y_{1, n} - u_1 \mu_n \right\Vert_{L^2(0,T)}^2 
        \lesssim \sum_{n \geq 1} \lambda_n^s \mu_n^2 \left\Vert u_1 \right\Vert_{\widetilde{H}^{-\frac{s}{2}}}^2 
        = \Vert \mu \Vert_{\mathcal{H}^s_2}^2 \left\Vert u_1 \right\Vert_{\widetilde{H}^{-\frac{s}{2}}}^2. 
    \end{equation}
\end{proof}

In the case $s > 0$, we will also need the following additional estimates.

\begin{lemma}\label{lem:linear_estimates_new_additional}
    Let $T > 0$, $s \in (0,1)$, $u \in \widetilde{H}^{-\frac{1+s}{2}}(0,T)$, and $\mu \in \mathcal{H}_2^{s}(0,1)$. There exists a constant $C > 0$ depending only on $\mu$ and $s$, such that
    \begin{equation}\label{eq:lem:linear_estimates_new_4}
        \left\Vert y_1 \right\Vert_{X_2^s} \leq C \left\Vert u_1 \right\Vert_{\widetilde{H}^{- \frac{s}{2}}}
        \quad\text{ and } \quad
        \left\Vert y_1 \right\Vert_{Z_2^s} \leq C \left\Vert u_1 \right\Vert_{\widetilde{H}^{1+\sigma(s)}}.
    \end{equation}
\end{lemma}

\begin{proof}
    Applying \eqref{eq:proof:lem:linear_estimates_new_3} with $r = r^\prime = - \frac{s}{2} \in \left(- \frac{1}{2}, \frac{1}{2} \right)$, one finds
    \begin{equation}
        \left\Vert y_1 \right\Vert_{\widetilde{H}^{-\frac{s}{2}}((0,T);\mathcal{H}^s_2)}
        \leq \left\Vert y_1 - u_1 \mu \right\Vert_{\widetilde{H}^{-\frac{s}{2}}((0,T);\mathcal{H}^s_2)} + \Vert \mu \Vert_{\mathcal{H}^s_2} \left\Vert u_1 \right\Vert_{\widetilde{H}^{-\frac{s}{2}}}
        \lesssim \Vert \mu \Vert_{\mathcal{H}^s_2} \left\Vert u_1 \right\Vert_{\widetilde{H}^{-\frac{s}{2}}}. 
    \end{equation}
    The proof of the second estimate of \eqref{eq:lem:linear_estimates_new_4} is similar, since $1+\sigma(s) = \frac{1-s}{2} \in \left(- \frac{1}{2}, \frac{1}{2} \right)$ if $s \in \left( 0, \frac{1}{2} \right]$ and  $1+\sigma(s) = \frac{s}{2} \in \left(- \frac{1}{2}, \frac{1}{2} \right)$ if $s \in \left[\frac{1}{2}, 1 \right)$.
\end{proof}

\begin{remark}[Optimality of the exponents at the linear level.]\label{rq:optimal_y_1}
    The following heuristic suggests that the assumption $u \in \widetilde{H}^{-\frac{1+s}{2}}(0,T)$ is optimal for the well-posedness of the linearized system when $\mu \in \mathcal{H}_2^{s}(0,1)$. It also explains why a smallness assumption involving $\Vert u \Vert_{\widetilde{H}^{-\frac{1+s}{2}}}$ seems natural for proving an obstruction to controllability quantified by the $\widetilde{H}^{-1-\frac{s}{2}}$-norm of the control. By definition, $y_1$ is the solution of the heat equation with source term $F := u \mu$. In order to ensure that $y_1 \in L^2((0,T);H_0^1(0,1))$, it is natural to look for a space $X$ such that the map $F \in X \mapsto \partial_x y \in L^2(L^2)$ is continuous, where $y$ is the solution of 
    \begin{equation}
        \left\{
        \begin{array}{lll}
            \partial_t y - \partial_x^2 y = F(t,x)
            & \quad t \in (0, T), & \quad x \in (0, 1), \\
            y(t,x) = 0
            & \quad t \in (0, T),  & \quad x \in \{0, 1\},\\
            y(0,x)=0
            & & \quad x \in (0, 1).
        \end{array}
        \right.
    \end{equation}
    The adjoint of this operator is $G \in L^2(L^2) \mapsto z \in X^\prime$, where $z$ is the solution of
    \begin{equation}
        \left\{
        \begin{array}{lll}
            \partial_t z + \partial_x^2 z = \partial_x G(t,x)
            & \quad t \in (0, T), & \quad x \in (0, 1), \\
            z(t,x) = 0
            & \quad t \in (0, T),  & \quad x \in \{0, 1\},\\
            z(T,x)=0
            & & \quad x \in (0, 1).
        \end{array}
        \right.
    \end{equation}
    For $G \in L^2(L^2)$, Lemma \ref{lem:heat_weak_wellposedness} gives $z \in L^2((0,T); H_0^1(0,1)) \cap H^1((0,T); H^{-1}(0,1))$. By interpolation, this yields $z \in H^{\frac{1-r}{2}}((0,T); H^{r}(0,1))$, for $r \in [-1,1]$. For $r = -s$, the corresponding source space is $X = H^{-\frac{1+s}{2}}((0,T); H^{s}(0,1))$, implying that when $\mu \in \mathcal{H}_2^{s}(0,1)$, the optimal regularity expected for $u$ is $\widetilde{H}^{-\frac{1+s}{2}}(0,T)$. In order to avoid technical difficulties related to interpolation and duality for the spaces $\mathcal{H}_p^{s}$, and since we aim to obtain estimates with constants independent of $T$, we base our proofs solely on Fourier analysis.
\end{remark}

\subsection{Quadratic remainder estimates}
 
We prove the following estimates for the quadratic term $y_2$ and the quadratic remainder $R_2$.

 \begin{proposition}[Estimates on the quadratic expansion and on the quadratic remainder]\label{prop:quad_estimates_new}
    Let $s \in \left[-1, 1 \right)$, $T>0$, $\mu \in \mathcal{H}^s_2(0,1)$, and $u \in L^2(0,T)$. Recall that $y_1$ and $y_2$ are the solutions of \eqref{eq:Burgers_y_1_intro} and \eqref{eq:Burgers_y_2_intro}, and set $R_2(t) := y(t;0,u) - y_1(t;u)$. 
    First, there exists $C > 0$ which may depend on $s$ and $\mu$ only, such that
    \begin{equation}\label{eq:prop:quad_estimates_new_0}
        \Vert y_2 \Vert_{L^2((0, T);H^1)} + \Vert y_2 \Vert_{L^\infty([0, T];L^2)} \leq C \Vert u \Vert_{\widetilde{H}^{\sigma(s)}}^2,
    \end{equation}
    \begin{equation}\label{eq:prop:quad_estimates_new_1}
        \Vert y_2 \Vert_{L^2((0, T);L^2)} + \Vert y_2 \Vert_{L^\infty([0, T];H^{-1})} \leq C \Vert u \Vert_{\widetilde{H}^{-1-\frac{s}{2}}} \Vert u \Vert_{\widetilde{H}^{\sigma(s)}} \text{ if } s \leq 0,
    \end{equation}
    and
    \begin{equation}\label{eq:prop:quad_estimates_new_2}
        \Vert y_2 \Vert_{L^2((0, T);L^2)} + \Vert y_2 \Vert_{L^\infty([0, T];H^{-1})} \leq C \Vert u_1 \Vert_{\widetilde{H}^{-\frac{s}{2}}} \left( \Vert u \Vert_{\widetilde{H}^{\sigma(s)}} + \Vert u_1 \Vert_{\widetilde{H}^{1+\sigma(s)}} \right) \text{ if } s > 0.
    \end{equation}
    Second, there exist $C, \delta > 0$ which may depend on $s$ and $\mu$ only, such that if $\Vert u \Vert_{\widetilde{H}^{\sigma(s)}(0,T)} \leq \delta$, then the above estimates hold true with $y_2$ replaced by $R_2$.
 \end{proposition}
 
 \begin{proof}
    In this proof, $\lesssim$ is used for constants which may depend on $s$ and $\mu$ only. By definition, one has 
    \begin{equation}
        \partial_t R_2 - \partial_x^2 R_2 = - \frac{1}{2} \partial_x \left( (y_1 + R_2)^2 \right),
    \end{equation}
    and so
    \begin{equation}\label{eq:proof:prop:quad_estimate_new_02}
        R_2 = y_2 - \mathcal{B}(y_1, R_2) - \frac{1}{2} \mathcal{B}(R_2, R_2). 
    \end{equation}
    Since $y_2 = - \frac{1}{2} \mathcal{B}(y_1, y_1)$, \eqref{eq:prop:quad_estimates_new_0} follows from \eqref{eq:standard_estimates_on_B_2} and \eqref{eq:lem:linear_estimates_new_1}, and \eqref{eq:prop:quad_estimates_new_1} follows from \eqref{eq:standard_estimates_on_B_1}, \eqref{eq:lem:linear_estimates_new_1} and \eqref{eq:lem:linear_estimates_new_2}. 
    
    \textbf{Step 1: proof of \eqref{eq:prop:quad_estimates_new_2}.}
    The estimate of $\Vert y_2 \Vert_{L^2((0, T);L^2)}$ follows immediately from Lemma \ref{lem:linear_estimates_new_additional} and Lemma \ref{lem:parabolic_estimates_anisotropic_sobolev}, which imply
    \begin{equation}
        \Vert y_2 \Vert_{L^2((0, T);L^2)} 
        \lesssim \left\Vert y_1 \right\Vert_{\widetilde{H}^{- \frac{s}{2}}((0,T);\mathcal{H}^s_2)} \left\Vert y_1 \right\Vert_{\widetilde{H}^{1+\sigma(s)}((0,T);\mathcal{H}^s_2)}
        \lesssim \left\Vert u_1 \right\Vert_{\widetilde{H}^{- \frac{s}{2}}} \left\Vert u_1 \right\Vert_{\widetilde{H}^{1+\sigma(s)}}.
    \end{equation}
    The estimate of $\Vert y_2 \Vert_{L^\infty([0, T];H^{-1})}$ requires a specific argument. Write $z_1 := y_1 - u_1 \mu$, so that
    \begin{equation}\label{eq:proof:prop:quad_estimate_new_01}
        y_2 = - \frac{1}{2} \mathcal{B}(y_1, y_1) = - \frac{1}{2} \mathcal{B}(u_1\mu, u_1 \mu) - \frac{1}{2} \mathcal{B}(u_1 \mu, z_1) - \frac{1}{2} \mathcal{B}(z_1 , y_1).
    \end{equation}
    We estimate the three terms on the right-hand side of \eqref{eq:proof:prop:quad_estimate_new_01} separately.
    
    \underline{Estimate of the third term.} 
    By \eqref{eq:standard_estimates_on_B_1}, \eqref{eq:lem:linear_estimates_new_1} and \eqref{eq:lem:linear_estimates_new_3}, one has
    \begin{equation}\label{eq:proof:prop:quad_estimate_new_05}
        \left\Vert \mathcal{B}(z_1,y_1) \right\Vert_{L^\infty(H^{-1})} 
        \lesssim \Vert z_1 \Vert_{L^2(L^2)} \Vert y_1 \Vert_{L^2(H^1)} 
        \lesssim \Vert u_1 \Vert_{\widetilde{H}^{-\frac{s}{2}}} \Vert u \Vert_{\widetilde{H}^{-\frac{1+s}{2}}}
        \lesssim \Vert u_1 \Vert_{\widetilde{H}^{-\frac{s}{2}}} \Vert u \Vert_{\widetilde{H}^{\sigma(s)}}.
    \end{equation}
    
    \underline{Estimate of the second term.} 
    We prove that
    \begin{equation}\label{eq:proof:prop:quad_estimate_new_06}
        \left\Vert \mathcal{B}(u_1\mu,z_1) \right\Vert_{L^\infty([0,T];H^{-1})} 
        \lesssim \Vert u_1 \Vert_{\widetilde{H}^{-\frac{s}{2}}} \Vert u_1 \Vert_{\widetilde{H}^{1 + \sigma(s)}}.
    \end{equation}
    Let $t_0 \in [0,T]$. Set $p := \frac{2}{1 - 2(1+\sigma(s))}$. Note that $1+\sigma(s) \in \left( 0, \frac{1}{2} \right)$. We use the Sobolev embedding $\Vert u_1 \Vert_{L^{p}(\mathbb{R})} \lesssim \Vert u_1 \Vert_{\widetilde{H}^{1+\sigma(s)}(0,T)}$, where $u_1$ is extended by zero outside $[0,T]$. Using Hölder's inequality, with $\left( 1+\sigma(s) \right) + \frac{1}{p} + \frac{1}{2} = 1$, one finds
    \begin{equation}
    \begin{split}
        \left\vert \int_0^{t_0} e^{-\lambda_n (t_0-t)} u_1(t) \langle \mu z_1(t), \varphi_n^\prime \rangle \dd t \right\vert^2 
        & \leq \left\Vert e^{-\lambda_n \cdot} \right\Vert_{L^{\frac{1}{1+\sigma(s)}}(0, +\infty)}^2 \Vert u_1 \Vert_{L^p(\mathbb{R})}^2 \int_0^T \langle \mu z_1(t), \varphi_n^\prime \rangle^2 \dd t \\
        & \lesssim \frac{1}{\lambda_n^{2 + 2\sigma(s)}} \Vert u_1 \Vert_{\widetilde{H}^{1+\sigma(s)}}^2 \int_0^T \langle \mu z_1(t), \varphi_n^\prime \rangle^2 \dd t ,
    \end{split}
    \end{equation}
    yielding
    \begin{equation}\label{eq:proof:prop:quad_estimate_new_07}
        \begin{split}
            \Vert \mathcal{B}(u_1\mu,z_1)(t_0) \Vert_{H^{-1}}^2 
            & = \sum_{n\geq 1} \frac{1}{\lambda_n} \left\vert \int_0^{t_0} e^{-\lambda_n (t_0-t)} u_1(t) \langle \mu z_1(t), \varphi_n^\prime \rangle \dd t \right\vert^2 \\
            & \lesssim \Vert u_1 \Vert_{\widetilde{H}^{1+\sigma(s)}}^2 \sum_{n \geq 1} \frac{1}{\lambda_n^{2 + 2\sigma(s)}} \int_0^T \left\langle \mu z_1(t), \frac{\varphi_n^\prime}{n} \right\rangle^2 \dd t \\
            & \lesssim \Vert u_1 \Vert_{\widetilde{H}^{1+\sigma(s)}}^2 \int_0^T \left\Vert \mu z_1(t) \right\Vert_{H^{-2 - 2\sigma(s)}(0,1)}^2 \dd t ,
        \end{split}
    \end{equation}
    where we have used the fact that the $H^{-2 - 2\sigma(s)}(0,1)$-norm is greater than the $\ell^2(\mathbb{N}^\ast)$-norm of $\left( n^{-2 - 2\sigma(s)} \left\langle \cdot, \frac{\varphi_n^\prime}{n} \right\rangle \right)_{n\geq 1}$. By Lemma \ref{lem:product_sobolev_classical}, one has 
    \begin{equation}\label{eq:proof:prop:quad_estimate_new_07_bis}
        \left\Vert \mu z_1(t) \right\Vert_{H^{-2 - 2\sigma(s)}(0,1)} \lesssim \left\Vert \mu \right\Vert_{H^{s}(0,1)} \left\Vert z_1(t) \right\Vert_{L^2}    .
    \end{equation}
    Indeed, if $s \in \left(0, \frac{1}{2} \right)$ then one has $-2 - 2\sigma(s) = s-1$, and hence \eqref{eq:proof:prop:quad_estimate_new_07_bis} follows from $s \geq s-1$, $0\geq s-1$, $s - \frac{1}{2} > s-1$ and $s+0 > 0$, and if $s \in \left[\frac{1}{2}, 1 \right)$ then one has $-2 - 2\sigma(s) = -s$, and hence \eqref{eq:proof:prop:quad_estimate_new_07_bis} follows from $s \geq -s$, $0 \geq -s$, $s - \frac{1}{2} > -s$ and $s+0 > 0$.
    Hence, one obtains
    \begin{equation}
        \Vert \mathcal{B}(u_1\mu,z_1)(t_0) \Vert_{H^{-1}}^2 
        \lesssim \Vert u_1 \Vert_{\widetilde{H}^{1+\sigma(s)}}^2 \Vert z_1 \Vert_{L^2(L^2)}^2
        \lesssim \Vert u_1 \Vert_{\widetilde{H}^{1+\sigma(s)}}^2 \Vert u_1 \Vert_{\widetilde{H}^{-\frac{s}{2}}}^2,
    \end{equation}
    by \eqref{eq:lem:linear_estimates_new_3}.

    \underline{Estimate of the first term.} 
    We prove that
    \begin{equation}\label{eq:proof:prop:quad_estimate_new_08}
        \left\Vert \mathcal{B}(u_1\mu,u_1\mu) \right\Vert_{L^\infty(H^{-1})} 
        \lesssim \Vert u_1 \Vert_{\widetilde{H}^{-\frac{s}{2}}} \Vert u_1 \Vert_{\widetilde{H}^{1 + \sigma(s)}}.
    \end{equation}

    We start with the case $s \in \left[ \frac{1}{2}, 1 \right)$. Using \eqref{eq:standard_estimates_on_B_1}, one finds
    \begin{equation}
        \left\Vert \mathcal{B}(u_1\mu,u_1\mu) \right\Vert_{L^\infty(H^{-1})} 
        \lesssim \Vert u_1 \Vert_{L^2(0,T)}^2 \Vert \mu \Vert_{L^4(0,1)}^2.
    \end{equation}
    Since $s \geq \frac{1}{2} \geq \frac{1}{4}$, one has $\mu \in L^4(0,1)$. Using the Plancherel identity and Cauchy-Schwarz, one finds 
    \begin{equation}
        \Vert u_1 \Vert_{L^2(0,T)}^2 \lesssim \Vert u_1 \Vert_{\widetilde{H}^{-\frac{s}{2}}} \Vert u_1 \Vert_{\widetilde{H}^{+\frac{s}{2}}} = \Vert u_1 \Vert_{\widetilde{H}^{-\frac{s}{2}}} \Vert u_1 \Vert_{\widetilde{H}^{1 + \sigma(s)}}.
    \end{equation}
    Note that this argument in fact covers the case $s \in \left[ \frac{1}{4}, 1 \right)$.
    
    Now, we prove \eqref{eq:proof:prop:quad_estimate_new_08} when $s \in \left( 0, \frac{1}{2} \right)$.  Let $t_0 \in [0, T]$. By definition, one has
    \begin{equation}
        \left\Vert \mathcal{B}(u_1\mu,u_1\mu)(t_0) \right\Vert_{H^{-1}}^2 = \sum_{n\geq 1} \frac{1}{\lambda_n} \langle \mu^2, \varphi_n^\prime \rangle^2 \left\vert \int_0^{t_0} e^{-\lambda_n (t_0-t)} u_1(t)^2 \dd t \right\vert^2 .
    \end{equation}
    Let $n \geq 1$. Let $\alpha$ be such that $s < \alpha < \min \left( \frac{1}{4} + s, \frac{1}{2} \right)$. Recall that $u_1$ is extended by zero outside $[0,T]$. Lemma \ref{lem:easy_integral_with_exp_lambda} implies
    \begin{equation}
        \left\Vert \mathcal{B}(u_1\mu,u_1\mu)(t_0) \right\Vert_{H^{-1}}^2 
        \lesssim \sum_{n\geq 1} \frac{1}{\lambda_n^{2-2\alpha}} \langle \mu^2, \varphi_n^\prime \rangle^2 \left\Vert (u_1)^2 \right\Vert_{H^{-\alpha}(\mathbb{R})}^2.
    \end{equation}
    By Lemma \ref{lem:product_sobolev_classical}, one has
    \begin{equation}
        \left\Vert (u_1)^2 \right\Vert_{H^{-\alpha}(\mathbb{R})} \lesssim \left\Vert u_1 \right\Vert_{H^{-\frac{s}{2}}(\mathbb{R})} \left\Vert u_1 \right\Vert_{H^{1+\sigma(s)}(\mathbb{R})}.
    \end{equation}
    Indeed, since $s \in \left( 0, \frac{1}{2} \right)$, one has $1+\sigma(s) = \frac{1-s}{2}$, and Lemma \ref{lem:product_sobolev_classical} can be applied with $-\frac{s}{2} \geq -\alpha$, $\frac{1-s}{2} \geq - \alpha$, $-\frac{s}{2} + \frac{1-s}{2} - \frac{1}{2} > -\alpha$ and $-\frac{s}{2} + \frac{1-s}{2} > 0$. Hence, one obtains 
    \begin{equation}
        \begin{split}
            \left\Vert \mathcal{B}(u_1\mu,u_1\mu)(t_0) \right\Vert_{H^{-1}}^2 
            & \lesssim \left\Vert u_1 \right\Vert_{H^{-\frac{s}{2}}(\mathbb{R})}^2 \left\Vert u_1 \right\Vert_{H^{1+\sigma(s)}(\mathbb{R})}^2 
               \sum_{n\geq 1} \frac{1}{\lambda_n^{2-2\alpha}} \langle \mu^2, \varphi_n^\prime \rangle^2 \\
            & \lesssim  \Vert u_1 \Vert_{\widetilde{H}^{-\frac{s}{2}}}^2 \Vert u_1 \Vert_{\widetilde{H}^{1+\sigma(s)}}^2 \Vert \mu^2 \Vert_{H^{2\alpha-1}(0,1)}^2 ,
        \end{split}
    \end{equation}
    where we have used the fact that the $H^{2\alpha-1}(0,1)$-norm is greater than the $\ell^2(\mathbb{N}^\ast)$-norm of 
    \begin{equation}
        \left( n^{2\alpha-1} \left\langle \cdot, \frac{\varphi_n^\prime}{n} \right\rangle \right)_{n\geq 1}.    
    \end{equation}
    Since $s < \frac{1}{2}$, one has $\mathcal{H}_2^s(0,1) = H^s(0,1)$, with equivalent norms. Hence, we can use Lemma \ref{lem:product_sobolev_classical}: one has $s \geq 2\alpha-1$, $s+s - \frac{1}{2} > 2\alpha-1$ and $s+s > 0$, implying $\mu^2 \in H^{2\alpha-1}(0,1)$, which yields
    \begin{equation}
        \left\Vert \mathcal{B}(u_1\mu,u_1\mu)(t_0) \right\Vert_{H^{-1}}^2 \lesssim \Vert u_1 \Vert_{\widetilde{H}^{-\frac{s}{2}}}^2 \Vert u_1 \Vert_{\widetilde{H}^{1+\sigma(s)}}^2.
    \end{equation}
 
    \textbf{Step 2: estimate of $R_2$ in $L^2((0,T);L^2)$ and $L^\infty((0,T);H^1)$.}
    We prove that if $\Vert u \Vert_{\widetilde{H}^{\sigma(s)}}$ is sufficiently small, then
    \begin{equation}\label{eq:proof:prop:quad_estimate_new_09}
        \left\Vert R_2 \right\Vert_{L^2(H^1)} + \left\Vert R_2 \right\Vert_{L^\infty(L^2)} \lesssim \Vert u \Vert_{\widetilde{H}^{\sigma(s)}(0,T)}^2.
    \end{equation}
    For $t \in [0, T]$, set $\Vert f \Vert_{E_t} := \Vert f \Vert_{L^\infty((0,t),L^2)} + \Vert f \Vert_{L^2((0,t),H^1)}$. Using \eqref{eq:proof:prop:quad_estimate_new_02}, one finds
    \begin{equation}
        \left\Vert R_2 \right\Vert_{E_t} \lesssim \left\Vert y_2 \right\Vert_{E_t} + \left\Vert \mathcal{B}(y_1, R_2) \right\Vert_{E_t} + \left\Vert \mathcal{B}(R_2, R_2) \right\Vert_{E_t}.
    \end{equation} 
    By \eqref{eq:standard_estimates_on_B_2}, \eqref{eq:prop:quad_estimates_new_0} and \eqref{eq:lem:linear_estimates_new_1}, this gives
    \begin{equation}
        \left\Vert R_2 \right\Vert_{E_t} \lesssim C \Vert u \Vert_{\widetilde{H}^{\sigma(s)}(0,T)}^2 + C \left( \Vert u \Vert_{\widetilde{H}^{\sigma(s)}(0,T)} + \left\Vert R_2 \right\Vert_{E_t} \right) \left\Vert R_2 \right\Vert_{E_t},
    \end{equation}
    for some $C>0$ independent of $t$. Let $\delta \in \left( 0, \frac{1}{4C} \right)$. Assume that $\Vert u \Vert_{\widetilde{H}^{\sigma(s)}(0,T)} \leq \delta$. It implies
    \begin{equation}
        \frac{3}{4} \left\Vert R_2 \right\Vert_{E_t} \lesssim C \Vert u \Vert_{\widetilde{H}^{\sigma(s)}(0,T)}^2 + C \left\Vert R_2 \right\Vert_{E_t}^2.
    \end{equation}
    In particular, if $\left\Vert R_2 \right\Vert_{E_t} = 4 C \Vert u \Vert_{\widetilde{H}^{\sigma(s)}(0,T)}^2$, then
    \begin{equation}
        2 C \Vert u \Vert_{\widetilde{H}^{\sigma(s)}(0,T)}^2 \leq 16 C^3 \Vert u \Vert_{\widetilde{H}^{\sigma(s)}(0,T)}^4,
    \end{equation}
    an inequality which is false if $\Vert u \Vert_{\widetilde{H}^{\sigma(s)}(0,T)}$ is sufficiently small. Hence, there exists $\delta^\prime \in (0, \delta)$ such that if $\Vert u \Vert_{\widetilde{H}^{\sigma(s)}(0,T)} \leq \delta^\prime$, then for all $t \in [0, T]$, one has  $\left\Vert R_2 \right\Vert_{E_t} \neq 4 C \Vert u \Vert_{\widetilde{H}^{\sigma(s)}(0,T)}^2$. By the intermediate value theorem, applied to $t \mapsto \left\Vert R_2 \right\Vert_{E_t}$, and since $\left\Vert R_2 \right\Vert_{E_t}=0$ when $t = 0$, one has $\left\Vert R_2 \right\Vert_{E_t} \leq 4 C \Vert u \Vert_{\widetilde{H}^{\sigma(s)}(0,T)}^2$ for all $t \in [0, T]$. Note that this argument is valid since $R_2 \in C^0([0,T], L^2(0,1))$. This proves \eqref{eq:proof:prop:quad_estimate_new_09}.
    
    \textbf{Step 3: proof of the estimate for $R_2$.}
    Using \eqref{eq:proof:prop:quad_estimate_new_02} and \eqref{eq:standard_estimates_on_B_1}, one finds
    \begin{equation}
        \begin{split}
            \Vert R_2 \Vert_{L^2(L^2)} + \Vert R_2 \Vert_{L^\infty(H^{-1})} 
            & \lesssim \Vert y_2 \Vert_{L^2(L^2)} + \Vert y_2 \Vert_{L^\infty(H^{-1})} + \Vert y_1 \Vert_{L^2(H^1)} \Vert R_2 \Vert_{L^2(L^2)} \\
            & \quad + \Vert R_2 \Vert_{L^2(H^1)} \Vert R_2 \Vert_{L^2(L^2)} .
        \end{split}
    \end{equation}
    By \eqref{eq:prop:quad_estimates_new_1}, \eqref{eq:prop:quad_estimates_new_2}, \eqref{eq:lem:linear_estimates_new_1} and \eqref{eq:proof:prop:quad_estimate_new_09}, there exists $\delta > 0$ such that if $ \Vert u \Vert_{\widetilde{H}^{\sigma(s)}} \leq \delta$, then
    \begin{equation}
        \Vert R_2 \Vert_{L^2(L^2)} + \Vert R_2 \Vert_{L^\infty(H^{-1})} 
        \lesssim \Vert u \Vert_{\widetilde{H}^{-1-\frac{s}{2}}} \Vert u \Vert_{\widetilde{H}^{\sigma(s)}} + \delta \Vert R_2 \Vert_{L^2(L^2)}, \text{ if } s \leq 0,
    \end{equation}
    and
    \begin{equation}
        \Vert R_2 \Vert_{L^2(L^2)} + \Vert R_2 \Vert_{L^\infty(H^{-1})} 
        \lesssim \Vert u_1 \Vert_{\widetilde{H}^{-\frac{s}{2}}} \left( \Vert u_1 \Vert_{\widetilde{H}^{1+\sigma(s)}} + \Vert u \Vert_{\widetilde{H}^{\sigma(s)}}  \right) + \delta \Vert R_2 \Vert_{L^2(L^2)}, \text{ if } s > 0.
    \end{equation}
    Up to reducing $\delta$, this gives the claimed estimate for $R_2$.
 \end{proof}

In the case $s > 0$, the following additional estimates will be used to estimate $R_3$.

\begin{lemma}\label{lem:quad_estimates_additional_anisotropic_estimates}
    Let $T > 0$, $s \in (0,1)$, $u \in L^2(0,T)$, and $\mu \in \mathcal{H}_2^{s}(0,1)$. There exists a constant $C > 0$ depending only on $\mu$ and $s$, such that 
    \begin{equation}\label{eq:lem:y_2_estimates_anisotropic_sobolev_2}
        \left\Vert y_2 \right\Vert_{Z_2^s} \leq C \left\Vert u_1 \right\Vert_{\widetilde{H}^{1+\sigma(s)}(0,T)}^2,
    \end{equation}
    and
    \begin{equation}\label{eq:lem:y_2_estimates_anisotropic_sobolev_1}
        \left\Vert y_2 \right\Vert_{X_2^s} + \left\Vert y_2 \right\Vert_{Y_2^s}
        \leq C \left\Vert u_1 \right\Vert_{\widetilde{H}^{- \frac{s}{2}}(0,T)} \left\Vert u_1 \right\Vert_{\widetilde{H}^{1+\sigma(s)}(0,T)}.
    \end{equation}
\end{lemma}

\begin{proof}
    Since $y_2 = - \frac{1}{2} \mathcal{B}(y_1, y_1)$, \eqref{eq:lem:y_2_estimates_anisotropic_sobolev_2} and \eqref{eq:lem:y_2_estimates_anisotropic_sobolev_1} follow from Lemma \ref{lem:parabolic_estimates_anisotropic_sobolev} and Lemma \ref{lem:linear_estimates_new_additional}.
\end{proof}

\begin{remark}[Optimality of the exponents at the quadratic level.]\label{rq:optimal_y_2}
    We explain why the exponent $\sigma(s)$ defined in \eqref{eq:defn_sigma(s)} is expected to be optimal for quadratic remainder estimates, and thus for our main obstruction results. By definition, $\sigma(s) = \max\left(- \frac{1+s}{2}, -1+\frac{s}{2} \right)$. The exponent $- \frac{1+s}{2}$ is the natural threshold for the well-posedness of the linearized system when $\mu \in \mathcal{H}_2^s$ (see Remark \ref{rq:optimal_y_1}). We explain below why the second exponent, $-1+\frac{s}{2}$, arises naturally at the quadratic level. For simplicity, we present the argument as a heuristic remark, although the computations could be made rigorous. Let $\mu \in \mathcal{H}_2^s$, with $s > \frac{1}{2}$. Assume that $\mu^2$ is not constant, or equivalently that there exists $n \geq 1$ such that $\langle (\mu^2)^\prime, \varphi_n \rangle \neq 0$. Using the approximation $y_1 \approx u_1 \mu$ (see \eqref{eq:lem:linear_estimates_new_3}), one finds
    \begin{equation}
        \left\langle y_2(T), \varphi_n \right\rangle \approx - \frac{1}{2} \langle (\mu^2)^\prime, \varphi_n \rangle \int_0^T e^{-\lambda_n (T-t)} \left( u_1(t) \right)^2 \dd t.
    \end{equation}
    Let $\chi \in C^\infty_\mathrm{c}(0,T)$ be a nonzero cutoff, and let $N \geq 1$. We choose $u := \left( t \mapsto \chi(t) \cos(N t) \right)^\prime$. Then 
    \begin{equation}
        \int_0^T e^{-\lambda_n (T-t)} \left( u_1(t) \right)^2 \dd t = \frac{1}{2} \int_0^T e^{-\lambda_n (T-t)} \chi(t)^2 \left( 1 + \cos(2Nt) \right) \dd t \longrightarrow c_0,
    \end{equation}
    for some $c_0 \neq 0$, when $N \rightarrow + \infty$. On the other hand, for all $r < \frac{1}{2}$, one has
    \begin{equation}
        \left\Vert u_1 \right\Vert_{\widetilde{H}^{r}(0,T)}^2 =  \frac{1}{4} \int_{\mathbb{R}} (1+\omega^2)^r \left\vert \widehat{\chi}(\omega+N) + \widehat{\chi}(\omega-N) \right\vert^2 \dd \omega \sim c_1 N^{2r},
    \end{equation}
     when $N \rightarrow + \infty$, for some $c_1 > 0$ which depends on $r$ and $\chi$. Hence, an estimate of the form 
     \begin{equation}\label{eq:rq:optimal_y_2}
         \left\Vert y_2(T) \right\Vert_{H^{-1}} \leq C \Vert u_1 \Vert_{\widetilde{H}^{-\frac{s}{2}}} \Vert u_1 \Vert_{\widetilde{H}^{1+\sigma}} 
     \end{equation}
    implies that for all $N$ sufficiently large, one has $c \leq N^{-\frac{s}{2}} N^{1 + \sigma}$, where $c, C > 0$ are constants independent of $N$. Hence, an estimate of the form \eqref{eq:rq:optimal_y_2} can only hold if $\sigma \geq -1+\frac{s}{2}$.
\end{remark}
 
\subsection{Cubic remainder estimates}
 
Finally, we prove the following estimates for the cubic remainder $R_3$.

\begin{proposition}[Estimates on the cubic remainder]\label{prop:cubic_estimates_new}
    Let $s \in \left[ -1, 1 \right)$, $T>0$ and $k \geq 1$. Let $\mu \in \mathcal{H}^s_2(0,1)$, and $u \in L^2(0,T)$. Set $R_3 := y - y_1 - y_2$, where $y$, $y_1$ and $y_2$ are the solutions of \eqref{eq:Burgers_intro} with $y_0 = 0$, \eqref{eq:Burgers_y_1_intro} and \eqref{eq:Burgers_y_2_intro}, respectively. There exist $C, \delta > 0$ which may depend on $s$, $k$ and $\mu$ only, such that if
    \begin{equation}
        \Vert u \Vert_{\widetilde{H}^{\sigma(s)}(0,T)} \leq \delta \text{ if } s \leq 0, \quad \text{ and } \quad \Vert u_1 \Vert_{\widetilde{H}^{1+\sigma(s)}(0,T)} \leq \delta \text{ if } s > 0,
    \end{equation}
    then
    \begin{equation}\label{eq:prop:cubic_estimates_new_1}
        \left\vert \left\langle R_3(T), \varphi_k \right\rangle \right\vert \leq C \Vert u \Vert_{\widetilde{H}^{-1-\frac{s}{2}}}^2 \Vert u \Vert_{\widetilde{H}^{\sigma(s)}} \text{ if } s \leq 0,
    \end{equation}
    and
    \begin{equation}\label{eq:prop:cubic_estimates_new_2}
        \left\vert \left\langle R_3(T), \varphi_k \right\rangle \right\vert \leq C \Vert u_1 \Vert_{\widetilde{H}^{-\frac{s}{2}}}^2  \Vert u_1 \Vert_{\widetilde{H}^{1 + \sigma(s)}} \text{ if } s > 0.
    \end{equation}
\end{proposition}

\begin{proof}
    The symbol $\lesssim$ is used for constants that  may depend on $s$, $k$ and $\mu$ only. One has
    \begin{equation}
        \partial_t R_3 - \partial_x^2 R_3 + R_3 \partial_x R_3 + \partial_x \left[ \left( y_1 + y_2 \right) R_3 \right] = - \partial_x \left[ y_1 y_2 + \frac{(y_2)^2}{2} \right],
    \end{equation}   
    or equivalently
    \begin{equation}\label{eq:proof:new_cubic_estimates_1}
        R_3 = - \mathcal{B}(y_1, y_2) - \frac{1}{2} \mathcal{B}(y_2, y_2) - \mathcal{B}(y_1 + y_2, R_3) - \frac{1}{2} \mathcal{B}(R_3, R_3).
    \end{equation}
    Let $\delta \in (0,1)$, that will be chosen small below. 

    \textbf{Step 1: proof of \eqref{eq:prop:cubic_estimates_new_1}.}
    Here, we assume that $s \in [-1, 0]$, and that $\Vert u \Vert_{\widetilde{H}^{\sigma(s)}(0,T)} \leq \delta$. We will need the estimate
    \begin{equation}\label{eq:proof:new_cubic_estimates_step_1_1}
        \Vert y_2 \Vert_{L^2(H^1)} \lesssim \left\Vert \partial_x \left( (y_1)^2 \right) \right\Vert_{L^1(L^2)} 
        \lesssim \Vert y_1 \Vert_{L^2(H^1)}^2
        \lesssim \Vert u \Vert_{\widetilde{H}^{\sigma(s)}}^2,  
    \end{equation}
    where we have used Lemma \ref{lem:well-posedness_heat_L1L2}, \eqref{eq:lem:linear_estimates_new_1}, and the fact that $H^1(0,1)$ is an algebra. Set $f = - y_1 y_2 - \frac{(y_2)^2}{2}$ and $g = y_1 + y_2$, so that $R_3$ is the solution of a forced Burgers equation in the sense of Lemma \ref{lem:forced_burgers}. In particular, if $\Vert f \Vert_{L^1(H^1)}$ and $\Vert g \Vert_{L^2(H^1)}$ are sufficiently small, then the last estimate of Lemma \ref{lem:forced_burgers} gives
    \begin{equation}\label{eq:proof:new_cubic_estimates_step_1_2}
        \left\vert \left\langle R_3(T), \varphi_k \right\rangle \right\vert 
        \lesssim \Vert f \Vert_{L^1(L^1)} + \left( \Vert g \Vert_{L^2(L^2)} + \Vert f \Vert_{L^1(L^2)} \right) \Vert f \Vert_{L^1(L^2)}.
    \end{equation}
    
    On the one hand, since $H^1(0,1)$ is an algebra, one has
    \begin{equation}
        \Vert f \Vert_{L^1(H^1)} + \Vert g \Vert_{L^2(H^1)} \lesssim \Vert y_1 \Vert_{L^2(H^1)} \Vert y_2 \Vert_{L^2(H^1)} + \Vert y_2 \Vert_{L^2(H^1)}^2 + \Vert y_1 \Vert_{L^2(H^1)} + \Vert y_2 \Vert_{L^2(H^1)} \lesssim \delta,
    \end{equation}
    by \eqref{eq:lem:linear_estimates_new_1} and \eqref{eq:proof:new_cubic_estimates_step_1_1}. Hence, Lemma \ref{lem:forced_burgers} can be applied if $\delta$ is sufficiently small.
    
    On the other hand, using the embedding $H^1(0,1) \subset L^\infty(0,1)$, \eqref{eq:lem:linear_estimates_new_2}, Lemma \ref{prop:quad_estimates_new} and \eqref{eq:proof:new_cubic_estimates_step_1_1}, one obtains
    \begin{equation}
        \Vert f \Vert_{L^1(L^1)} \lesssim \Vert y_1 \Vert_{L^2(L^2)} \Vert y_2 \Vert_{L^2(L^2)} + \Vert y_2 \Vert_{L^2(L^2)}^2 \lesssim \Vert u \Vert_{\widetilde{H}^{-1-\frac{s}{2}}}^2 \Vert u \Vert_{\widetilde{H}^{\sigma(s)}},
    \end{equation}
    \begin{equation}
        \Vert g \Vert_{L^2(L^2)} \lesssim \Vert y_1 \Vert_{L^2(L^2)} + \Vert y_2 \Vert_{L^2(L^2)}  \lesssim \Vert u \Vert_{\widetilde{H}^{-1-\frac{s}{2}}},
    \end{equation}
    and
    \begin{equation}
        \Vert f \Vert_{L^1(L^2)} \lesssim \Vert y_1 \Vert_{L^2(L^2)} \Vert y_2 \Vert_{L^2(H^1)} + \Vert y_2 \Vert_{L^2(L^2)} \Vert y_2 \Vert_{L^2(H^1)} \lesssim \Vert u \Vert_{\widetilde{H}^{-1-\frac{s}{2}}} \Vert u \Vert_{\widetilde{H}^{\sigma(s)}}^2.
    \end{equation}
    Hence, \eqref{eq:prop:cubic_estimates_new_1} follows from \eqref{eq:proof:new_cubic_estimates_step_1_2}.
    
    For the rest of the proof, we assume that $s \in (0, 1)$, and that $\Vert u_1 \Vert_{\widetilde{H}^{1+\sigma(s)}(0,T)} \leq \delta$.
     
    \textbf{Step 2: we prove that}
    \begin{equation}\label{eq:proof:new_cubic_estimates_step_2_1}
        \left\Vert R_3 \right\Vert_{Z_2^s} \lesssim \left\Vert u_1 \right\Vert_{\widetilde{H}^{1+\sigma(s)}}^3. 
    \end{equation}
    Using the first estimate of Lemma \ref{lem:parabolic_estimates_anisotropic_sobolev}, the second estimate of \eqref{eq:lem:linear_estimates_new_4} and the first estimate of Lemma \ref{lem:quad_estimates_additional_anisotropic_estimates}, in \eqref{eq:proof:new_cubic_estimates_1}, one finds
    \begin{equation}
        \begin{split}
            \left\Vert R_3 \right\Vert_{Z_2^s} 
        	& \lesssim \left\Vert u_1 \right\Vert_{\widetilde{H}^{1+\sigma(s)}}^3 + \left\Vert u_1 \right\Vert_{\widetilde{H}^{1+\sigma(s)}}^4  
        	     + \left( \left\Vert u_1 \right\Vert_{\widetilde{H}^{1+\sigma(s)}} + \left\Vert u_1 \right\Vert_{\widetilde{H}^{1+\sigma(s)}}^2 + \left\Vert R_3 \right\Vert_{Z_2^s} \right) \left\Vert R_3 \right\Vert_{Z_2^s} \\
        	& \lesssim \left\Vert u_1 \right\Vert_{\widetilde{H}^{1+\sigma(s)}}^3 
        	    + \left( \delta + \left\Vert R_3 \right\Vert_{Z_2^s} \right) \left\Vert R_3 \right\Vert_{Z_2^s} .
        \end{split}
    \end{equation}
    For $\delta$ sufficiently small, this gives $\left\Vert R_3 \right\Vert_{Z_2^s} \lesssim \left\Vert u_1 \right\Vert_{\widetilde{H}^{1+\sigma(s)}}^3 + \left\Vert R_3 \right\Vert_{Z_2^s}^2$. A bootstrap argument similar to that of Step 2 of the proof of Proposition \ref{prop:quad_estimates_new} gives \eqref{eq:proof:new_cubic_estimates_step_2_1} if $\delta$ is sufficiently small.
    
    \textbf{Step 3: we prove that}
    \begin{equation}\label{eq:proof:new_cubic_estimates_step_3_1}
        \left\Vert R_3 \right\Vert_{X_2^s} + \left\Vert R_3 \right\Vert_{Y_2^s} 
        \lesssim \left\Vert u_1 \right\Vert_{\widetilde{H}^{1+\sigma(s)}}^2 \left\Vert u_1 \right\Vert_{\widetilde{H}^{- \frac{s}{2}}}.
    \end{equation}
    Using the second estimate of Lemma \ref{lem:parabolic_estimates_anisotropic_sobolev} in \eqref{eq:proof:new_cubic_estimates_1}, one finds
    \begin{equation}
        \begin{split}
        \left\Vert R_3 \right\Vert_{X_2^s} + \left\Vert R_3 \right\Vert_{Y_2^s} 
        	& \lesssim \left\Vert y_1 \right\Vert_{X_2^s} \left\Vert y_2 \right\Vert_{Z_2^s} + \left\Vert y_2 \right\Vert_{X_2^s} \left\Vert y_2 \right\Vert_{Z_2^s}
			+ \left( \left\Vert y_1 \right\Vert_{X_2^s} + \left\Vert y_2 \right\Vert_{X_2^s} + \left\Vert R_3 \right\Vert_{X_2^s} \right) \left\Vert R_3 \right\Vert_{Z_2^s}.
        \end{split}
    \end{equation}
    Using Lemma \ref{lem:linear_estimates_new_additional}, Lemma \ref{lem:quad_estimates_additional_anisotropic_estimates} and \eqref{eq:proof:new_cubic_estimates_step_2_1}, one finds
    \begin{equation}
        \left\Vert R_3 \right\Vert_{X_2^s} + \left\Vert R_3 \right\Vert_{Y_2^s} 
        \lesssim  \left\Vert u_1 \right\Vert_{\widetilde{H}^{- \frac{s}{2}}} \left\Vert u_1 \right\Vert_{\widetilde{H}^{1+\sigma(s)}}^2 + \delta^3 \left\Vert R_3 \right\Vert_{X_2^s}.
    \end{equation}
    This gives \eqref{eq:proof:new_cubic_estimates_step_3_1} if $\delta$ is sufficiently small.

    \textbf{Step 4: proof of \eqref{eq:prop:cubic_estimates_new_2}.}
    Using the third estimate of Lemma \ref{lem:parabolic_estimates_anisotropic_sobolev} in \eqref{eq:proof:new_cubic_estimates_1}, one finds
    \begin{equation}
        \begin{split}
            \left\vert \left\langle R_3(T), \varphi_k \right\rangle \right\vert
            & \lesssim \left\Vert y_1 \right\Vert_{X_2^s} \left\Vert y_2 \right\Vert_{Y_2^s} + \left\Vert y_2 \right\Vert_{X_2^s} \left\Vert y_2 \right\Vert_{Y_2^s}
			+ \left( \left\Vert y_1 \right\Vert_{X_2^s} + \left\Vert y_2 \right\Vert_{X_2^s} + \left\Vert R_3 \right\Vert_{X_2^s} \right) \left\Vert R_3 \right\Vert_{Y_2^s}.
        \end{split}
    \end{equation}
    Using the first estimate of \eqref{eq:lem:linear_estimates_new_4}, the second estimate of Lemma \ref{lem:quad_estimates_additional_anisotropic_estimates}, and \eqref{eq:proof:new_cubic_estimates_step_3_1}, one obtains \eqref{eq:prop:cubic_estimates_new_2}.
\end{proof}
 
\subsection{Remainder estimates in the case $p=\infty$}

We gave above fully detailed proofs of the remainder estimates in the case $\mu \in \mathcal{H}^s_2(0,1)$, for $s \in [-1,1)$. One could apply those estimates to $\mu \in \mathcal{H}^{s+\frac{1}{2}+\varepsilon}_\infty(0,1) \hookrightarrow \mathcal{H}^s_2(0,1)$, but this would lead to an $\varepsilon$-loss in the smallness assumption of most of our main results. We therefore provide here remainder estimates specifically adapted to the case $\mu \in \mathcal{H}^{s+\frac{1}{2}}_\infty(0,1)$. For shortness, and since the proofs follow the same structure as those presented above, we provide sketches and give details only for the points that require a specific argument.

\begin{proposition}[Linear, quadratic and cubic estimates in the case $p=\infty$]\label{prop:remainder_estimates_p_infty}
    The results of Lemma \ref{lem:linear_estimates_new}, Proposition \ref{prop:quad_estimates_new} and Proposition \ref{prop:cubic_estimates_new} are true with the assumption $\mu \in \mathcal{H}^{s}_2(0,1)$ replaced by $\mu \in \mathcal{H}^{s+\frac{1}{2}}_\infty(0,1)$, for $s \in (-1, 0) \cup (0, 1)$.
\end{proposition}

\begin{proof}[Sketch of proof.]
    Let $T> 0$, $k \geq 1$, $s \in (-1, 0) \cup (0, 1)$, $\mu \in \mathcal{H}^{s+\frac{1}{2}}_\infty(0,1)$, and $u \in L^2(0,T)$. We use the symbol $\lesssim$ for constants that may depend on $s$, $k$ and $\mu$ only. 
    Recall that the spaces $X_\infty^s$, $Y_\infty^s$ and $Z_\infty^s$ are defined before Lemma \ref{lem:parabolic_estimates_anisotropic_sobolev}, and that $\mathcal{B}(f,g)$ is defined by \eqref{eq:def:mathcalB}.
    
    \textbf{Step 1: Linear estimates.}
    Using \eqref{eq:proof:lem:easy_integral_with_exp_lambda_1} and \eqref{eq:proof:lem:linear_estimates_new_1}, one finds
    \begin{equation}
        \left\Vert y_1 \right\Vert_{L^2((0,T);L^2)}^2 
        \lesssim \sum_{n \geq 1} \int_{\mathbb{R}} \frac{\mu_n^2}{\lambda_n^2 + \omega^2} \left\vert \widehat{u}(\omega) \right\vert^2 \dd \omega
        \lesssim \Vert \mu \Vert_{\mathcal{H}^{s + \frac{1}{2}}_\infty}^2 \int_{\mathbb{R}} \sum_{n \geq 1} \frac{\lambda_n^{-s-\frac{1}{2}}}{\lambda_n^2 + \omega^2} \left\vert \widehat{u}(\omega) \right\vert^2 \dd \omega.
    \end{equation}
    Hence, a sum-integral comparison gives $\left\Vert y_1 \right\Vert_{L^2((0,T);L^2)} \lesssim \Vert u \Vert_{\widetilde{H}^{-1-\frac{s}{2}}}$ if $s \in (-1, 0)$. The two other estimates of Lemma \ref{lem:linear_estimates_new} are proved in the same way. Moreover, if $s \in (0,1)$, analogous arguments give
    \begin{equation}
        \begin{split}
            \Vert y_1 \Vert_{X_\infty^s} 
            & \lesssim \left\Vert u_1 \right\Vert_{\widetilde{H}^{- \frac{s}{2}}} 
                + \sup_{n \geq 1} \left\{ n^{s + \frac{1}{2}} \left\vert \mu_n \right\vert \lambda_n \left( \int_{\mathbb{R}} \frac{(1 + \omega^2)^{-{\frac{s}{2}}}}{\lambda_n^2 + \omega^2} \left\vert \widehat{u_1}(\omega) \right\vert^2 \dd \omega \right)^{\frac{1}{2}} \right\} \\
            & \lesssim \left\Vert u_1 \right\Vert_{\widetilde{H}^{- \frac{s}{2}}} 
                + \left( \int_{\mathbb{R}} \left( \sup_{n \geq 1} \frac{\lambda_n^2}{\lambda_n^2 + \omega^2} \right) (1 + \omega^2)^{-{\frac{s}{2}}} \left\vert \widehat{u_1}(\omega) \right\vert^2 \dd \omega \right)^{\frac{1}{2}} 
             \lesssim \left\Vert u_1 \right\Vert_{\widetilde{H}^{- \frac{s}{2}}} ,
        \end{split}
    \end{equation}
    and similarly $\Vert y_1 \Vert_{Z_\infty^s} \lesssim \left\Vert u_1 \right\Vert_{\widetilde{H}^{1 + \sigma(s)}}$.

    \textbf{Step 2: Quadratic estimates.}
    The proof is similar to that of Proposition \ref{prop:quad_estimates_new}, except for the following points. First, when proving \eqref{eq:proof:prop:quad_estimate_new_06}, one must replace \eqref{eq:proof:prop:quad_estimate_new_07_bis} by 
    \begin{equation}
        \left\Vert \mu z_1(t) \right\Vert_{H^{-2 - 2\sigma(s)}(0,1)} \lesssim \left\Vert \mu \right\Vert_{H^{s-\varepsilon}(0,1)} \left\Vert z_1(t) \right\Vert_{L^2} \lesssim \left\Vert z_1(t) \right\Vert_{L^2},
    \end{equation}
    for $s \in (0,1)$ and $\varepsilon > 0$ sufficiently small, since $\mathcal{H}_\infty^{s+\frac{1}{2}}(0,1) \hookrightarrow \mathcal{H}_2^{s - \varepsilon}(0,1)$. Second, in the proof of \eqref{eq:proof:prop:quad_estimate_new_08}, the argument that relies on $\mu \in L^4(0,1)$ can still be used when $s \geq \frac{1}{2}$, since $\mathcal{H}_\infty^{s+\frac{1}{2}} \hookrightarrow \mathcal{H}_2^{\frac{1}{4}} \hookrightarrow L^4$, and the argument used when $s < \frac{1}{2}$ works with the condition $s < \alpha < \min \left( \frac{1}{4} + s, \frac{1}{2} \right)$ replaced by $s < \alpha < \min \left( \frac{1}{4} + s - \varepsilon, \frac{1}{2} \right)$, for $\varepsilon$ sufficiently small, again by the embedding $\mathcal{H}_\infty^{s+\frac{1}{2}}(0,1) \hookrightarrow \mathcal{H}_2^{s - \varepsilon}(0,1)$.
    
    The rest of the proof of the quadratic estimates is similar to that of Proposition \ref{prop:quad_estimates_new}. As in Lemma \ref{lem:quad_estimates_additional_anisotropic_estimates}, one has the following additional estimates: 
    \begin{equation}
        \left\Vert y_2 \right\Vert_{Z_\infty^s} \lesssim \left\Vert u_1 \right\Vert_{\widetilde{H}^{1+\sigma(s)}}^2 \quad \text{ and } \quad 
        \left\Vert y_2 \right\Vert_{X_\infty^s} + \left\Vert y_2 \right\Vert_{Y_\infty^s} 
        \lesssim \left\Vert u_1 \right\Vert_{\widetilde{H}^{- \frac{s}{2}}} \left\Vert u_1 \right\Vert_{\widetilde{H}^{1+\sigma(s)}}.
    \end{equation}
    
    \textbf{Step 3: Cubic estimates.}
    The proof of Proposition \ref{prop:cubic_estimates_new} applies without any changes, using the spaces $X_\infty^s$, $Y_\infty^s$ and $Z_\infty^s$ instead of $X_2^s$, $Y_2^s$ and $Z_2^s$.
\end{proof}

\subsection{Closed loop estimates}

We prove the following closed loop estimates.

\begin{lemma}\label{lem:closed_loop}
    Let $\mu \in H^{-1}(0,1)$, $\mu \neq 0$. Let $s \in [-1, 1)$ and $\lambda \in \mathbb{R}$. There exist $C > 0$ and $\nu > 0$ such that for all $T \in \left( 0, 1 \right)$ and $u \in L^2(0,T)$, one has
    \begin{equation}\label{eq:lem:closed_loop_1}
        \left\vert \int_0^T e^{\lambda t} u(t) \dd t \right\vert \leq C \left( T^{\nu} \left\Vert u \right\Vert_{\widetilde{H}^{-1 - \frac{s}{2}}(0,T)} +  \left\Vert y_1(T) \right\Vert_{H^{-1}} \right). 
    \end{equation}
\end{lemma}

\begin{proof}
    Let $j \geq 1$ be such that $\langle \mu, \varphi_j \rangle \neq 0$. By definition of $y_1(T)$, one has
    \begin{equation}
        \left\vert \int_0^T e^{\lambda t} u(t) \dd t \right\vert
        \leq \left\vert \int_0^T \left( e^{(\lambda - \lambda_j) t} -1 \right) e^{\lambda_j t} u(t) \dd t \right\vert + e^{\lambda_j T} \left\vert \left\langle \mu, \varphi_j \right\rangle \right\vert^{-1} \left\vert \left\langle y_1(T), \varphi_j \right\rangle \right\vert.
    \end{equation}
    By Lemma \ref{lem:appendix:removing_exp_from_sobolev_norm_small_time}, this gives
    \begin{equation}
        \left\vert \int_0^T e^{\lambda t} u(t) \dd t \right\vert
        \leq C \left( T^{\nu} \left\Vert u e^{\lambda_j \cdot} \right\Vert_{\widetilde{H}^{-1 - \frac{s}{2}}(0,T)} +  \left\Vert y_1(T) \right\Vert_{H^{-1}} \right),
    \end{equation}
    for some $C, \nu > 0$, which may depend on $\mu$, $\lambda$ and $s$. Lemma \ref{lem:appendix:removing_exp_from_sobolev_norm_small_time} also implies $\left\Vert u e^{\lambda_j \cdot} \right\Vert_{\widetilde{H}^{-1 - \frac{s}{2}}} \leq C \left\Vert u  \right\Vert_{\widetilde{H}^{-1 - \frac{s}{2}}}$ for some $C>0$, completing the proof of \eqref{eq:lem:closed_loop_1}.
\end{proof}

\subsection{Decoupling estimate for initial data and control}

Recall that once a source profile $\mu$ is fixed, we denote by $t \mapsto y(t; y_0, u)$ the solution of the Burgers equation \eqref{eq:Burgers_intro} with initial data $y_0$ and control $u$. We establish here an estimate that will allow us to decompose $y(t; y_0, u)$ as the sum of the controlled solution with zero initial data, $y(t; 0, u)$, and the uncontrolled solution, $y(t; y_0, 0)$.

\begin{lemma}\label{lem:decoupling_estimate}
    Let $T > 0$, $y_0 \in L^2(0,1)$, $k \geq 1$, and $\mu \in H^{-1}(0,1)$. Let $u$ be such that the solutions $y(t; y_0, u)$ and $y(t; 0, u)$ of the Burgers equation \eqref{eq:Burgers_intro} are well-defined and belong to $L^2((0,T); H^1(0,1))$. 
    Set $R(t; y_0, u) :=  y(t; y_0, u) - y(t; y_0, 0) - y(t; 0, u)$, where the solutions of the Burgers equation implicitly depend on $\mu$. We introduce the shorthand
    \begin{equation}
        y^u(t) :=  y(t; 0, u).
    \end{equation}
    Then, there exist constants $\varepsilon_0, C > 0$, which may depend only on $k$, such that if 
    \begin{equation}\label{eq:lem:decoupling_estimate_0}
        \Vert y_0 \Vert_{L^2} + \left\Vert y^u \right\Vert_{L^2((0, T);H^1)} \leq \varepsilon_0, 
    \end{equation}
    then
    \begin{equation}\label{eq:lem:decoupling_estimate_1}
        \left\Vert R(T; y_0, u) \right\Vert_{H^{-1}}^2 \leq C \sqrt{T} \Vert y_0 \Vert_{L^2}^2 \left\Vert y^u \right\Vert_{L^2((0,T);L^2)}^2,
    \end{equation}
    and 
    \begin{equation}\label{eq:lem:decoupling_estimate_2}
        \left\vert \left\langle R(T; y_0, u), \varphi_k \right\rangle \right\vert \leq C T^{\frac{1}{4}} \Vert y_0 \Vert_{L^2} \left\Vert y^u \right\Vert_{L^2((0,T);L^2)} \left( 1 + T^{\frac{1}{2}} \right).
    \end{equation}
\end{lemma}

\begin{proof}
    In this proof, the symbol $\lesssim$ is used for constants that may depend only on $k$.
    With the notations of Lemma \ref{lem:forced_burgers}, $t \mapsto R(t; y_0, u)$ is the solution of a forced Burgers equation with $f(t) := - y(t; y_0, 0) y(t; 0, u)$ and $g(t) := y(t; y_0, 0) + y(t; 0, u)$. By Lemma \ref{lem:burgers_weak_wellposedness} (applied with $f = 0$), one has 
    \begin{equation}\label{eq:proof:lem:decoupling_estimate_0}
        \left\Vert t \mapsto y(t; y_0, 0) \right\Vert_{L^2(H^1)} + \left\Vert t \mapsto y(t; y_0, 0) \right\Vert_{L^\infty(L^2)} \lesssim \Vert y_0 \Vert_{L^2} + \Vert y_0 \Vert_{L^2}^2.  
    \end{equation}
    This gives
    \begin{equation}
        \Vert g \Vert_{L^2(H^1)} \lesssim \Vert y_0 \Vert_{L^2} + \Vert y_0 \Vert_{L^2}^2 + \left\Vert y^u \right\Vert_{L^2(H^1)},
    \end{equation}
    and, by the Cauchy-Schwarz inequality together with the algebra property of $H^1(0,1)$,
    \begin{equation}
        \Vert f \Vert_{L^1(H^1)} \lesssim \left\Vert t \mapsto y(t; y_0, 0) \right\Vert_{L^2(H^1)} \left\Vert y^u \right\Vert_{L^2(H^1)} \lesssim \Vert y_0 \Vert_{L^2}^2 + \Vert y_0 \Vert_{L^2}^4 + \left\Vert y^u \right\Vert_{L^2(H^1)}^2 .
    \end{equation}
    Hence, if $ \Vert y_0 \Vert_{L^2}$ and $\left\Vert y^u \right\Vert_{L^2(H^1)} $ are sufficiently small, then Lemma \ref{lem:forced_burgers} can be applied. It gives
    \begin{equation}\label{eq:proof:lem:decoupling_estimate_1}
        \left\Vert R(T; y_0, u) \right\Vert_{H^{-1}} \lesssim \Vert f \Vert_{L^1(L^2)},
    \end{equation}
    and 
    \begin{equation}\label{eq:proof:lem:decoupling_estimate_2}
        \left\vert \left\langle R(T; y_0, u), \varphi_k \right\rangle \right\vert \lesssim \Vert f \Vert_{L^1(L^1)} + \left( \Vert g \Vert_{L^2(L^2)} + \left\Vert f \right\Vert_{L^1(L^2)} \right) \left\Vert f \right\Vert_{L^1(L^2)}.
    \end{equation}
    We prove that this implies \eqref{eq:lem:decoupling_estimate_1} and \eqref{eq:lem:decoupling_estimate_2}.
    
    First, by Lemma \ref{lem:easy_estimate_L1L2} and \eqref{eq:proof:lem:decoupling_estimate_0}, one has
    \begin{align}
        \left\Vert f \right\Vert_{L^1(L^2)}^2
        & \lesssim \left\Vert t \mapsto y(t; y_0, 0) \right\Vert_{L^2(H^1)} \left\Vert t \mapsto y(t; y_0, 0) \right\Vert_{L^2(L^2)} \left\Vert y^u \right\Vert_{L^2(L^2)}^2 \\
        & \lesssim \sqrt{T} \left\Vert t \mapsto y(t; y_0, 0) \right\Vert_{L^2(H^1)} \left\Vert t \mapsto y(t; y_0, 0) \right\Vert_{L^\infty(L^2)} \left\Vert y^u \right\Vert_{L^2(L^2)}^2 \\
        & \lesssim \sqrt{T} \left\Vert y_0 \right\Vert_{L^2}^2 \left\Vert y^u \right\Vert_{L^2(L^2)}^2, \label{eq:proof:lem:decoupling_estimate_3}
    \end{align}
    and this gives \eqref{eq:lem:decoupling_estimate_1}.
    
    Second, using the Cauchy-Schwarz inequality and \eqref{eq:proof:lem:decoupling_estimate_0}, one finds
    \begin{equation}
        \Vert f \Vert_{L^1(L^1)} \lesssim \left\Vert t \mapsto y(t; y_0, 0) \right\Vert_{L^2(L^2)} \left\Vert y^u \right\Vert_{L^2(L^2)}
        \lesssim \sqrt{T} \Vert y_0 \Vert_{L^2} \left\Vert y^u \right\Vert_{L^2(L^2)}. \label{eq:proof:lem:decoupling_estimate_4}
    \end{equation}    
    Note that \eqref{eq:proof:lem:decoupling_estimate_0} also gives
    \begin{equation}
        \Vert g \Vert_{L^2(L^2)} \lesssim \sqrt{T} \left\Vert y_0 \right\Vert_{L^2} + \left\Vert y^u \right\Vert_{L^2(L^2)}. \label{eq:proof:lem:decoupling_estimate_5}
    \end{equation}   
    Gathering \eqref{eq:proof:lem:decoupling_estimate_2}, \eqref{eq:proof:lem:decoupling_estimate_3}, \eqref{eq:proof:lem:decoupling_estimate_4} and \eqref{eq:proof:lem:decoupling_estimate_5}, one obtains 
    \begin{equation}
        \begin{split}
            \left\vert \left\langle R(T; y_0, u), \varphi_k \right\rangle \right\vert 
            \leq \ & C T^{\frac{1}{4}} \Vert y_0 \Vert_{L^2} \left\Vert y^u \right\Vert_{L^2((0,T);L^2)} \\
            &\times \left( T^{\frac{1}{4}} + T^{\frac{1}{2}} \Vert y_0 \Vert_{L^2} + \left\Vert y^u \right\Vert_{L^2((0,T);L^2)} + T^{\frac{1}{4}} \Vert y_0 \Vert_{L^2} \left\Vert y^u \right\Vert_{L^2((0,T);L^2)} \right).
        \end{split}
    \end{equation}
    Together with \eqref{eq:lem:decoupling_estimate_0}, this gives \eqref{eq:lem:decoupling_estimate_2}.
\end{proof}

\appendix

\section{Asymptotic sum estimates}\label{sec:appendix_sum_estimates}

We prove the following result, which is used several times in establishing the asymptotic properties of kernels in Section \ref{sec:properties_of_kernels}.

\begin{lemma}\label{lem:asymptotic_sum_appendix}
    The following asymptotic estimates hold, where $\mathcal{O}$ is taken as $\omega \rightarrow + \infty$.
    \begin{enumerate}[label=(\roman*)]
        \item  For $A > 0$, $\alpha \in (-1, 3)$ and $\beta \in (-1, 7)$, there exists $\nu > 0$ such that 
            \begin{equation}\label{eq:lem:asymptotic_sum_appendix_1}
                \sum_{n \geq 1} \frac{n^\alpha}{ A n^4 + \omega^2 } = \frac{1}{\omega^{\frac{3}{2} - \frac{\alpha}{2}}} \int_0^\infty \frac{\tau^\alpha}{ A \tau^4 + 1 } \dd \tau + \mathcal{O} \left( \frac{1}{\omega^{\frac{3}{2} - \frac{\alpha}{2}+\nu}} \right),
            \end{equation}
            and
            \begin{equation}\label{eq:lem:asymptotic_sum_appendix_2}
                \sum_{n \geq 1} \frac{n^\beta}{\left( A n^4 + \omega^2 \right)^2} = \frac{1}{\omega^{\frac{7}{2} - \frac{\beta}{2}}} \int_0^\infty \frac{\tau^\beta}{ \left( A \tau^4 + 1 \right)^2 } \dd \tau + \mathcal{O} \left( \frac{1}{\omega^{\frac{7}{2} - \frac{\beta}{2}+\nu}} \right).
            \end{equation}
            Moreover, when $\alpha = 2$ and $A = \pi^4$, \eqref{eq:lem:asymptotic_sum_appendix_1} can be improved as
            \begin{equation}\label{eq:lem:asymptotic_sum_appendix_3}
                \sum_{n \geq 1} \frac{n^2}{\pi^4 n^4 + \omega^2} = \frac{1}{2 \sqrt{2} \pi^2 \sqrt{\omega}} + \mathcal{O} \left( \frac{1}{\omega^{\frac{3}{2}}} \right).
            \end{equation}
            
        \item One has
            \begin{equation}\label{eq:lem:asymptotic_sum_appendix_4}
                \left\vert \sum_{n \geq 1} \frac{(-1)^n n^2}{\pi^4 n^4 + \omega^2} \right\vert = \mathcal{O} \left( \frac{1}{\omega^{\frac{3}{2}}} \right).
            \end{equation}
            
        \item  Let $\rho \in \mathbb{R}$, $k \geq 2$, and $s \in (-2, 0)$. There exists $\nu > 0$ such that
            \begin{equation}\label{eq:lem:asymptotic_sum_appendix_5}
                \sum_{n \geq 1} \frac{\sin \left(\frac{n \pi}{k} + \rho \right)^2}{ n^{2s + 1} \left( \pi^4 n^4 + \omega^2 \right)} 
                = \frac{1}{2 \omega^{s+2}} \int_0^\infty \frac{1}{ t^{2s + 1} \left( \pi^4 t^4 + 1 \right)} \dd t
                + \mathcal{O}\left( \frac{1}{\omega^{2+s+\nu}} \right).
            \end{equation}
            
    \end{enumerate}
\end{lemma}

\begin{proof}
    \underline{Proof of \emph{(i)}.}
    We first prove \eqref{eq:lem:asymptotic_sum_appendix_1} and \eqref{eq:lem:asymptotic_sum_appendix_2}. Let $m \in \{1,2\}$, let $\sigma \in (-1, 3)$ if $m=1$ and $\sigma \in (-1, 7)$ if $m=2$.  For $\tau > 0$, set $f(\tau) = \frac{\tau^\sigma}{ \left( A \tau^4 + 1 \right)^m}$. We assume $\sigma < 0$, since the other cases are easier. We use the symbol $\lesssim$ for constants that may depend on $m$, $\sigma$ and $A$.
    
    It suffices to show that there exists $\nu > 0$ such that for all $\delta \in (0,1)$, one has
    \begin{equation}\label{eq:proof:lem:asymptotic_sum_appendix_0}
        \left\vert \sum_{n \geq 1} \delta f(n\delta) - \int_0^\infty f(\tau) \dd \tau \right\vert \lesssim \delta^{\nu},
    \end{equation}
    where $\delta$ plays the role of $\frac{1}{\sqrt{\omega}}$. For small $\tau$, we will use $f(\tau) \leq \tau^\sigma$ and $\vert f^\prime(\tau) \vert \lesssim \tau^{\sigma-1}$, and for large $\tau$, we will use $f(\tau) \leq \tau^{\sigma-4m}$ and $\vert f^\prime(\tau) \vert \lesssim \tau^{\sigma-4m-1}$. Set $N(\delta) := \lfloor \delta^{-\frac{1}{2}} \rfloor$ and $M(\delta) := \lfloor \delta^{-2} \rfloor$. First, one has
    \begin{equation}\label{eq:proof:lem:asymptotic_sum_appendix_1}
        \begin{split}
            \left\vert \delta \sum_{n = N(\delta)}^{M(\delta)-1} f(n\delta) - \int_{\delta N(\delta)}^{\delta M(\delta)} f(\tau) \dd \tau \right\vert 
            & \leq \delta \int_{\delta N(\delta)}^{\delta M(\delta)} \vert f^\prime(\tau) \vert \dd \tau \\
            & \lesssim \delta \int_{\delta N(\delta)}^1 \tau^{\sigma-1} \dd \tau + \delta \int_{1}^{\delta M(\delta)} \tau^{\sigma-4m-1}\dd \tau \\
            & \lesssim \delta \left( \delta N(\delta) \right)^{\sigma} + \delta \left( \delta M(\delta) \right)^{\sigma-4m} \\
            & \lesssim \delta^{1 + \frac{\sigma}{2}} + \delta^{4m-\sigma+1} ,
        \end{split}
    \end{equation}
    where we have used $\sigma \leq 0$.
    Second, one has
    \begin{equation}\label{eq:proof:lem:asymptotic_sum_appendix_2}
        \delta \sum_{n = 1}^{N(\delta)-1} f(n\delta) + \int_0^{\delta N(\delta)} f(\tau) \dd \tau
        \lesssim \delta^{1+\sigma} N(\delta)^{1+\sigma} \lesssim \delta^{\frac{1}{2}+ \frac{\sigma}{2}}.
    \end{equation}
    Third, one has
    \begin{equation}\label{eq:proof:lem:asymptotic_sum_appendix_3}
        \delta \sum_{n \geq M(\delta)} f(n \delta) + \int_{\delta M(\delta)}^\infty f(\tau) \dd \tau 
        \lesssim \delta^{\sigma-4m+1} M(\delta)^{\sigma-4m+1} \lesssim \delta^{4m-\sigma-1}. 
    \end{equation}
    Combining \eqref{eq:proof:lem:asymptotic_sum_appendix_1}, \eqref{eq:proof:lem:asymptotic_sum_appendix_2} and \eqref{eq:proof:lem:asymptotic_sum_appendix_3}, one obtains \eqref{eq:proof:lem:asymptotic_sum_appendix_0} with $\nu = \min\left\{ \frac{1+\sigma}{2}, \frac{1}{2}, 4m-\sigma-1\right\} > 0$.

    Now, we prove \eqref{eq:lem:asymptotic_sum_appendix_3}. Let $f(\tau) = \frac{\tau^2}{\pi^4 \tau^4 + 1}$. To obtain the $\mathcal{O} \left( \frac{1}{\omega^{\frac{3}{2}}} \right)$ remainder, we need a sharper estimate of the error between the integral and its Riemann sum than \eqref{eq:proof:lem:asymptotic_sum_appendix_1}. We use the trapezoidal formula
    \begin{equation}
        \int_a^b f(\tau) \dd \tau = \frac{1}{2} (b-a) \left( f(a) + f(b) \right) + \frac{1}{2} \int_a^b (b-\tau) (a-\tau) f^{\prime \prime}(\tau) \dd \tau,
    \end{equation}
    for $0 < a < b < + \infty$. Since $f^{\prime \prime} \in L^1((0, +\infty); \mathbb{R})$, this gives
    \begin{equation}
        \left\vert \int_0^\infty f(\tau) \dd \tau - \frac{\delta}{2} \sum_{n \geq 0} \left\{ f(n\delta) + f((n+1)\delta) \right\} \right\vert
        \leq \frac{\delta^2}{2} \int_0^\infty \left\vert f^{\prime \prime}(\tau) \right\vert \dd \tau,
    \end{equation}
    for all $\delta > 0$. Since $f(0) = 0$, applying this to $\delta=\frac{1}{\sqrt{\omega}}$ yields
    \begin{equation}
        \left\vert \int_0^\infty f(\tau) \dd \tau - \sqrt{\omega} \sum_{n \geq 1} \frac{n^2}{\pi^4 n^4 + \omega^2} \right\vert
        \leq \frac{1}{2\omega} \int_0^\infty \left\vert f^{\prime \prime}(\tau) \right\vert \dd \tau.
    \end{equation}
    Using $\int_0^\infty f = \frac{1}{2\sqrt{2} \pi^2}$, one obtains \eqref{eq:lem:asymptotic_sum_appendix_3}.
    
    \underline{Proof of \emph{(ii)}.} As above, to obtain the $\mathcal{O} \left( \frac{1}{\omega^{\frac{3}{2}}} \right)$ remainder in \eqref{eq:lem:asymptotic_sum_appendix_4}, we need an estimate involving $f^{\prime \prime} \in L^1((0, +\infty); \mathbb{R})$, where $f(\tau) = \frac{\tau^2}{\pi^4 \tau^4 + 1}$. We use
    \begin{equation}
        \sum_{n \geq 1} (-1)^n f(n\delta) = - \frac{f(\delta)}{2} + \frac{1}{2} \sum_{n \geq 1} \left\{ 2 f(2n\delta) - f((2n+1)\delta) - f((2n-1)\delta) \right\},
    \end{equation}
    and
    \begin{equation}
        2 f\left( \frac{a+b}{2} \right) - f(a) - f(b) =  \int_a^{\frac{a+b}{2}} (a-\tau) f^{\prime \prime}(\tau) \dd \tau + \int_{\frac{a+b}{2}}^b (\tau-b) f^{\prime \prime}(\tau) \dd \tau,
    \end{equation}
    which imply
    \begin{equation}
        \left\vert \sum_{n \geq 1} (-1)^n f(n\delta) \right\vert \leq \frac{\delta}{2} + \frac{\delta}{2} \sum_{n \geq 1} \int_{(2n-1)\delta}^{(2n+1)\delta} \left\vert f^{\prime \prime}(\tau) \right\vert \dd \tau \leq C \delta,
    \end{equation}
    for some absolute constant $C>0$, and for all $\delta \in (0,1)$. With $\delta = \frac{1}{\sqrt{\omega}}$, this gives \eqref{eq:lem:asymptotic_sum_appendix_4}.
    
    \underline{Proof of \emph{(iii)}.} Since
    \begin{equation}
        \sin \left(\frac{n \pi}{k} + \rho \right)^2 = \frac{1}{2} - \frac{1}{2} \cos \left(\frac{2 n \pi}{k} + 2 \rho \right),
    \end{equation}
    we can apply \eqref{eq:lem:asymptotic_sum_appendix_1}, with $\alpha = -2s -1 \in (-1, 3)$, to find
    \begin{equation}
        \sum_{n \geq 1} \frac{\sin \left(\frac{n \pi}{k} + \rho \right)^2}{ n^{2s + 1} \left( \pi^4 n^4 + \omega^2 \right)} 
        = \frac{1}{2 \omega^{s+2}} \int_0^\infty \frac{1}{ t^{2s + 1} \left( \pi^4 t^4 + 1 \right)} \dd t
        - \frac{1}{2} \sum_{n \geq 1} \frac{\cos \left(\frac{2 n \pi}{k} + 2 \rho \right)}{ n^{2s + 1} \left( \pi^4 n^4 + \omega^2 \right)} 
        + \mathcal{O}\left( \frac{1}{\omega^{2+s+\nu}} \right),
    \end{equation}
    for some $\nu > 0$. It remains to estimate the sum involving $c_n := \cos \left(\frac{2 n \pi}{k} + 2 \rho \right)$. Set $f(\tau) = \frac{\tau^{-2s - 1}}{\pi^4 \tau^4 + 1}$. As above, it suffices to show that there exist $C>0$ and $\nu > 0$ such that for all $\delta \in (0,1)$, one has
    \begin{equation}\label{eq:proof:lem:asymptotic_sum_appendix_4}
        \left\vert \sum_{n \geq 1} c_n f(n\delta) \right\vert \leq C \delta^{-1 + 2 \nu}.
    \end{equation}
    Since $\sum_{n = 1}^k c_n = 0$, one has
    \begin{equation}
        \left\vert \sum_{n \geq 1} c_n f(n\delta) \right\vert 
        \leq \sum_{p \geq 0} \sum_{n = kp+1}^{k(p+1)} \vert c_n \vert \left\vert  f(n\delta) - f((kp+1)\delta) \right\vert
        \leq C \sum_{n \geq 1} \left\vert f((n+1)\delta) - f(n\delta) \right\vert ,
    \end{equation}
    for some $C>0$ independent of $\delta$. Instead of estimating this last sum by repeating the argument of \emph{(i)}, we use \emph{(i)} directly. For $n \geq 1$, one has
    \begin{equation}
        \begin{split}
            \left\vert  f((n+1)\delta) - f(n\delta) \right\vert 
            \leq \ &  \frac{1}{ \left( \delta^4 n^4 + 1 \right) n^{2s+1} \delta^{2s+1} } \left\vert 1 - \left( 1 + \frac{1}{n} \right)^{-1 - 2s} \right\vert \\
            & + \frac{1}{ n^{2s+1} \delta^{2s+1} } \left\vert \frac{1}{ \pi^4 \delta^4 n^4 + 1 } - \frac{1}{ \pi^4 \delta^4 (n+1)^4 + 1 } \right\vert \\
            \leq \ & C \frac{n^{-2-2s}}{ \delta^{2s+1} \left( \delta^4 n^4 + 1 \right)  } + C \frac{n^{2-2s}}{ \delta^{2s-3} \left( \delta^4 n^4 + 1 \right)^2  } \\
            \leq \ &  2 C \frac{n^{-2-2s}}{ \delta^{2s+1} \left( \delta^4 n^4 + 1 \right)  },
        \end{split}
    \end{equation}
    for some $C >0$ independent of $\delta$ and $n$. Since $- 2 -2s \notin (-1, 3)$, we cannot apply \emph{(i)} with $\alpha = - 2 -2s$. Since $s < 0$, we may choose $\theta \in (0,1)$ such that $\theta > 1 + 2s$. Since $n^{-2-2s} \leq n^{-2-2s+\theta}$, and since $-2-2s+\theta \in (-1, 3)$, we can apply \emph{(i)} with $\alpha = - 2 -2s+\theta$: it gives
    \begin{equation}
        \sum_{n \geq 1} \left\vert f((n+1)\delta) - f(n\delta) \right\vert \leq C \frac{ \delta^{2+2s - \theta -1} }{ \delta^{2s+1} } = C  \delta^{- \theta} .
    \end{equation}
    Hence, \eqref{eq:proof:lem:asymptotic_sum_appendix_4} holds true with $\nu = \frac{1 - \theta}{2} > 0$. This completes the proof of \eqref{eq:lem:asymptotic_sum_appendix_5}.
\end{proof}

\section{Some useful technical results about Sobolev spaces}\label{sec:appendix_sobolev}

\subsection{Elementary and product estimates}

We first give an elementary estimate in $L^1((0,T);L^2(0,1))$, followed by Sobolev product estimates that play a crucial role in the proof of the quadratic and cubic remainder estimates. We begin with the following standard estimate, which is used repeatedly throughout the article.

\begin{lemma}\label{lem:easy_estimate_L1L2}
    There exists an absolute constant $C > 0$ such that for all $\varphi \in H_0^1(0,1)$,
    \begin{equation}\label{eq:lem:easy_estimate_L1L2_1}
        \Vert \varphi \Vert_{L^\infty}^2 \leq C \Vert \varphi \Vert_{H^1} \Vert \varphi \Vert_{L^2},
    \end{equation}
    and for all $T > 0$, $f \in L^2((0,T); H_0^1(0,1))$ and $g \in L^2((0,T); H^1(0,1))$, one has
    \begin{equation}\label{lem:estimate_L1L2_2}
        \Vert fg \Vert_{L^1((0, T);L^2)} \leq C \Vert f \Vert_{L^2((0, T);H^1)}^{\frac{1}{2}} \Vert f \Vert_{L^2((0, T);L^2)}^{\frac{1}{2}} \Vert g \Vert_{L^2((0, T);L^2)}.
    \end{equation}
\end{lemma}

\begin{proof}
    First, writing $\varphi(x)^2 = 2 \int_0^x \varphi^\prime \varphi$ and using the Cauchy-Schwarz inequality, one obtains \eqref{eq:lem:easy_estimate_L1L2_1}. Second, using the Cauchy-Schwarz inequality again, one finds
    \begin{equation}
        \Vert fg \Vert_{L^1(L^2)} \leq \Vert f \Vert_{L^2(L^\infty)} \Vert g \Vert_{L^2((0,T);L^2)},
    \end{equation}
    which, together with \eqref{eq:lem:easy_estimate_L1L2_1}, implies \eqref{lem:estimate_L1L2_2}.
\end{proof}

The next results are the key technical tools to obtain quadratic and cubic estimates for $\mu \in \mathcal{H}^s_2(0,1)$ with $s \geq 0$. First, we recall the following classical result (see, for instance, \cite[Section 4.6.1, Theorem 1]{RunstSickel1996} and \cite[Section 4.6.1, Proposition 1]{RunstSickel1996}).

\begin{lemma}\label{lem:product_sobolev_classical}
    Let $I \subset \mathbb{R}$ be an open interval, not necessarily bounded. If $a \geq r$, $b \geq r$, $a + b - \frac{1}{2} > r$ and $a + b > 0$, then there exists $C>0$ such that for all $f \in H^a(I, \mathbb{R})$ and $g \in H^b(I, \mathbb{R})$, one has
    \begin{equation}
        \Vert fg \Vert_{H^r(I)} \leq C \Vert f \Vert_{H^a(I)} \Vert g \Vert_{H^b(I)}.
    \end{equation}
    In addition, this estimate also holds when $a + b = 0$, provided $r < - \frac{1}{2}$ and $r \leq a < 0$.
\end{lemma}

Second, we will also need the following result, which follows from \cite[Section 4.6.3, Theorem 1]{RunstSickel1996} and $\mathbbm{1}_{(a,b)} = \mathbbm{1}_{(a,+\infty)} \mathbbm{1}_{(- \infty,b)}$.

\begin{lemma}\label{lem:product_indicatrice}
    For $s \in \left(-\frac{1}{2}, \frac{1}{2} \right)$, there exists $C > 0$ such that for all $a,b \in \mathbb{R}$ with $a<b$, and all $u \in H^s(\mathbb{R})$, one has
    \begin{equation}
        \left\Vert \mathbbm{1}_{(a,b)} u \right\Vert_{H^s(\mathbb{R})} \leq C \left\Vert u \right\Vert_{H^s(\mathbb{R})}.
    \end{equation}
\end{lemma}

Third, when proving remainder estimates for $\mu \in \mathcal{H}^{s}_2$, with $s>0$, we need the following technical result, whose proof is based on multiplication results in Sobolev spaces of the torus.

\begin{lemma}\label{lem:product_kind_of_discrete_sobolev}
    Let $s \in \left( 0, 1 \right)$ and $\alpha$ be such that $s < \alpha < s + \frac{1}{4}$ if $s \in \left(0, \frac{1}{2} \right)$, and $\frac{1}{2} < \alpha < 1 - \frac{s}{2}$ if $s \in \left[ \frac{1}{2}, 1 \right)$. There exists $C > 0$ such that for all non-negative sequences $(a_p)$ and $(b_q)$,
    \begin{equation}\label{eq:proof:cubic_new_singular_term_3}
        \sum_{n\geq 1} \lambda_n^{2\alpha-2} \left( \sum_{p, q \geq 1} \left\vert \left\langle \varphi_p \varphi_q, \varphi_n^\prime  \right\rangle \right\vert a_p b_q \right)^2 \leq C \left( \sum_{p\geq 1} \lambda_p^{s} a_p^2 \right) \left( \sum_{q \geq 1} \lambda_q^{s} b_q^2 \right).
    \end{equation}
\end{lemma}

\begin{proof}
    We use the symbol $\lesssim$ for constants that may depend $s$ and $\alpha$. For $p, q \in \mathbb{Z}$, set $A_p := a_{\vert p \vert}$ if $p \neq 0$, and $A_p := 0$ if $p = 0$, and similarly, $B_q := b_{\vert q \vert}$ if $q \neq 0$, and $B_q := 0$ if $q = 0$. By \eqref{eq:expression_varphi_varphi_varphiprime}, one has
    \begin{equation}\label{eq:bound_varphi_varphi_varphiprime}
        \left\vert \left\langle \varphi_p \varphi_q, \varphi_n^\prime  \right\rangle \right\vert \lesssim n \left( \mathbbm{1}_{p+q=n} + \mathbbm{1}_{\vert p - q \vert = n} \right).
    \end{equation}
    In particular, for all $n \geq 1$, one has
    \begin{equation}
        \begin{split}
             \frac{1}{n} \sum_{p, q \geq 1} \left\vert \left\langle \varphi_p \varphi_q, \varphi_n^\prime  \right\rangle \right\vert a_p b_q 
            & \lesssim \sum_{p = 1}^{n-1} a_p b_{n-p} + \sum_{p \geq 1} a_p b_{n+p} + \sum_{q \geq 1} a_{n+q} b_q \\
            & = \sum_{p = 1}^{n-1} A_p B_{n-p} + \sum_{p \geq 1} A_{-p} B_{n+p} + \sum_{q \geq 1} A_{n+q} B_{-q} 
             \lesssim \left( A \ast B \right)_n,
        \end{split}
    \end{equation}
    where $\left( A \ast B \right)_n := \sum_{m \in \mathbb{Z}} A_m B_{n-m}$. Hence, it suffices to show
    \begin{equation}\label{eq:proof:lem:product_kind_of_discrete_sobolev}
        \sum_{n\geq 1} \lambda_n^{2\alpha-1} \left(  \left( A \ast B \right)_n \right)^2 \lesssim \left( \sum_{p\geq 1} \lambda_p^{s} A_p^2 \right) \left( \sum_{q \geq 1} \lambda_q^{s} B_q^2 \right).
    \end{equation}
    We claim that we are reduced to a result of product of Sobolev functions on the torus. Indeed, write $\mathbb{T} := \mathbb{R} / \mathbb{Z}$, and consider the Sobolev spaces $H^\sigma(\mathbb{T}; \mathbb{C})$, defined using the Hilbertian basis $\left( e^{2 i \pi n \cdot } \right)_{n\in \mathbb{Z}}$, naturally endowed with the norm 
    \begin{equation}
        \Vert h \Vert_{H^\sigma(\mathbb{T})}^2 := \sum_{n \in \mathbb{Z}} (1+n^2)^\sigma \left\vert \left\langle h, e^{2 i \pi n \cdot} \right\rangle_{L^2(\mathbb{T})} \right\vert^2.
    \end{equation}
    Set $f(x) = \sum_{p \in \mathbb{Z}} A_p e^{2 i \pi p x}$ and $g(x) = \sum_{q \in \mathbb{Z}} B_q e^{2 i \pi q x}$. Then $f(x) g(x) = \sum_{n \in \mathbb{Z}} \left( A \ast B \right)_n e^{2 i \pi n x}$, yielding
    \begin{equation}
        \Vert fg \Vert_{H^{2\alpha - 1}(\mathbb{T})}^2 = \sum_{n \in \mathbb{Z}} (1+n^2)^{2\alpha-1} \left(  \left( A \ast B \right)_n \right)^2 .
    \end{equation}
    The assumptions on $\alpha$ imply that $s \geq 2 \alpha - 1$ and $s + s - \frac{1}{2} > 2 \alpha - 1$. Hence, by the standard product estimate on the torus, analogous to Lemma \ref{lem:product_sobolev_classical}, we obtain \eqref{eq:proof:lem:product_kind_of_discrete_sobolev} (see, for instance, \cite{BahouriChemin} or \cite{RunstSickel1996}).
\end{proof}

Finally, when proving remainder estimates for $\mu \in \mathcal{H}^{s+\frac{1}{2}}_\infty$, with $s > 0$, we will need the following technical result.

\begin{lemma}\label{lem:product_kind_of_discrete_sobolev_p_infty}
    Let $s \in \left( 0, 1 \right)$. For $n \geq 1$, set $I_{n,s} := \sum_{p,q \geq 1} \left\vert \left\langle \varphi_p \varphi_q, \varphi_n^\prime \right\rangle \right\vert p^{-s-\frac{1}{2}} q^{-s-\frac{1}{2}}$. Then, there exists $C > 0$ such that 
    \begin{equation}\label{eq:lem:product_kind_of_discrete_sobolev_p_infty}
        I_{n,s} \leq 
        \begin{cases}
              C n^{1-2s} & \text{ if } s < \frac{1}{2},\\
              C \ln(2+n) & \text{ if } s = \frac{1}{2},\\
              C n^{-s+\frac{1}{2}} & \text{ if } s > \frac{1}{2}.
        \end{cases}
    \end{equation}
\end{lemma}

\begin{proof}
    We use the symbol $\lesssim$ for constants that may depend on $s$. By \eqref{eq:bound_varphi_varphi_varphiprime}, one has $I_{n,s} \lesssim I_{n,s}^1 + I_{n,s}^2$, where $I_{n,s}^1 := n \sum_{p = 1}^{n-1} p^{-s-\frac{1}{2}} (n-p)^{-s-\frac{1}{2}}$ and $I_{n,s}^2 := n \sum_{p \geq 1} p^{-s-\frac{1}{2}} (p+n)^{-s-\frac{1}{2}}$. Note that
    \begin{equation}
        I_{n,s}^1 \leq 2 n \sum_{1 \leq p \leq \frac{n}{2}} p^{-s-\frac{1}{2}} (n-p)^{-s-\frac{1}{2}}
        \leq n \sum_{1 \leq p \leq \frac{n}{2}} p^{-s-\frac{1}{2}} \left( \frac{n}{2} \right)^{-s-\frac{1}{2}},
    \end{equation}
    implying that $I_{n,s}^1$ satisfies \eqref{eq:lem:product_kind_of_discrete_sobolev_p_infty}. To estimate $I_{n,s}^2$, we distinguish three cases. First, assume that $s > \frac{1}{2}$. Then $I_{n,s}^2 \leq n \sum_{p \geq 1} p^{-s-\frac{1}{2}} n^{-s-\frac{1}{2}} \lesssim n^{-s+\frac{1}{2}}$. Second, assume that $s < \frac{1}{2}$. Then
    \begin{equation}
        I_{n,s}^2 \leq n \sum_{p = 1}^{n-1} p^{-s-\frac{1}{2}} n^{-s-\frac{1}{2}} + n \sum_{p = n}^\infty p^{-s-\frac{1}{2}} p^{-s-\frac{1}{2}} \lesssim n^{1-2s}.
    \end{equation}
    The case $s = \frac{1}{2}$ is similar.
\end{proof}

\subsection{Estimates in weak Sobolev norms}

Here, we prove three results involving weak Sobolev norms. The first shows how an estimate involving $u_1$ can be converted into an estimate involving $u$. The second shows that some Sobolev regularity can be traded for a power of $T$. The third will be used when the kernel $K_{\mu,k}$ decays faster than expected from the regularity of $\mu$.

We start with the following result, which will be used both for large times, in which case the $T$-dependence of the constant is allowed, and for small time, with a constant independent of $T$. Recall that $u_1(t) := \int_0^t u(\tau) \dd \tau$.

\begin{lemma}\label{lem:appendix:weak_norm_primitive_new}
    Let $r < \frac{1}{2}$ and $T > 0$. There exists $C_T > 0$ such that for all $u \in L^2(0,T)$, one has
    \begin{equation}\label{eq:lem:appendix:weak_norm_primitive_new}
        \Vert u_1 \Vert_{\widetilde{H}^r(0,T)} \leq C_T \Vert u \Vert_{\widetilde{H}^{r-1}(0,T)}.
    \end{equation}
    In addition, there exists $C > 0$ which depends only on $r$ such that if $T \in (0,1)$ then $C_T \leq C$.
\end{lemma}

\begin{proof}
    We use the symbol $\lesssim$ for constants that may depend only on $r$. We still write $u$ and $u_1$ for their extension by zero outside $(0,T)$. For $\omega \in \mathbb{R}$, by definition of $u_1$, one has 
    \begin{equation}
        i \omega \widehat{u_1}(\omega) 
        = \int_0^T u(s) \left( e^{- i \omega s} - e^{- i \omega T} \right) \dd s 
        = \widehat{u}(\omega) - u_1(T) e^{- i \omega T}.
    \end{equation}
    Since $r < \frac{1}{2}$, this gives
    \begin{equation}
        \begin{split}
            \Vert u_1 \Vert_{\widetilde{H}^r(0,T)}^2
            & = \int_{\vert \omega \vert \leq 1} \left( 1 + \omega^2 \right)^r \left\vert \widehat{u_1}(\omega) \right\vert^2 \dd \omega
                + \int_{\vert \omega \vert > 1} \left( 1 + \omega^2 \right)^r \left\vert \widehat{u_1}(\omega) \right\vert^2 \dd \omega \\
            & \lesssim \sup_{\vert \omega \vert \leq 1} \left\vert \widehat{u_1}(\omega) \right\vert^2
                + 2 \int_{\vert \omega \vert > 1} \left( 1 + \omega^2 \right)^{r-1} \omega^2 \left\vert \widehat{u_1}(\omega) \right\vert^2 \dd \omega \\
            & \lesssim \sup_{\vert \omega \vert \leq 1} \left\vert \widehat{u_1}(\omega) \right\vert^2 + \left\vert u_1(T) \right\vert^2
                + \Vert u \Vert_{\widetilde{H}^{r-1}(0,T)}^2.
        \end{split}
    \end{equation}
    We prove below that
    \begin{equation}\label{eq:proof:lem:appendix:weak_norm_primitive_new_2}
        \left\vert u_1(T) \right\vert \leq C_T  \Vert u \Vert_{\widetilde{H}^{r-1}(0,T)},
    \end{equation}
    and
    \begin{equation}\label{eq:proof:lem:appendix:weak_norm_primitive_new_3}
        \sup_{\vert \omega \vert \leq 1} \left\vert \widehat{u_1}(\omega) \right\vert \leq C_T  \Vert u \Vert_{\widetilde{H}^{r-1}(0,T)},
    \end{equation}
    for some $C_T > 0$ bounded uniformly for $T \in (0,1)$, which yield \eqref{eq:lem:appendix:weak_norm_primitive_new}.

    \underline{Proof of \eqref{eq:proof:lem:appendix:weak_norm_primitive_new_2}.}
        Assume that $T \in (0,1)$. Let $\chi \in C^\infty_\mathrm{c}(-2, 2)$ be such that $\chi = 1$ on $[0,1]$ (and thus on $[0, T]$). Since $u$ is supported in $[0, T]$, one has
        \begin{equation}
            \left\vert u_1(T) \right\vert
            = \left\vert \int_{\mathbb{R}} u(t) \chi(t) \dd t \right\vert 
            \lesssim \Vert u \Vert_{\widetilde{H}^{r-1}(0,T)} \left\Vert \chi \right\Vert_{H^{1-r}(\mathbb{R})} 
            \lesssim \Vert u \Vert_{\widetilde{H}^{r-1}(0,T)},
        \end{equation}
        by duality. A similar argument gives \eqref{eq:proof:lem:appendix:weak_norm_primitive_new_2} for $T \geq 1$, with a constant depending on $T$.
    
    \underline{Proof of \eqref{eq:proof:lem:appendix:weak_norm_primitive_new_3}.}
        Let $\omega \in [-1,1]$, $\omega \neq 0$. Let $\chi$ be as above. Since $u$ is supported in $[0, T]$, one has
        \begin{equation}
            \left\vert \widehat{u_1}(\omega) \right\vert
            = \left\vert \int_{\mathbb{R}} u(t) \frac{1 - e^{i \omega (T-t)}}{\omega} \chi(t) \dd t \right\vert 
            \lesssim \Vert u \Vert_{\widetilde{H}^{r-1}(0,T)} \left\Vert t \longmapsto \frac{1 - e^{i \omega (T-t)}}{\omega} \chi(t) \right\Vert_{H^{1-r}(\mathbb{R})},
        \end{equation}
        by duality. Set $f_\omega(t) := \frac{1 - e^{i \omega t}}{\omega}$. Then $\sup_{\vert t \vert < 3} \left\vert f_\omega(t) \right\vert \lesssim 1$, and for $n \geq 1$, $\sup_{\vert t \vert < 3} \left\vert f_\omega^{(n)}(t) \right\vert \lesssim \vert \omega \vert^{n-1} \leq 1$.  In particular, one has
        \begin{equation}
            \left\Vert t \longmapsto \frac{1 - e^{i \omega (T-t)}}{\omega} \chi(t) \right\Vert_{H^{1-r}(\mathbb{R})} \lesssim 1,
        \end{equation}
        implying $\left\vert \widehat{u_1}(\omega) \right\vert \lesssim  \Vert u \Vert_{\widetilde{H}^{r-1}(0,T)}$. By continuity of $\widehat{u_1}$, this also holds at $\omega=0$. A similar argument gives \eqref{eq:proof:lem:appendix:weak_norm_primitive_new_3} for $T \geq 1$, with a constant depending on $T$.
\end{proof}

Second, the following result will be used to prove the small-time drift estimates; it shows that a norm weaker than the one appearing in the drift estimate can be absorbed.

\begin{lemma}\label{lem:appendix:weak_norm_power_T_new}
    Let $r \in \left[ \frac{1}{2}, \frac{3}{2} \right)$, and let $\nu > 0$. 
    There exist $C, \nu_0 > 0$ such that for all $T \in \left(0,1\right)$ and all $u \in L^2(0, T)$, one has 
    \begin{equation}\label{eq:lem:appendix:weak_norm_power_T_new}
        \Vert u \Vert_{\widetilde{H}^{-r - \nu}(0,T)} \leq C T^{\nu_0} \Vert u \Vert_{\widetilde{H}^{-r}(0,T)} + C \left\vert u_1(T) \right\vert.
    \end{equation}
\end{lemma}
 
\begin{proof}
    In this proof, the symbol $\lesssim$ is used for constants that may depend on $r$ and $\nu$. By definition, one has 
    \begin{equation}
        \Vert u \Vert_{\widetilde{H}^{-r - \nu}(0,T)}^2 = \int_{\vert \omega \vert < \frac{1}{T}} \frac{\left\vert \widehat{u}(\omega) \right\vert^2}{(1 + \omega^2)^{r+\nu}} \dd \omega +  \int_{\vert \omega \vert \geq \frac{1}{T}} \frac{\left\vert \widehat{u}(\omega) \right\vert^2}{(1 + \omega^2)^{r+\nu}} \dd \omega.
    \end{equation}
    The high-frequency term satisfies
    \begin{equation}
        \int_{\vert \omega \vert \geq \frac{1}{T}} \frac{\left\vert \widehat{u}(\omega) \right\vert^2}{(1 + \omega^2)^{r+\nu}} \dd \omega 
        \leq T^{2 \nu} \int_{\vert \omega \vert \geq \frac{1}{T}} \frac{\left\vert \widehat{u}(\omega) \right\vert^2}{(1 + \omega^2)^{r}} \dd \omega
        \leq T^{2 \nu} \Vert u \Vert_{\widetilde{H}^{-r}(0,T)}^2.
    \end{equation}
    Hence, it suffices to estimate the low-frequency term. For $\omega \in \left( - \frac{1}{T}, \frac{1}{T} \right)$, one has
    \begin{equation}
        \left\vert \widehat{u}(\omega) \right\vert 
        \leq \left\vert \int_0^\omega \left( \widehat{u} \right)^\prime (\xi) \dd \xi \right\vert + \left\vert \widehat{u}(0) \right\vert 
        \leq \vert \omega \vert \left( \sup_{\vert \xi \vert < \frac{1}{T}} \left\vert \left( \widehat{u} \right)^\prime (\xi) \right\vert \right) + \left\vert u_1(T) \right\vert,
    \end{equation}
    yielding
    \begin{equation}
        \int_{\vert \omega \vert < \frac{1}{T}} \frac{\left\vert \widehat{u}(\omega) \right\vert^2}{(1 + \omega^2)^{r+\nu}} \dd \omega
        \lesssim \left( \sup_{\vert \xi \vert < \frac{1}{T}} \left\vert \left( \widehat{u} \right)^\prime (\xi) \right\vert^2 \right) \int_{\vert \omega \vert < \frac{1}{T}} \frac{\left\vert \omega \right\vert^2}{(1 + \omega^2)^{r+\nu}} \dd \omega + \left\vert u_1(T) \right\vert^2,
    \end{equation}
    since $r + \nu > \frac{1}{2}$. We prove below that 
    \begin{equation}\label{eq:proof:lem:appendix:weak_norm_power_T_new_1}
        I := \sup_{\vert \xi \vert < \frac{1}{T}} \left\vert \left( \widehat{u} \right)^\prime (\xi) \right\vert \lesssim T^{\frac{3}{2} - r} \Vert u \Vert_{\widetilde{H}^{-r}(0,T)},
    \end{equation}
    which is sufficient to conclude. Indeed, if $r + \nu > \frac{3}{2}$, then 
    \begin{equation}
        J := \int_{\vert \omega \vert < \frac{1}{T}} \frac{\left\vert \omega \right\vert^2}{(1 + \omega^2)^{r+\nu}} \dd \omega \lesssim \int_{\mathbb{R}} \frac{\left\vert \omega \right\vert^2}{(1 + \omega^2)^{r+\nu}} \dd \omega \lesssim 1,
    \end{equation}
    if $r + \nu = \frac{3}{2}$, then $J \lesssim 1-\ln(T) $, and if $r + \nu < \frac{3}{2}$, then 
    \begin{equation}
        J \lesssim \int_{\vert \omega \vert < \frac{1}{T}} \left\vert \omega \right\vert^{2-2r-2\nu} \dd \omega \lesssim T^{-3+2r+2\nu}.
    \end{equation}
    Hence, in any case, one has 
    \begin{equation}
        I^2 J \lesssim T^{3 - 2r} \Vert u \Vert_{\widetilde{H}^{-r}(0,T)}^2 J \lesssim T^{2\nu_0} \Vert u \Vert_{\widetilde{H}^{-r}(0,T)}^2,
    \end{equation}
    for some $\nu_0 > 0$.
    
    It remains to prove \eqref{eq:proof:lem:appendix:weak_norm_power_T_new_1}. Let $\chi \in C^\infty_\mathrm{c}(-1, 2)$ be such that $\chi = 1$ on $[0,1]$. We still write $u$ for its extension by zero outside $[0,T]$. Let $\vert \xi \vert < \frac{1}{T}$. By duality, one has
    \begin{equation}
        \left\vert \left( \widehat{u} \right)^\prime (\xi) \right\vert 
        = \left\vert \int_{\mathbb{R}} u(t) t e^{-i \xi t} \chi\left(\frac{t}{T}\right) \dd t \right\vert
        \lesssim \Vert u \Vert_{\widetilde{H}^{-r}(0,T)} \Vert \varphi_\xi \Vert_{H^{r}(\mathbb{R})},
    \end{equation}
    where $\varphi_\xi(t) := t e^{-i \xi t} \chi\left(\frac{t}{T}\right)$. One has
    \begin{equation}
        \Vert \varphi_\xi \Vert_{L^2(\mathbb{R})}^2 = \int_{\mathbb{R}} t^2 \chi\left(\frac{t}{T}\right)^2 \dd t \lesssim T^3,
    \end{equation}
    \begin{equation}
        \Vert \varphi_\xi^\prime \Vert_{L^2(\mathbb{R})}^2 
        \lesssim \int_{\mathbb{R}} \left( \chi\left(\frac{t}{T}\right)^2 + \frac{t^2}{T^2} \chi^\prime\left(\frac{t}{T}\right)^2 + \xi^2 t^2 \chi\left(\frac{t}{T}\right)^2 \right)\dd t \lesssim T,
    \end{equation}
    and similarly, $ \Vert \varphi_\xi^{\prime \prime} \Vert_{L^2(\mathbb{R})}^2 \lesssim \frac{1}{T}$. By interpolation, this yields $\Vert \varphi_\xi \Vert_{H^{r}(\mathbb{R})} \lesssim T^{\frac{3}{2} - r}$, completing the proof of \eqref{eq:proof:lem:appendix:weak_norm_power_T_new_1}.
\end{proof}
 
Third, the following result will be useful when the kernel $K_{\mu, k}$ decays faster than expected from the regularity of $\mu$: in that case, one needs to convert the norm given by the remainder estimates into the norm adapted to the decay of $K_{\mu, k}$.

\begin{lemma}\label{lem:remainder_K_decays_too_fast}
    Let $a, b \in \mathbb{R}$ with $a < b$, and let $\delta \geq 0$. Then, there exists $C > 0$ such that for all $f \in H^{b+2\delta}(\mathbb{R})$, one has
    \begin{equation}\label{eq:lem:remainder_K_decays_too_fast_1}
        \left\Vert f \right\Vert_{H^a(\mathbb{R})}^2 \left\Vert f \right\Vert_{H^b(\mathbb{R})} 
        \leq C \left\Vert f \right\Vert_{H^{a-\delta}(\mathbb{R})}^2 \left\Vert f \right\Vert_{H^{b+2\delta}(\mathbb{R})}.
    \end{equation}
    In particular, for $s \in [-1, 1]$, $\sigma > -1 - \frac{s}{2}$ and $s_0 \geq s$, with $\sigma + s_0 - s < \frac{1}{2}$ there exists $C > 0$ such that for all $T>0$ and $u \in \widetilde{H}^{\sigma + s_0 - s}(0,T)$, one has
    \begin{equation}\label{eq:lem:remainder_K_decays_too_fast_2}
        \left\Vert u \right\Vert_{\widetilde{H}^{-1 - \frac{s}{2}}(0, T)}^2 \left\Vert u \right\Vert_{\widetilde{H}^{\sigma}(0, T)}
        \leq C \left\Vert u \right\Vert_{\widetilde{H}^{-1 - \frac{s_0}{2}}(0, T)}^2 \left\Vert u \right\Vert_{\widetilde{H}^{\sigma + s_0 - s}(0, T)}.
    \end{equation}
\end{lemma}
 
\begin{proof}
    In this proof, we write $\left\Vert f \right\Vert_{r} := \left\Vert f \right\Vert_{H^r(\mathbb{R})}$, for $r \in \mathbb{R}$. We can assume that $\delta > 0$. We use the standard consequence of the Hölder inequality
    \begin{equation}
        \left\Vert f \right\Vert_{r} \leq \left\Vert f \right\Vert_p^{\frac{q-r}{q-p}} \left\Vert f \right\Vert_{q}^{\frac{r-p}{q-p}},
    \end{equation}
    which is valid for all $p < r < q$. It gives
    \begin{equation}
        \left\Vert f \right\Vert_{a}^2 \left\Vert f \right\Vert_{b} 
        \leq \left\Vert f \right\Vert_{a-\delta}^{\frac{2\left( b-a+2\delta \right)}{b-a+3\delta}} \left\Vert f \right\Vert_{b+2\delta}^{\frac{2\delta}{b-a+3\delta}}
                \left\Vert f \right\Vert_{a-\delta}^{\frac{2\delta}{b-a+3\delta}} \left\Vert f \right\Vert_{b+2\delta}^{\frac{b-a+\delta}{b-a+3\delta}}
        = \left\Vert f \right\Vert_{a-\delta}^2 \left\Vert f \right\Vert_{b+2\delta}.
    \end{equation}
     Finally, \eqref{eq:lem:remainder_K_decays_too_fast_2} follows from \eqref{eq:lem:remainder_K_decays_too_fast_1}, applied with $\delta = \frac{s_0 - s}{2}$ and $f = \mathbbm{1}_{[0,T]} u$.
\end{proof}
 
\subsection{Removing exponential factors}

In this section, we provide some results which will be used to estimate quantities of the form $ \Vert u e^{\lambda \cdot} \Vert_{\widetilde{H}^{-\sigma}}$ and $\int_0^T u(t) e^{\lambda t} \dd t$. For small $T > 0$, we will use the following lemma.

\begin{lemma}\label{lem:appendix:removing_exp_from_sobolev_norm_small_time}
    Let $s \in [-1, 1)$ and $\lambda \in \mathbb{R}$. Set $\rho(t) = e^{\lambda t}$. There exist $C > 0$ and $\nu > 0$ such that for all $T \in (0, 1]$ and all $u \in L^2(0,T)$, one has
    \begin{equation}\label{eq:lem:appendix:sobolev_norm_rho_small_time}
        \left\vert \int_0^T \left( \rho(t) - 1 \right) u(t) \dd t \right\vert 
        + \left\vert \Vert u \rho \Vert_{\widetilde{H}^{-1 - \frac{s}{2}}(0,T)} - \Vert u \Vert_{\widetilde{H}^{-1 - \frac{s}{2}}(0,T)} \right\vert 
        \leq C T^{\nu} \Vert u \Vert_{\widetilde{H}^{-1 - \frac{s}{2}}(0,T)}.
    \end{equation}
\end{lemma}

\begin{proof}
    In this proof, $\lesssim$ is used for constants that may depend on $\lambda$ and $s$. Let $T \in (0, 1]$ and $u \in L^2(0,T)$. Let $\psi \in C^\infty_\mathrm{c}((-2, 2); [0, 1])$ be such that $\psi = 1$ on $[0, 1]$. Then $\chi := (\rho - 1) \psi\left( \frac{\cdot}{T} \right)$ is an extension of $(\rho - 1) \mathbbm{1}_{[0, T]}$. On the one hand, Lemma \ref{lem:product_sobolev_classical} gives
    \begin{equation}
        \left\vert \Vert u \rho \Vert_{\widetilde{H}^{-1 - \frac{s}{2}}} - \Vert u \Vert_{\widetilde{H}^{-1 - \frac{s}{2}}} \right\vert 
        \leq \left\Vert (\rho - 1) u \right\Vert_{\widetilde{H}^{-1 - \frac{s}{2}}}
        \lesssim \Vert u \Vert_{\widetilde{H}^{-1 - \frac{s}{2}}} \Vert \chi \Vert_{H^{\theta_s}(\mathbb{R})},
    \end{equation}
    where $\theta_s := 1 + \frac{s}{2}$ if $s \in (-1,1)$, and (for instance) $\theta_s := 1$ if $s = -1$. On the other hand, by duality, one has
    \begin{equation}
        \left\vert \int_0^T \left( \rho(t) - 1 \right) u(t) \dd t \right\vert 
        = \left\vert \int_{\mathbb{R}} \left( \mathbbm{1}_{[0,T]}(t) u(t) \right) \chi(t) \dd t \right\vert 
        \lesssim \Vert u \Vert_{\widetilde{H}^{-1 - \frac{s}{2}}} \Vert \chi \Vert_{H^{1 + \frac{s}{2}}(\mathbb{R})}.
    \end{equation}
    We prove below that     
    \begin{equation}\label{eq:proof:lem:appendix:sobolev_norm_rho_small_time_2}
        \left\Vert \chi  \right\Vert_{H^{1 + \frac{s}{2}}(\mathbb{R})}^2 \lesssim T^{1-s},
    \end{equation}
    which yields \eqref{eq:lem:appendix:sobolev_norm_rho_small_time}.
    
    To prove \eqref{eq:proof:lem:appendix:sobolev_norm_rho_small_time_2}, we use the elementary estimate $\vert e^x - 1 \vert \leq \vert x \vert e^{\vert x \vert}$, which gives
    \begin{equation}
        \left\Vert \chi \right\Vert_{L^2(\mathbb{R})}^2 = \int_{-2T}^{2T} \left( e^{\lambda t} - 1 \right)^2 \psi\left( \frac{t}{T} \right)^2 \dd t 
        \lesssim \int_{-2T}^{2T} t^2 \dd t \lesssim T^3,
    \end{equation}
    \begin{equation}\label{eq:proof:lem:appendix:sobolev_norm_rho_small_time}
        \begin{split}
            \left\Vert \chi \right\Vert_{H^1(\mathbb{R})}^2
            & \lesssim T^3 + \int_{-2T}^{2T} \left( e^{\lambda t} - 1 \right)^2 \frac{1}{T^2} \psi^\prime\left( \frac{t}{T} \right)^2 \dd t + \int_{-2T}^{2T} e^{2 \lambda t} \psi\left( \frac{t}{T} \right)^2 \dd t \\
            & \lesssim T^3 + \frac{1}{T^2} \int_{-2T}^{2T} t^2 \dd t + T \int_{-2}^{2} \psi(t)^2 \dd t 
            \lesssim T,
        \end{split}
    \end{equation}
    and similarly $\left\Vert \chi \right\Vert_{H^2(\mathbb{R})}^2 \lesssim \frac{1}{T}$. By interpolation, one obtains \eqref{eq:proof:lem:appendix:sobolev_norm_rho_small_time_2}.
\end{proof}

For large $T > 0$, we will use the following two lemmas.

\begin{lemma}\label{lem:appendix:removing_exp_from_sobolev_norm_large_time}
    Let $\rho \in C^\infty([0, T])$, $\sigma \in \left[ \frac{1}{2}, \frac{3}{2} \right]$, and $T > 0$. There exists $C_T > 0$, which may depend on $\rho$, $\sigma$ and $T$, such that $\left\Vert u \rho \right\Vert_{\widetilde{H}^{-\sigma}} \leq C_T \left\Vert u \right\Vert_{\widetilde{H}^{-\sigma}}$.
\end{lemma}

\begin{proof}
    Let $\rho_0 \in C^\infty_{\mathrm{c}}(\mathbb{R})$ be an extension of $\rho$. By Lemma \ref{lem:product_sobolev_classical}, one has
    \begin{equation}
        \left\Vert u \rho \right\Vert_{\widetilde{H}^{-\sigma}} 
        = \left\Vert \left( u \mathbbm{1}_{[0,T]} \right) \rho_0 \right\Vert_{H^{-\sigma}(\mathbb{R})}
        \leq C \left\Vert u \right\Vert_{\widetilde{H}^{-\sigma}} \left\Vert \rho_0 \right\Vert_{H^{2}(\mathbb{R})}
        \leq C^\prime \left\Vert u \right\Vert_{\widetilde{H}^{-\sigma}},
    \end{equation}
    for some $C, C^\prime > 0$ which depend on $T$, $\rho$ and $\sigma$.
\end{proof}

\begin{lemma}\label{lem:u_1_finite_time}
    Let $\lambda \in \mathbb{R}$, $\sigma \geq 0$, and $T> 0$. There exists $C_T > 0$ such that for all $u \in L^2(0, T)$, one has
    \begin{equation}
        \left\vert \int_0^T u(t) e^{\lambda t} \dd t \right\vert \leq C_T \Vert u \Vert_{\widetilde{H}^{- \sigma}(0,T)}.
    \end{equation}
\end{lemma}

\begin{proof}
    In this proof, $\lesssim$ is used for constants that may depend on $\lambda$, $\sigma \geq 0$, and $T> 0$. Let $\chi \in C^\infty_\mathrm{c}(\mathbb{R}, \mathbb{R})$ be such that $\chi = 1$ on $[0, T]$. Recall that we write $u = \mathbbm{1}_{[0, T]} u$. Lemma \ref{lem:u_1_finite_time} follows from
    \begin{equation}
        \left\vert \int_0^T u(t) e^{\lambda t} \dd t \right\vert 
        = \left\vert \int_\mathbb{R} u(t) e^{\lambda t} \chi(t) \dd t \right\vert 
        \lesssim \Vert u \Vert_{H^{- \sigma}(\mathbb{R})} \Vert e^{\lambda \cdot} \chi \Vert_{H^{\sigma}(\mathbb{R})} 
        \lesssim \Vert u \Vert_{\widetilde{H}^{- \sigma}(0,T)}.
    \end{equation}
\end{proof}

\section{Well-posedness and regularity results for the heat equation}\label{sec:appendix_heat}

\subsection{Well-posedness results for the heat equation}

We provide here some regularity results for the linear heat equation, with constants independent of $T$. First, we prove that solutions of the heat equation with source term of the form $\partial_x f$, for some $f \in L^1((0,T);L^2(0,1))$, belong to $L^2((0, T); L^2(0,1)) \cap L^\infty((0, T); H^{-1}(0,1))$. We give a self-contained proof using Fourier analysis to show that the continuity constant is independent of $T$. 
 
\begin{lemma}[Heat equations with source $\partial_x f$ for $f \in L^1(L^2)$]\label{lem:well-posedness_heat_derivative_L1L2}
    Let $T > 0$. There exists an absolute constant $C > 0$ such that the following properties hold.
    If $f \in L^1((0,T);L^2(0,1))$, then there exists a unique solution $y$ of
    \begin{equation}\label{eq:appendix_heat_weak}
        \left\{
        \begin{array}{lll}
            \partial_t y - \partial_x^2 y  = \partial_x f
            & \quad t \in (0, T), & \quad x \in (0,1), \\
            y(t,0) = y(t,1) = 0
            & & \quad t \in (0, T), \\
            y(0,x)=0
            & & \quad x \in (0, 1),
        \end{array}
        \right.
    \end{equation}
    satisfying $y \in L^2((0, T); L^2(0,1))$. Moreover, one has
    \begin{equation}\label{eq:lem:well-posedness_heat_derivative_L1L2}
        \Vert y \Vert_{L^2((0, T);L^2)} + \Vert y \Vert_{L^\infty((0, T);H^{-1})} \leq C \Vert f \Vert_{L^1((0, T); L^2)}.
    \end{equation}
\end{lemma}

\begin{proof}
    We start with the estimate of $\Vert y \Vert_{L^2((0, T);L^2)}$. The Duhamel formula gives
    \begin{equation}
        y(t) = \int_0^t F(s,t) \dd s \quad \text{ with } F(s,t) := \sum_{j \geq 1} e^{-\lambda_j (t-s)} \left\langle \partial_x f(s) , \varphi_j \right\rangle \varphi_j.
    \end{equation}
    Starting from
    \begin{equation}
        \Vert y(t) \Vert_{L^2(0,1)} \leq \int_0^T \mathbbm{1}_{s \in (0,t)} \Vert F(s,t) \Vert_{L^2(0,1)} \dd s
    \end{equation}  
    and applying Minkowski's integral formula (or simply expressing the $L^2(0,T)$-norm by duality and using Fubini and Cauchy-Schwarz), one obtains
    \begin{align}
        \Vert y \Vert_{L^2((0,T);L^2)} & \leq \left( \int_0^T \left( \int_0^T \mathbbm{1}_{s\in(0,t)} \Vert F(s,t) \Vert_{L^2(0,1)} \dd s \right)^2 \dd t \right)^{\frac{1}{2}} \\
        & \leq  \int_0^T \left( \int_0^T \mathbbm{1}_{s \in (0,t)} \Vert F(s,t) \Vert_{L^2(0,1)}^2 \dd t \right)^{\frac{1}{2}} \dd s. \label{eq:proof:lem:well-posedness_heat_derivative_L1L2}
    \end{align}
    Integrating by parts, one finds
    \begin{equation}
        \Vert F(s,t) \Vert_{L^2(0,1)}^2 = \sum_{j \geq 1} e^{-2 \lambda_j (t-s)} \left\langle f(s) , \varphi_j^\prime \right\rangle^2,
    \end{equation}
    yielding
    \begin{equation}
        \int_0^T \mathbbm{1}_{s \in (0,t)} \Vert F(s,t) \Vert_{L^2(0,1)}^2 \dd t \leq \sum_{j \geq 1} \left\langle f(s) , \varphi_j^\prime \right\rangle^2 \int_s^T e^{-2 \lambda_j (t-s)} \dd t \leq \frac{1}{2} \sum_{j \geq 1} \frac{1}{\lambda_j} \left\langle f(s) , \varphi_j^\prime \right\rangle^2.
    \end{equation}
    Since the family $\left( \frac{\varphi_j^\prime}{\sqrt{\lambda_j}}\right)$ is orthonormal, this gives
    \begin{equation}
        \int_0^T \mathbbm{1}_{s \in (0,t)} \Vert F(s,t) \Vert_{L^2(0,1)}^2 \dd t \leq \frac{1}{2} \Vert f(s) \Vert_{L^2(0,1)}^2.
    \end{equation}
    Together with \eqref{eq:proof:lem:well-posedness_heat_derivative_L1L2}, this gives the desired estimate of $\Vert y \Vert_{L^2((0, T);L^2)}$ in \eqref{eq:lem:well-posedness_heat_derivative_L1L2} with $C = \frac{1}{\sqrt{2}}$.
    
    Now, we estimate $\Vert y \Vert_{L^\infty((0, T);H^{-1})}$. Fix $t \in [0,T]$. One has 
    \begin{equation}
        \Vert y(t) \Vert_{H^{-1}(0,1)}^2 = \sum_{j \geq 1} \left\vert \int_0^t h_j(s,t) \dd s \right\vert^2, \text{ with } h_j(s,t) := \frac{e^{-\lambda_j (t-s)}}{\sqrt{\lambda_j}} \left\langle f(s) , \varphi_j^\prime \right\rangle .
    \end{equation}
    Using duality, Fubini and Cauchy-Schwarz, one obtains
    \begin{equation}
        \begin{split}
            \Vert y(t) \Vert_{H^{-1}(0,1)} 
            & = \sup \left\{ \sum_{j \geq 1} \left( \int_0^t h_j(s,t) \dd s \right) a_j, \ (a_j) \in \ell^2(\mathbb{N}^\ast), \ \sum_{j\geq 1} a_j^2 \leq 1 \right\} \\
            & \leq \sup \left\{ \int_0^t \left( \sum_{j\geq 1} a_j^2 \right)^{\frac{1}{2}} \left( \sum_{j \geq 1}  h_j(s,t)^2  \right)^{\frac{1}{2}} \dd s , \ (a_j) \in \ell^2(\mathbb{N}^\ast), \ \sum_{j\geq 1} a_j^2 \leq 1 \right\} \\
            & = \int_0^t  \left( \sum_{j \geq 1}  h_j(s,t)^2  \right)^{\frac{1}{2}} \dd s .
        \end{split}
    \end{equation}
    As above, since the family $\left( \frac{\varphi_j^\prime}{\sqrt{\lambda_j}}\right)$ is orthonormal, one has
    \begin{equation}
        \sum_{j \geq 1}  h_j(s,t)^2  \leq \sum_{j \geq 1}  \left\langle f(s) , \frac{\varphi_j^\prime}{\sqrt{\lambda_j}} \right\rangle^2 \leq \Vert f(s) \Vert_{L^2(0,1)}^2.
    \end{equation}
    This completes the proof of \eqref{eq:lem:well-posedness_heat_derivative_L1L2}.
\end{proof}

\ifarxiv
    Second, we prove a well-posedness result for the linear heat equation with source term $f \in L^2((0, T), H^{-1}(0,1))$, which will be the main tool for the proof of the well-posedness of the controlled Burgers equation \eqref{eq:Burgers_intro} (see Lemma \ref{lem:burgers_weak_wellposedness}).
\fi

\begin{lemma}[Heat equation with source in $L^2(H^{-1})$]\label{lem:heat_weak_wellposedness}
    For $T > 0$, $f \in L^2((0, T), H^{-1}(0,1))$ and $y_0 \in L^2(0,1)$, there exists a unique solution 
    \begin{equation}\label{eq:lem:Heat_weak_wellposedness_0}
        y \in C^0([0, T]; L^2(0,1)) \ \cap \ L^2((0,T); H_0^1(0,1)) \ \cap \ H^1((0,T); H^{-1}(0,1))
    \end{equation} 
    of the heat equation
    \begin{equation}\label{eq:lem:heat_weak_wellposedness}
        \left\{
        \begin{array}{cll}
            \partial_t y - \partial_x^2 y = f
            & \quad t \in (0, T), & \quad x \in (0,1), \\
            y(t,0) = y(t,1) = 0
            & & \quad t \in (0, T), \\
            y(0,x)=y_0
            & & \quad x \in (0, 1).
        \end{array}
        \right.
    \end{equation}
    In addition, there exists an absolute constant $C > 0$ such that
    \begin{equation}\label{eq:lem:Heat_weak_wellposedness_1}
        \Vert y \Vert_{L^\infty((0,T); L^2)} + \Vert y \Vert_{L^2((0,T); H_0^1)} + \Vert y \Vert_{H^{1}((0,T); H^{-1})} \leq C \left( \Vert y_0 \Vert_{L^2} + \Vert f \Vert_{L^2((0, T); H^{-1})} \right).
    \end{equation}
\end{lemma}

\ifarxiv
\begin{proof}
    In this proof, the symbol $\lesssim$ denotes an inequality with an absolute constant. We only recall why the solution $y$ given by the formula
    \begin{equation}
        y(t) = \sum_{k \geq 1} \left( e^{- \lambda_k t} \left\langle y_0, \varphi_k \right\rangle + \int_0^t e^{- \lambda_k (t-s)} \left\langle f(s), \varphi_k \right\rangle \dd s \right) \varphi_k 
    \end{equation}
    satisfies \eqref{eq:lem:Heat_weak_wellposedness_0} and \eqref{eq:lem:Heat_weak_wellposedness_1}.
    
    For $k \geq 1$ and $t \in [0, T]$, the Cauchy-Schwarz inequality gives
    \begin{equation}
        \left\vert \langle y(t), \varphi_k \rangle \right\vert^2 \lesssim \left\vert \langle y_0, \varphi_k \rangle \right\vert^2 + \frac{1}{\lambda_k} \int_0^t  \left\vert \langle f(s), \varphi_k \rangle \right\vert^2 \dd s,
    \end{equation}
    yielding 
    \begin{equation}\label{eq:proof:lem:heat_weak_wellposedness_1}
        \Vert y \Vert_{L^\infty((0, T);L^2)}^2 \lesssim \Vert y_0 \Vert_{L^2}^2 + \Vert f \Vert_{L^2((0, T);H^{-1})}^2.
    \end{equation}
    To get an estimate in $L^2((0,T); H_0^1(0,1))$, we use Young's convolution inequality, which gives
    \begin{align}
        \int_0^T \left\Vert \partial_x y(t) \right\Vert_{L^2}^2 \dd t 
        & \lesssim \sum_{k \geq 1} \lambda_k \int_0^T e^{-2 \lambda_k t} \langle y_0, \varphi_k \rangle^2 \dd t \\
        & \hspace{0.5cm} + \sum_{k \geq 1} \lambda_k \int_0^T \left( \int_0^t e^{- \lambda_k (t-s)} \langle f(s), \varphi_k \rangle \dd s \right)^2 \dd t \\
        & \lesssim \Vert y_0 \Vert_{L^2}^2 + \sum_{k \geq 1} \lambda_k \left( \int_0^T \langle f(t), \varphi_k \rangle^2 \dd t \right) \left( \int_0^\infty e^{-\lambda_k t} \dd t \right)^2 \\
        & \lesssim \Vert y_0 \Vert_{L^2}^2 + \Vert f \Vert_{L^2((0, T);H^{-1})}^2. \label{eq:proof:lem:heat_weak_wellposedness_2}
    \end{align}
    Finally, using the first line of \eqref{eq:lem:heat_weak_wellposedness}, one finds
    \begin{align}
        \Vert \partial_t y \Vert_{L^2((0, T);H^{-1})} & \lesssim \Vert \partial_x^2 y \Vert_{L^2((0, T);H^{-1})} + \Vert f \Vert_{L^2((0, T);H^{-1})} \\
        & \lesssim \Vert y \Vert_{L^2((0, T);H^1)} + \Vert f \Vert_{L^2((0, T);H^{-1})}, \label{eq:proof:lem:heat_weak_wellposedness_3}
    \end{align}
    and this gives \eqref{eq:lem:Heat_weak_wellposedness_0}. Note that \eqref{eq:lem:Heat_weak_wellposedness_1} follows from \eqref{eq:proof:lem:heat_weak_wellposedness_1}, \eqref{eq:proof:lem:heat_weak_wellposedness_2} and \eqref{eq:proof:lem:heat_weak_wellposedness_3}.
\end{proof}
\fi

Finally, we recall, without proof, the following classical result. The fact that the constants are independent of $T$ can be verified, for instance, by a straightforward inspection of \cite[Lemma A.1]{Nguyen2025Burgers}. 

\begin{lemma}\label{lem:well-posedness_heat_L1L2}
    If $f \in L^1((0, T); L^2(0, 1))$ and $y_0 \in L^2(0, 1)$, then there exists a unique solution $y$ of \eqref{eq:lem:heat_weak_wellposedness} satisfying
    \begin{equation}
        y \in C^0([0, T]; L^2(0,1)) \cap L^2((0, T); H_0^1(0,1)).
    \end{equation}
    In addition, there exists $C > 0$ independent of $T$, such that
    \begin{equation}\label{eq:lem:well-posedness_1}
        \Vert y \Vert_{L^\infty((0, T);L^2)} + \Vert y \Vert_{L^2((0, T);H_0^1)} \leq C \left( \Vert y_0 \Vert_{L^2} + \Vert f \Vert_{L^1((0, T); L^2)} \right).
    \end{equation}
\end{lemma} 
 
\subsection{Bilinear parabolic estimates}

We prove here the bilinear parabolic estimates used in the proof of cubic remainder estimates when $s > 0$. Recall that the spaces $X_p^s$, $Y_p^s$ and $Z_p^s$ are introduced around \eqref{eq:def_X_s_2} and \eqref{eq:def_X_s_infty}, and that $\mathcal{B}$ is defined by \eqref{eq:def:mathcalB}. 

\begin{lemma}\label{lem:parabolic_estimates_anisotropic_sobolev}
   Let $s \in (0, 1)$, $p \in \{2, +\infty \}$ and $k \geq 1$. There exists $C>0$ that depends only on $s$ and $p$ such that for all $T >0$, and all $f,g$, one has
   \begin{equation}\label{eq:lem:parabolic_estimates_anisotropic_sobolev_2}
       \left\Vert \mathcal{B}(f,g) \right\Vert_{Z^s_p} \leq C \left\Vert f \right\Vert_{Z^s_p} \left\Vert g \right\Vert_{Z^s_p},
   \end{equation}
   and
   \begin{equation}\label{eq:lem:parabolic_estimates_anisotropic_sobolev_1}
       \begin{split}
           &  \left\Vert \mathcal{B}(f,g) \right\Vert_{X^s_p} + \left\Vert \mathcal{B}(f,g) \right\Vert_{Y^s_p} + \left\Vert \mathcal{B}(f,g) \right\Vert_{L^2((0,T);L^2)}
           \leq  C \left\Vert f \right\Vert_{X^s_p} \left\Vert g \right\Vert_{Z^s_p}.
       \end{split}
   \end{equation}
   In addition, for $k \geq 1$, there exists $C > 0$ that depends only on $s$, $p$ and $k$, such that for all $T >0$, and all $f,g$, one has
   \begin{equation}\label{eq:lem:parabolic_estimates_anisotropic_sobolev_3}
       \left\vert \left\langle \mathcal{B}(f,g)(T), \varphi_k \right\rangle \right\vert \leq C \left\Vert f \right\Vert_{X^s_p} \left\Vert g \right\Vert_{Y^s_p} .
   \end{equation}
\end{lemma}

\begin{proof}
    We use $\lesssim$ for constants that may depend only on $s$ and $k$. Write $F_n(\tau) := \left\langle f(\tau) g(\tau), \varphi_n^\prime \right\rangle$, so that 
    \begin{equation}
         \left\langle \mathcal{B}(f,g)(t), \varphi_n \right\rangle = - \int_0^t e^{-\lambda_n(t-\tau)} F_n(\tau) \dd \tau.
    \end{equation}
    We still write $f$ and $g$ for their extension by zero outside $[0,T]$.
    
    \textbf{Step 1: we prove \eqref{eq:lem:parabolic_estimates_anisotropic_sobolev_1} in the case $p=2$.}
    Let $\alpha$ be such that $s < \alpha < s + \frac{1}{4}$ if $s \in \left(0, \frac{1}{2} \right)$, and $\frac{1}{2} < \alpha < 1 - \frac{s}{2}$ if $s \in \left[ \frac{1}{2}, 1 \right)$. One has $\pm \frac{s}{2} \leq \alpha \leq 1 \pm \frac{s}{2}$, so that Lemma \ref{lem:easy_integral_with_exp_lambda} implies
    \begin{equation}
        \left\Vert t \mapsto \int_0^t e^{-\lambda_n(t-\tau)} F_n(\tau) \dd \tau \right\Vert_{\widetilde{H}^{\pm\frac{s}{2}}(0,T)}^2 
        \lesssim \frac{1}{\lambda_n^{-2\alpha \mp s + 2}} \left\Vert F_n \right\Vert_{H^{-\alpha}(\mathbb{R})}^2,
    \end{equation}
    yielding
    \begin{equation}
        \begin{split}
            \left\Vert \mathcal{B}(f,g) \right\Vert_{\widetilde{H}^{\pm \frac{s}{2}}((0,T);\mathcal{H}^{\mp s}_2)}^2 
            = & \sum_{n \geq 1} \lambda_n^{\mp s} \left\Vert t \mapsto \int_0^t e^{-\lambda_n(t-\tau)} F_n(\tau) \dd \tau \right\Vert_{\widetilde{H}^{\pm \frac{s}{2}}(0,T)}^2 \\
            \lesssim & \sum_{n \geq 1} \frac{1}{\lambda_n^{2-2\alpha}} \left\Vert F_n \right\Vert_{H^{-\alpha}(\mathbb{R})}^2.
        \end{split}
    \end{equation}
    Arguing similarly, one finds
    \begin{equation}
        \left\Vert \mathcal{B}(f,g) \right\Vert_{L^2((0,T);L^2)}^2 \lesssim \sum_{n \geq 1} \frac{1}{\lambda_n^{2-2\alpha}} \left\Vert F_n \right\Vert_{H^{-\alpha}(\mathbb{R})}^2.
    \end{equation}
    Hence, it suffices to estimate $\sum_{n \geq 1} \frac{1}{\lambda_n^{2-2\alpha}} \left\Vert F_n \right\Vert_{H^{-\alpha}(\mathbb{R})}^2$. The triangular inequality gives
    \begin{equation}
        \sum_{n \geq 1} \frac{1}{\lambda_n^{2-2\alpha}} \left\Vert F_n \right\Vert_{H^{-\alpha}(\mathbb{R})}^2
        \lesssim \sum_{n \geq 1} \frac{1}{\lambda_n^{2-2\alpha}} \left( \sum_{p, q \geq 1} \left\vert \left\langle \varphi_p \varphi_q, \varphi_n^\prime  \right\rangle \right\vert \left\Vert \langle f, \varphi_p \rangle \langle g, \varphi_q \rangle \right\Vert_{H^{-\alpha}(\mathbb{R})} \right)^2.
    \end{equation}
    If $s \in \left[ \frac{1}{2}, 1 \right)$, then Lemma \ref{lem:product_sobolev_classical} (in the case $a+b=0$) gives
    \begin{equation}
         \left\Vert \langle f, \varphi_p \rangle \langle g, \varphi_q \rangle \right\Vert_{H^{-\alpha}(\mathbb{R})} 
         \lesssim  \left\Vert \langle f, \varphi_p \rangle \right\Vert_{H^{-\frac{s}{2}}(\mathbb{R})}  \left\Vert \langle g, \varphi_q \rangle \right\Vert_{H^{\frac{s}{2}}(\mathbb{R})},
    \end{equation}
    since $-\alpha < - \frac{1}{2}$ and $-\frac{s}{2} \geq - \alpha$. If $s \in \left(0, \frac{1}{2} \right)$, then Lemma \ref{lem:product_sobolev_classical} gives
    \begin{equation}
         \left\Vert \langle f, \varphi_p \rangle \langle g, \varphi_q \rangle \right\Vert_{H^{-\alpha}(\mathbb{R})} 
         \lesssim  \left\Vert \langle f, \varphi_p \rangle \right\Vert_{H^{-\frac{s}{2}}(\mathbb{R})} \left\Vert \langle g, \varphi_q \rangle \right\Vert_{H^{\frac{1-s}{2}}(\mathbb{R})},
    \end{equation}
    since $-\frac{s}{2} \geq -\alpha$, $\frac{1-s}{2} \geq -\alpha$, and $-\frac{s}{2} + \frac{1-s}{2} - \frac{1}{2} > - \alpha$, and $-\frac{s}{2} + \frac{1-s}{2} > 0$. Hence, in any case, one obtains
    \begin{equation}
        \sum_{n \geq 1} \frac{1}{\lambda_n^{2-2\alpha}} \left\Vert F_n \right\Vert_{H^{-\alpha}(\mathbb{R})}^2 
        \lesssim  \sum_{n \geq 1} \frac{1}{\lambda_n^{2-2\alpha}} \left( \sum_{p, q \geq 1} \left\vert \left\langle \varphi_p \varphi_q, \varphi_n^\prime  \right\rangle \right\vert \left\Vert \langle f, \varphi_p \rangle \right\Vert_{H^{-\frac{s}{2}}(\mathbb{R})} \left\Vert \langle g, \varphi_q \rangle \right\Vert_{H^{1+\sigma(s)}(\mathbb{R})} \right)^2.
    \end{equation}
    By Lemma \ref{lem:product_kind_of_discrete_sobolev}, this gives \eqref{eq:lem:parabolic_estimates_anisotropic_sobolev_1}.
    
    \textbf{Step 2: we prove \eqref{eq:lem:parabolic_estimates_anisotropic_sobolev_2} in the case $p=2$.}
    Let $\beta$ be such that $\frac{1}{4} - s < \beta < \frac{1}{2} - s$ and $\beta > 0$ if $s \in \left(0, \frac{1}{2} \right)$, and $\frac{3s}{2} - 1 < \beta < s - \frac{1}{2}$ if $s \in \left[ \frac{1}{2}, 1 \right)$. Arguing as above, using Lemma \ref{lem:easy_integral_with_exp_lambda} with $\beta \in [\sigma(s), 1+\sigma(s)]$, one finds
    \begin{equation}
        \left\Vert \mathcal{B}(f,g) \right\Vert_{Z_2^s}^2 
        \lesssim \sum_{n \geq 1} \frac{1}{\lambda_n^{2 \beta - 2 \sigma(s) - s}} \left\Vert F_n \right\Vert_{H^{\beta}(\mathbb{R})}^2.
    \end{equation}
    Using  Lemma \ref{lem:product_sobolev_classical} with $1 + \sigma(s) \geq \beta$, $2(1+\sigma(s)) -\frac{1}{2} > \beta$, and $2(1+\sigma(s)) > 0$, one obtains
    \begin{equation}
         \left\Vert \langle f, \varphi_p \rangle \langle g, \varphi_q \rangle \right\Vert_{H^{\beta}(\mathbb{R})} 
         \lesssim  \left\Vert \langle f, \varphi_p \rangle \right\Vert_{H^{1+\sigma(s)}(\mathbb{R})}  \left\Vert \langle g, \varphi_q \rangle \right\Vert_{H^{1+\sigma(s)}(\mathbb{R})},
    \end{equation}
    yielding
    \begin{equation}
        \begin{split}
            & \left\Vert \mathcal{B}(f,g) \right\Vert_{Z_2^s}^2 
            \lesssim  \sum_{n \geq 1} \frac{1}{\lambda_n^{2 \beta - 2 \sigma(s) - s}} \left( \sum_{p, q \geq 1} \left\vert \left\langle \varphi_p \varphi_q, \varphi_n^\prime  \right\rangle \right\vert \left\Vert \langle f, \varphi_p \rangle \right\Vert_{H^{1+\sigma(s)}(\mathbb{R})} \left\Vert \langle g, \varphi_q \rangle \right\Vert_{H^{1+\sigma(s)}(\mathbb{R})} \right)^2.
        \end{split}
    \end{equation}
    Hence, \eqref{eq:lem:parabolic_estimates_anisotropic_sobolev_2} follows from Lemma \ref{lem:product_kind_of_discrete_sobolev} applied with $\alpha := 1 - \beta + \sigma(s) + \frac{s}{2}$, which satisfies $s < \alpha < s + \frac{1}{4}$ if $s \in \left(0, \frac{1}{2} \right)$, and $\frac{1}{2} < \alpha < 1 - \frac{s}{2}$ if $s \in \left[ \frac{1}{2}, 1 \right)$.
    
    \textbf{Step 3: we prove \eqref{eq:lem:parabolic_estimates_anisotropic_sobolev_3} in the case $p=2$.}
    By definition, one has
    \begin{equation}
        \left\vert \left\langle \mathcal{B}(f,g)(T), \varphi_k \right\rangle \right\vert 
        \lesssim  \sum_{p,q \geq 1} \left\vert \left\langle \varphi_p \varphi_q, \varphi_k^\prime \right\rangle \right\vert 
        \left\vert \int_0^T e^{-\lambda_k (T-t)} \langle f(t), \varphi_p \rangle \langle g(t), \varphi_q \rangle \dd t \right\vert.
    \end{equation}
    Using duality and Lemma \ref{lem:product_sobolev_classical}, one finds
    \begin{equation}
        \begin{split}
            \left\vert \int_0^T e^{-\lambda_k (T-t)} \langle f(t), \varphi_p \rangle \langle g(t), \varphi_q \rangle \dd t \right\vert
            & \lesssim \left\Vert \mathbbm{1}_{(0,T)} e^{-\lambda_k (T-\cdot)} \langle f, \varphi_p \rangle \right\Vert_{H^{-\frac{s}{2}}(\mathbb{R})} \left\Vert \mathbbm{1}_{(0,T)} \langle g, \varphi_q \rangle \right\Vert_{H^{\frac{s}{2}}(\mathbb{R})} \\
            & \lesssim \left\Vert \langle f, \varphi_p \rangle \right\Vert_{\widetilde{H}^{-\frac{s}{2}}} \left\Vert \langle g, \varphi_q \rangle \right\Vert_{\widetilde{H}^{\frac{s}{2}}},
        \end{split}
    \end{equation}
    yielding
    \begin{equation}
         \left\vert \left\langle \mathcal{B}(f,g)(T), \varphi_k \right\rangle \right\vert^2 
        \lesssim  \left\Vert g \right\Vert_{Y_2^s}^2  
         \sum_{q \geq 1} \lambda_q^s \left(  \sum_{p \geq 1} \left\vert \left\langle \varphi_p \varphi_q, \varphi_k^\prime \right\rangle \right\vert 
        \left\Vert \langle f, \varphi_p \rangle \right\Vert_{\widetilde{H}^{-\frac{s}{2}}} \right)^2 
    \end{equation}
    by Cauchy-Schwarz. By \eqref{eq:expression_varphi_varphi_varphiprime}, the sum over $p$ contains at most $3$ terms, and in addition, there exists $C_k > 0$ such that for all $p, q \geq 1$, if $\left\langle \varphi_p \varphi_q, \varphi_k^\prime \right\rangle \neq 0$ then $\lambda_q \leq C_k \lambda_p$. This yields
    \begin{equation}
        \left\vert \left\langle \mathcal{B}(f,g)(T), \varphi_k \right\rangle \right\vert^2 
        \lesssim  \left\Vert g \right\Vert_{Y_2^s}^2  
         \sum_{p \geq 1} \lambda_p^s \left\Vert \langle f, \varphi_p \rangle \right\Vert_{\widetilde{H}^{-\frac{s}{2}}}^2 
        =  \left\Vert g \right\Vert_{Y_2^s}^2  \left\Vert f \right\Vert_{X_2^s}^2 .
    \end{equation}
    
    \textbf{Step 4: Sketch of proof in the case $p = \infty$.}
        The proofs are similar so we only sketch it. We rely on the technical result Lemma \ref{lem:product_kind_of_discrete_sobolev_p_infty}, instead of Lemma \ref{lem:product_kind_of_discrete_sobolev}.
        
        To prove \eqref{eq:lem:parabolic_estimates_anisotropic_sobolev_1}, we argue as in Step 1, with the same parameter $\alpha$. It gives
        \begin{equation}
            \begin{split}
                \left\Vert \mathcal{B}(f,g) \right\Vert_{X_\infty^s} 
                & \lesssim \sup_{n \geq 1} \left( \frac{1}{n^{\frac{3}{2}-2\alpha}} \sum_{p, q \geq 1} \left\vert \left\langle \varphi_p \varphi_q, \varphi_n^\prime  \right\rangle \right\vert\left\Vert \langle f, \varphi_p \rangle \right\Vert_{H^{-\frac{s}{2}}(\mathbb{R})} \left\Vert \langle g, \varphi_q \rangle \right\Vert_{H^{1+\sigma(s)}(\mathbb{R})} \right) \\
                & \lesssim \left\Vert f \right\Vert_{X_\infty^s} \left\Vert g \right\Vert_{Z_\infty^s} \sup_{n \geq 1} \frac{I_{n,s}}{n^{\frac{3}{2}-2\alpha}}, 
            \end{split}
        \end{equation}
        where $I_{n,s}$ is defined in Lemma \ref{lem:product_kind_of_discrete_sobolev_p_infty}. Since $\sup_{n \geq 1} \frac{I_{n,s}}{n^{\frac{3}{2}-2\alpha}} \lesssim 1$ by Lemma \ref{lem:product_kind_of_discrete_sobolev_p_infty} and by definition of $\alpha$, this proves the claimed estimate for $\left\Vert \mathcal{B}(f,g) \right\Vert_{X_\infty^s}$. Similarly, one finds     
        \begin{equation}
            \begin{split}
                \left\Vert \mathcal{B}(f,g) \right\Vert_{Y_\infty^s} + \left\Vert \mathcal{B}(f,g) \right\Vert_{L^2((0,T); L^2)}
                & \lesssim \left\Vert f \right\Vert_{X_\infty^s} \left\Vert g \right\Vert_{Z_\infty^s} 
                \left\{ 
                    \sum_{n \geq 1} \frac{I_{n,s}}{n^{\frac{5}{2}-2\alpha}} 
                    + \left( \sum_{n \geq 1} \frac{I_{n,s}^2}{\lambda_n^{2-2\alpha}} \right)^{\frac{1}{2}}
                \right\} \\
                & \lesssim \left\Vert f \right\Vert_{X_\infty^s} \left\Vert g \right\Vert_{Z_\infty^s}. 
            \end{split}
        \end{equation}
        
        To prove \eqref{eq:lem:parabolic_estimates_anisotropic_sobolev_2}, we argue as in the Step 2, with the same parameter $\beta$. It gives 
        \begin{equation}
            \left\Vert \mathcal{B}(f,g) \right\Vert_{Z_\infty^s} \lesssim \left\Vert f \right\Vert_{Z_\infty^s} \left\Vert g \right\Vert_{Z_\infty^s} \sup_{n \geq 1} \left( n^{\frac{1}{2} + s - 2 \beta + 2 \sigma(s) } I_{n,s} \right) .
        \end{equation}
        Using Lemma \ref{lem:product_kind_of_discrete_sobolev_p_infty} and the definition of $\beta$, one obtains \eqref{eq:lem:parabolic_estimates_anisotropic_sobolev_2}.
        
        Finally, the proof of \eqref{eq:lem:parabolic_estimates_anisotropic_sobolev_3} is similar: arguing as in Step 3, one finds
        \begin{equation}
            \left\vert \left\langle \mathcal{B}(f,g)(T), \varphi_k \right\rangle \right\vert 
            \lesssim \left\Vert f \right\Vert_{X_\infty^s} 
            \sum_{p, q \geq 1} \left\vert \left\langle \varphi_p \varphi_q, \varphi_k^\prime  \right\rangle \right\vert p^{-\frac{1}{2} - s} \left\Vert \langle g, \varphi_q \rangle \right\Vert_{\widetilde{H}^{\frac{s}{2}}} 
            \lesssim \left\Vert f \right\Vert_{X_\infty^s} \left\Vert g \right\Vert_{Y_\infty^s}.  
        \end{equation}
\end{proof}

\section{Well-posedness and regularity results for the Burgers equation}\label{sec:appendix_burgers}

\ifarxiv
    \subsection{Well-posedness for the Burgers equation with a rough source profile}\label{sec:lem:burgers_weak_wellposedness}
    
    Here, we prove Lemma \ref{lem:burgers_weak_wellposedness}. The proof is based on Lemma \ref{lem:heat_weak_wellposedness}, and on the following two lemmas.
    
    \begin{lemma}[Schaefer's fixed point theorem]\label{lem:Schaefer}
        Let $E$ be a Banach space, and let $\Gamma \in C^0(E,E)$ be compact. Assume that there exists $R > 0$ such that 
        \begin{equation}\label{eq:lem:Schaefer}
            \left\{ x \in E, \, \exists \lambda \in [0,1], x = \lambda \Gamma(x) \right\} \subset B_E(0, R),
        \end{equation}
        where $B_E(0, R) \subset E$ denotes the ball of center $0$ and of radius $R$. Then $\Gamma$ has a fixed point $x \in B_E(0, R)$.
    \end{lemma}
    
    \begin{lemma}[A compact embedding]\label{lem:compact_embedding}
        For $T > 0$, set
        \begin{equation}\label{eq:def_Z_T}
            Z_T := C^0([0, T]; L^2(0,1)) \ \cap \ L^2((0,T); H_0^1(0,1)) \ \cap \ H^1((0,T); H^{-1}(0,1)).
        \end{equation}
        Then, the embedding $Z_T \hookrightarrow L^4((0,T) \times (0,1))$ is compact.
    \end{lemma}
    
    Lemma \ref{lem:Schaefer} is standard, and a proof of Lemma \ref{lem:compact_embedding} is given below. 
    We now use Lemmas \ref{lem:heat_weak_wellposedness}, \ref{lem:Schaefer} and \ref{lem:compact_embedding} to prove Lemma \ref{lem:burgers_weak_wellposedness}.
    
    \begin{proof}[Proof of Lemma \ref{lem:burgers_weak_wellposedness}.]
        \textbf{Step 1.} \textit{Existence of the solution.} 
        Set $E = L^4((0,T) \times (0,1))$, and write simply $Z = Z_T$ for the set defined by \eqref{eq:def_Z_T}. For $Y \in E$, we define $\Gamma(Y)$ as the solution $y$ of 
        \begin{equation}\label{eq:burgers_weak_wellposedness_proof_1}
            \left\{
            \begin{array}{cll}
                \partial_t y - \partial_x^2 y = f - \frac{1}{2} \partial_x (Y^2)
                & \quad t \in (0, T), & \quad x \in (0,1), \\
                y(t,0) = y(t,1) = 0
                & & \quad t \in (0, T), \\
                y(0,x)=y_0
                & & \quad x \in (0, 1).
            \end{array}
            \right.
        \end{equation}
        For $Y \in E$, one has $Y^2 \in L^2((0,T); L^2(0,1))$ and hence $\partial_x (Y^2) \in L^2((0, T); H^{-1}(0,1))$, so that Lemma \ref{lem:heat_weak_wellposedness} implies $\Gamma(Y) \in Z$. In addition, $\Gamma: E \rightarrow Z$ is continuous, so that Lemma \ref{lem:compact_embedding} implies that $\Gamma: E \rightarrow E$ is compact. Moreover, note that if $y \in E$ is a fixed point of $\Gamma$, then $y = \Gamma(y) \in Z$, which implies $\frac{1}{2} \partial_x (y^2) = y \partial_x y$, and thus $y$ is a solution of \eqref{eq:burgers_weak_wellposedness}.
    
        We now prove that $\Gamma$ satisfies the assumption \eqref{eq:lem:Schaefer}.
        Let $y \in E$ and $\lambda \in [0,1]$ be such that $y = \lambda \Gamma(y)$. Then $y \in Z$ is the solution of
        \begin{equation}\label{eq:burgers_weak_wellposedness_proof_2}
            \left\{
            \begin{array}{cll}
                \partial_t y - \partial_x^2 y = \lambda f - \lambda y \partial_x y
                & \quad t \in (0, T), & \quad x \in (0,1), \\
                y(t,0) = y(t,1) = 0
                & & \quad t \in (0, T), \\
                y(0,x)= \lambda y_0
                & & \quad x \in (0, 1).
            \end{array}
            \right.
        \end{equation}
        Let $t \in (0,T)$. Note that the first line of \eqref{eq:burgers_weak_wellposedness_proof_2} holds in $L^2((0, T); H^{-1}(0,1))$, and that $y \in L^2((0, T); H_0^1(0,1))$. Hence, one has
        \begin{equation}
            \int_0^t \left\langle \partial_t y - \partial_x^2 y, y \right\rangle_{H^{-1}, H_0^1} \dd s = \lambda \int_0^t \left\langle f, y \right\rangle_{H^{-1}, H_0^1} \dd s - \lambda \int_0^t \left\langle y \partial_x y, y \right\rangle_{H^{-1}, H_0^1} \dd s.
        \end{equation}
        Integrating by parts, one finds
        \begin{equation}
            \int_0^t \left\langle \partial_t y - \partial_x^2 y, y \right\rangle_{H^{-1}, H_0^1} \dd s = \frac{1}{2} \Vert y(t) \Vert_{L^2}^2 - \frac{\lambda^2}{2} \Vert y_0 \Vert_{L^2}^2 + \left\Vert \partial_x y \right\Vert_{L^2((0,t) \times (0,1))}^2.
        \end{equation}
        For almost all $s \in (0,t)$, one has $y(s)^3 \in H_0^1(0,1)$, implying
        \begin{equation}
            \left\langle y(s) \partial_x y(s), y(s) \right\rangle_{H^{-1}, H_0^1} = \frac{1}{3} \int_0^1 \partial_x (y(s)^3) \dd x = 0.
        \end{equation}
        Moreover, there exists an absolute constant $C > 0$ such that 
        \begin{equation}
            \left\vert \int_0^t \left\langle f, y \right\rangle_{H^{-1}, H_0^1} \dd s \right\vert \leq C \Vert f \Vert_{L^2((0,T);H^{-1})} \Vert \partial_x y \Vert_{L^2((0,t) \times (0,1))}. 
        \end{equation}
        Using the elementary estimate $ab \leq \frac{a^2 + b^2}{2}$, this gives
        \begin{equation}\label{eq:burgers_weak_wellposedness_proof_3}
            \frac{1}{2} \Vert y(t) \Vert_{L^2}^2 + \frac{1}{2} \left\Vert \partial_x y \right\Vert_{L^2((0,t) \times (0,1))}^2 \leq \frac{1}{2} \Vert y_0 \Vert_{L^2}^2 + \frac{C^2}{2} \Vert f \Vert_{L^2((0,T);H^{-1})}^2.
        \end{equation}
        Finally, using also the equation satisfied by $y$, one finds
        \begin{equation}\label{eq:burgers_weak_wellposedness_proof_4}
            \Vert y \Vert_{H^{1}((0,T);H^{-1})} \leq C \left( \Vert y_0 \Vert_{L^2} + \Vert f \Vert_{L^2((0,T);H^{-1})} + \Vert y_0 \Vert_{L^2}^2 + \Vert f \Vert_{L^2((0,T);H^{-1})}^2 \right),
        \end{equation}
        for some absolute constant $C>0$. This proves that \eqref{eq:lem:Schaefer} is satisfied with $R$ of the form of the right-hand side of \eqref{eq:burgers_weak_wellposedness_proof_4}. Hence, Lemma \ref{lem:Schaefer} implies that $\Gamma$ has a fixed point $y \in B_E(0,R)$. As explained above, $y \in Z$ and $y$ is a solution of \eqref{eq:burgers_weak_wellposedness}.
    
        \textbf{Step 2.} \textit{Uniqueness of the solution.} 
        Let $y$ and $\widetilde{y}$ be solutions of \eqref{eq:burgers_weak_wellposedness}. Set $z = y + \widetilde{y}$ and $w = y - \widetilde{y}$. Then $w$ is the solution of 
        \begin{equation}
            \left\{
            \begin{array}{cll}
                \partial_t w - \partial_x^2 w = - \frac{1}{2} \partial_x (wz)
                & \quad t \in (0, T), & \quad x \in (0,1), \\
                w(t,0) = w(t,1) = 0
                & & \quad t \in (0, T), \\
                w(0,x)= 0
                & & \quad x \in (0, 1).
            \end{array}
            \right.
        \end{equation}
        Let $t \in (0,T)$. Multiplying by $w$ the equation satisfied by $w$, integrating on $(0, t) \times (0,1)$, and integrating by parts, one finds
        \begin{equation}
            \frac{1}{2} \Vert w(t) \Vert_{L^2}^2 + \int_0^t \left\Vert \partial_x w(s) \right\Vert_{L^2}^2 \dd s = \frac{1}{2} \int_0^t \int_0^1 z w \partial_x w \dd x \dd s.
        \end{equation}
        The embedding $H^1(0,1) \hookrightarrow L^\infty(0,1)$ yields
        \begin{equation}
            \int_0^t \int_0^1 z w \partial_x w \dd x \dd s \leq \frac{1}{2} \int_0^t \left\Vert \partial_x w(s) \right\Vert_{L^2}^2 \dd s + \frac{C}{2} \int_0^t \left\Vert z(s) \right\Vert_{H^1}^2 \left\Vert w(s) \right\Vert_{L^2}^2 \dd s,
        \end{equation}
        for some $C > 0$, and therefore
        \begin{equation}
            \Vert w(t) \Vert_{L^2}^2 \leq \frac{C}{2} \int_0^t \left\Vert z(s) \right\Vert_{H^1}^2 \left\Vert w(s) \right\Vert_{L^2}^2 \dd s.
        \end{equation}
        Note that $\int_0^t \left\Vert z(s) \right\Vert_{H^1}^2 \dd s < + \infty$. Hence, Grönwall's lemma gives $w = 0$.
    \end{proof}
     
    We now prove Lemma \ref{lem:compact_embedding}.
    
    \begin{proof}[Proof of Lemma \ref{lem:compact_embedding}.]
        First, since $H^1(0,1) \hookrightarrow L^\infty(0,1)$ is compact and $L^\infty(0,1) \hookrightarrow L^2(0,1) \hookrightarrow H^{-1}(0,1)$ is continuous, classical compactness results (see, for instance, \cite[Corollary~5]{Simon1986}) imply that 
        \begin{equation}
            L^2((0,T); H_0^1(0,1)) \cap H^1((0,T); H^{-1}(0,1)) 
            \hookrightarrow L^2((0,T); L^\infty(0,1))
        \end{equation}
        is compact. In particular, $Z_T \hookrightarrow L^2((0,T); L^\infty(0,1))$ is compact.
        
        Second, by definition, $Z_T \hookrightarrow L^\infty((0,T); L^2(0,1))$ is continuous. Hence, classical interpolation results (see, for instance, \cite[Theorem~3.8.1]{bergh_interpolation_1976}) imply that 
        \begin{equation}
            Z_T \hookrightarrow 
            \left[ L^2((0,T); L^\infty(0,1)),\, L^\infty((0,T); L^2(0,1)) \right]_{\frac{1}{2}}
            = L^4((0,T); L^4(0,1))
        \end{equation}
        is compact.
    \end{proof}
\fi

\subsection{A continuity estimate for a forced Burgers equation}

We prove the following result, which will be used in the proof of cubic remainder estimates, in the simpler case $\mu \in \mathcal{H}^s_2$ with $s \in [-1, 0]$, and in the proof of the decoupling estimate Lemma \ref{lem:decoupling_estimate}.

\begin{lemma}[Forced Burgers equations]\label{lem:forced_burgers}
    Let $T > 0$, and $k \geq 1$. Let $f \in L^1((0,T); H^1(0,1))$ and $g \in L^2((0,T); H^1(0,1))$. There exist constants $C, \delta_1, \delta_2 > 0$, which may depend only on $k$, such that if $\Vert f \Vert_{L^1((0,T);H^1)} \leq \delta_1$ and $\Vert g \Vert_{L^2((0,T);H^1)} \leq \delta_2$, then the solution $y$ of 
    \begin{equation}\label{eq:forced_Burgers}
        \left\{
        \begin{array}{lll}
            \partial_t y - \partial_x^2 y + y \partial_x y + \partial_x (g y) = \partial_x f
            & \quad t \in (0, T), & \quad x \in (0,1), \\
            y(t,x) = 0
            & \quad t \in (0, T),  & \quad x \in \{0, 1\},\\
            y(0,x)=0
            & & \quad x \in (0, 1),
        \end{array}
        \right.
    \end{equation}
    satisfies
    \begin{equation}\label{eq:lem:forced_burgers_1}
        \left\Vert y \right\Vert_{L^2((0,T);H^1)} \leq C \Vert f \Vert_{L^1(H^1)},
    \end{equation}
    \begin{equation}\label{eq:lem:forced_burgers_2}
        \left\Vert y \right\Vert_{L^2((0,T);L^2)} + \left\Vert y(T) \right\Vert_{H^{-1}} \leq C \left\Vert f \right\Vert_{L^1((0,T);L^2)},
    \end{equation}
    and 
    \begin{equation}\label{eq:lem:forced_burgers_3}
        \begin{split}
            \left\vert \left\langle y(T), \varphi_k \right\rangle \right\vert
            \leq \ & C \left\vert \int_0^T e^{-\lambda_k (T-t)} \left\langle f(t), \varphi_k^\prime \right\rangle \dd t \right\vert \\
            & + C \left( \Vert g \Vert_{L^2((0,T);L^2)} + \left\Vert f \right\Vert_{L^1((0,T);L^2)} \right) \left\Vert f \right\Vert_{L^1((0,T);L^2)}.
        \end{split}
    \end{equation}
\end{lemma}

\begin{proof}
    In this proof, $\lesssim$ is used for absolute constants.
    
    \textbf{Step 1.} \textit{Proof of \eqref{eq:lem:forced_burgers_1}.}
    Using Lemma \ref{lem:burgers_strong_wellposedness}, one finds 
    \begin{equation}
        \left\Vert y \right\Vert_{L^2(H^1)} \lesssim \left\Vert \partial_x f - \partial_x (g y) \right\Vert_{L^1(L^2)}
        \lesssim \Vert f \Vert_{L^1(H^1)} + \Vert g y \Vert_{L^1(H^1)} .
    \end{equation}
    Using the Cauchy-Schwarz inequality and the fact that $H^1(0,1)$ is an algebra, one finds
    \begin{equation}
        \left\Vert y \right\Vert_{L^2(H^1)} \lesssim \Vert f \Vert_{L^1(H^1)} + \Vert g \Vert_{L^2(H^1)} \Vert y \Vert_{L^2(H^1)} \lesssim \Vert f \Vert_{L^1(H^1)} + \delta_2 \Vert y \Vert_{L^2(H^1)}.
    \end{equation}
    If $\delta_2$ is sufficiently small, this gives \eqref{eq:lem:forced_burgers_1}.
    
    \textbf{Step 2.} \textit{Proof of \eqref{eq:lem:forced_burgers_2}.} 
    Using Lemma \ref{lem:well-posedness_heat_derivative_L1L2}, one finds 
    \begin{equation}
        \left\Vert y \right\Vert_{L^2(L^2)} + \left\Vert y(T) \right\Vert_{H^{-1}} \lesssim \left\Vert f - gy - \frac{y^2}{2} \right\Vert_{L^1(L^2)}.
    \end{equation}
    Using the Cauchy-Schwarz inequality, the Sobolev embedding $H^1(0,1) \hookrightarrow L^\infty(0,1)$, and \eqref{eq:lem:forced_burgers_1}, one obtains
    \begin{equation}
        \begin{split}
            \left\Vert y \right\Vert_{L^2(L^2)} + \left\Vert y(T) \right\Vert_{H^{-1}} 
            \lesssim & \left\Vert f \right\Vert_{L^1(L^2)} + \left( \left\Vert g \right\Vert_{L^2(H^1)} + \left\Vert y \right\Vert_{L^2(H^1)} \right) \left\Vert y \right\Vert_{L^2(L^2)} \\
            \lesssim & \left\Vert f \right\Vert_{L^1(L^2)} + \left( \delta_2 + \delta_1 \right) \left\Vert y \right\Vert_{L^2(L^2)}.
        \end{split}
    \end{equation}
    If $\delta_1 + \delta_2$ is sufficiently small, this gives \eqref{eq:lem:forced_burgers_2}.
    
    \textbf{Step 3.} \textit{Proof of \eqref{eq:lem:forced_burgers_3}.} 
    Let $k \geq 1$. For the rest of the proof, $\lesssim$ denotes constants that may depend only on $k$. Using the Duhamel formula and integration by parts, one finds
    \begin{equation}
        \left\langle y(T), \varphi_k \right\rangle = - \int_0^T e^{-\lambda_k (T-t)} \left\langle f(t) - g(t) y(t) -\frac{y(t)^2}{2} , \varphi_k^\prime \right\rangle \dd t.
    \end{equation}
    Using the Cauchy-Schwarz inequality, one obtains 
    \begin{equation}
        \left\vert \left\langle y(T), \varphi_k \right\rangle \right\vert
        \lesssim  \left\vert \int_0^T e^{-\lambda_k (T-t)} \left\langle f(t), \varphi_k^\prime \right\rangle \dd t \right\vert + \Vert g \Vert_{L^2(L^2)} \left\Vert y \right\Vert_{L^2(L^2)} + \left\Vert y \right\Vert_{L^2(L^2)}^2.
    \end{equation}
    Together with \eqref{eq:lem:forced_burgers_2}, this gives \eqref{eq:lem:forced_burgers_3}.
\end{proof}

\subsection{Additional results for the Burgers equation}

\paragraph{A classical regularity result.}
  
\ifarxiv
    The following result is classical; we include it to verify that the continuity constant is independent of time.
\fi

\begin{lemma}\label{lem:burgers_strong_wellposedness}
    For $T > 0$, $f \in L^1((0, T), L^2(0,1))$ and $y_0 \in L^2(0,1)$, there exists a unique solution $y \in C^0([0, T]; L^2(0,1)) \cap L^2((0,T); H_0^1(0,1))$ of the Burgers equation 
    \begin{equation}\label{eq:burgers_strong_wellposedness}
        \left\{
        \begin{array}{cll}
            \partial_t y - \partial_x^2 y + y \partial_x y = f
            & \quad t \in (0, T), & \quad x \in (0,1), \\
            y(t,0) = y(t,1) = 0
            & & \quad t \in (0, T), \\
            y(0,x)=y_0
            & & \quad x \in (0, 1).
        \end{array}
        \right.
    \end{equation}
    In addition, there exists an absolute constant $C > 0$ such that
    \begin{equation}\label{eq:lem:Burgers_strong_wellposedness}
        \Vert y \Vert_{L^2((0,T);H^1)} \leq C \left( \Vert y_0 \Vert_{L^2} + \Vert f \Vert_{L^1(L^2)} \right).
    \end{equation}
\end{lemma}

\ifarxiv
\begin{proof}
    We only provide a short proof of the fact that $C$ is independent of $T$. Multiplying the equation by $y$ and integrating in space, one finds
    \begin{equation}\label{eq:proof:lem:Burgers_strong_wellposedness_1}
        \frac{1}{2} \frac{\mathrm{d}}{\mathrm{d} t} \Vert y(t) \Vert_{L^2(0,1)}^2 + \Vert \partial_x y(t) \Vert_{L^2(0,1)}^2 = \left\langle f(t), y(t) \right\rangle_{L^2(0,1)}.
    \end{equation}
    It implies
    \begin{equation}
        \Vert y(t) \Vert_{L^2(0,1)} \frac{\mathrm{d}}{\mathrm{d} t} \Vert y(t) \Vert_{L^2(0,1)} 
        = \frac{1}{2} \frac{\mathrm{d}}{\mathrm{d} t} \Vert y(t) \Vert_{L^2(0,1)}^2
        \leq  \Vert f(t) \Vert_{L^2(0,1)} \Vert y(t) \Vert_{L^2(0,1)}.
    \end{equation}
    This gives $\frac{\mathrm{d}}{\mathrm{d} t} \Vert y(t) \Vert_{L^2(0,1)} \leq \Vert f(t) \Vert_{L^2(0,1)}$, implying
    \begin{equation}\label{eq:proof:lem:Burgers_strong_wellposedness_2}
        \Vert y \Vert_{L^\infty(L^2)} \leq \Vert f \Vert_{L^1(L^2)} + \Vert y_0 \Vert_{L^2} .
    \end{equation}
    
    Now, coming back to \eqref{eq:proof:lem:Burgers_strong_wellposedness_1} and integrating in time, one obtains
    \begin{equation}
        \Vert \partial_x y \Vert_{L^2((0,T);L^2)}^2 = \left\langle f, y \right\rangle_{L^2((0,T);L^2)} + \frac{\Vert y_0 \Vert_{L^2(0,1)}^2}{2} - \frac{\Vert y(T) \Vert_{L^2(0,1)}^2}{2} .
    \end{equation}
    Using \eqref{eq:proof:lem:Burgers_strong_wellposedness_2}, one finds
    \begin{equation}
        \Vert \partial_x y \Vert_{L^2((0,T);L^2)}^2 \leq \Vert f \Vert_{L^1(L^2)} \left( \Vert f \Vert_{L^1(L^2)} + \Vert y_0 \Vert_{L^2} \right) + \frac{\Vert y_0 \Vert_{L^2(0,1)}^2}{2} .
    \end{equation}  
    Together with the Poincaré inequality, this gives
    \begin{equation}
        \Vert y \Vert_{L^2((0,T);H^1)}^2 \leq C \left( \Vert f \Vert_{L^1(L^2)}^2  + \Vert y_0 \Vert_{L^2}^2 \right) ,
    \end{equation}    
    for some absolute constant $C > 0$.
\end{proof}
\fi

\paragraph{Free evolution of a Burgers equation with small initial data.}
In the proof of the controllability obstructions, we need to estimate $y(T;y_0,0)$ for small initial data $y_0\in L^2(0,1)$. For finite-time obstructions, we use \eqref{eq:lem:free_evolution_Burgers_1}, and for small-time obstructions, we use \eqref{eq:lem:free_evolution_Burgers_2}.

\begin{lemma}\label{lem:free_evolution_Burgers}
    There exists $C > 0$ such that for all $T > 0$ and $y_0 \in L^2(0, 1)$, one has 
    \begin{equation}\label{eq:lem:free_evolution_Burgers_1}
        \left\Vert y(T; y_0, 0) - S(T) y_0 \right\Vert_{H^{-1}} \leq C T^{\frac{3}{4}} \Vert y_0 \Vert_{L^2}^2,
    \end{equation}
    and
    \begin{equation}\label{eq:lem:free_evolution_Burgers_2}
        \left\Vert y(T; y_0, 0) - y_0 \right\Vert_{H^{-1}} \leq C T^{\frac{3}{4}} \Vert y_0 \Vert_{L^2}^2 + C T^{\frac{1}{2}} \Vert y_0 \Vert_{L^2},
    \end{equation}
    where $S$ is the heat semigroup.
\end{lemma}

\begin{proof}
    In this proof, $\lesssim$ is used for absolute constants. 
    First, by definition, $z(t) := y(t; y_0, 0) - S(t) y_0$ is the solution of 
    \begin{equation}
        \left\{
        \begin{array}{lll}
            \partial_t z - \partial_x^2 z = - \frac{1}{2} \partial_x (y^2) 
            & \quad t \in (0, T), & \quad x \in (0,1), \\
            z(t,x) = 0
            & \quad t \in (0, T),  & \quad x \in \{0, 1\},\\
            z(0,x)=0
            & & \quad x \in (0, 1).
        \end{array}
        \right.
    \end{equation}
    Hence, Lemma \ref{lem:well-posedness_heat_derivative_L1L2} and Lemma \ref{lem:easy_estimate_L1L2} imply
    \begin{equation}
        \left\Vert z \right\Vert_{L^\infty((0, T);H^{-1})} \lesssim \Vert y^2 \Vert_{L^1((0, T); L^2)} \lesssim \Vert y \Vert_{L^2((0, T); H^1)}^{\frac{1}{2}} \Vert y \Vert_{L^2((0, T); L^2)}^{\frac{3}{2}}.
    \end{equation}
    Together with Lemma \ref{lem:burgers_weak_wellposedness}, this gives
    \begin{equation}
        \left\Vert z \right\Vert_{L^\infty((0, T);H^{-1})} \lesssim T^{\frac{3}{4}} \Vert y \Vert_{L^2((0, T); H^1)}^{\frac{1}{2}} \Vert y \Vert_{L^\infty((0, T); L^2)}^{\frac{3}{2}} \lesssim T^{\frac{3}{4}} \Vert y_0 \Vert_{L^2}^2,
    \end{equation}
    which is \eqref{eq:lem:free_evolution_Burgers_1}. 
    
    Second, to prove \eqref{eq:lem:free_evolution_Burgers_2}, it suffices to show
    \begin{equation}
        \left\Vert S(T) y_0 - y_0 \right\Vert_{H^{-1}} \lesssim T^{\frac{1}{2}} \Vert y_0 \Vert_{L^2}. 
    \end{equation}
    By definition, one has
    \begin{equation}
        \begin{split}
            \left\Vert S(T) y_0 - y_0 \right\Vert_{H^{-1}}^2 
            & = \sum_{j \geq 1} \frac{1}{\lambda_j} \left\vert e^{-\lambda_j T} - 1 \right\vert^2 \left\langle y_0, \varphi_j \right\rangle^2 
            \leq \sum_{j \geq 1} \frac{1}{\lambda_j} \left\vert e^{-\lambda_j T} - 1 \right\vert \left\langle y_0, \varphi_j \right\rangle^2 \\
            & \leq T \sum_{j \geq 1} \left\langle y_0, \varphi_j \right\rangle^2 
            = T \Vert y_0 \Vert_{L^2}^2,
        \end{split}
    \end{equation}
    where we have used the elementary estimate $0 \leq 1 - e^{-x} \leq x$, for $x \geq 0$. 
\end{proof}

\section*{Acknowledgments}

The author acknowledges support from the Fondation Simone et Cino Del Duca - Institut de France, the Fondation Université de Rennes and from grant ANR-11-LABX-0020 (Labex Lebesgue). The author is deeply grateful to Karine Beauchard and Frédéric Marbach for introducing him to the subject and for their constant support throughout this work.

\bibliographystyle{plain}
\bibliography{biblio}

\noindent
\textsc{Perrin Thomas:} \texttt{thomas.perrin@ens-rennes.fr}

\end{document}